\documentclass{article}
\usepackage[a4paper,top=4cm,bottom=4cm,left=2.5cm,right=2.5cm,bindingoffset=5mm]{geometry}
\usepackage[dvipsnames]{xcolor}

\usepackage[hyphens]{url}\usepackage[breaklinks=true,colorlinks,linkcolor={blue},citecolor={red},urlcolor={purple},]{hyperref}

\usepackage{caption} 
\usepackage[subrefformat=parens,labelformat=parens]{subcaption}

\usepackage{enumitem}
\usepackage{amsmath}

\usepackage[textwidth=2cm]{todonotes}
\usepackage{overpic}
\tikzset{>=latex}
\usetikzlibrary{spy}
\usepackage{booktabs,multirow}
\tikzset{>=latex}
\usetikzlibrary{arrows}
\usepackage{circuitikz}
\usetikzlibrary{shapes.geometric}

\usepackage{mathtools}

\usepackage{amssymb,bm,amsmath,amsthm}
\usepackage{graphics} 
  
\graphicspath{{../grafici/}{./}}

\newcommand{\norm}[1]{\big\lVert#1\big\rVert}
\def\de{\partial}
\def\disp{\displaystyle}
\def\R{\mathbb{R}}
\def\Z{\mathbb{Z}}
\def\N{\mathbb{N}}
\def\vv{\mathcal{V}}
\def\uu{\mathcal{U}}
\def\ff{\mathcal{F}}
\def\qq{\mathcal{Q}}
\def\ncar{N}
\def\umax{u^{\mathrm{max}}}

\def\vmax{V^{\mathrm{max}}}

\def\rhomax{\rho^{\mathrm{max}}}

\def\rhof{\rho_{f}}

\def\wl{w_{L}}
\def\wr{w_{R}}

\def\nn{\widetilde N}
\def\tot{\mathrm{tot}}
\def\sens{\mathrm{sens}}
\def\cuno{C^{1}}

\newcommand{\bp}{\textbf{p}}
\def\dx{\Delta x}
\def\dy{\Delta y}
\def\dt{\Delta t}
\def\nx{N_{x}}
\def\nt{N_{t}}

\def\kmh{\mathrm{km/h}}
\def\km{\mathrm{km}}
\def\meter{\mathrm{m}}
\def\vehkm{\mathrm{veh/km}}
\def\second{\mathrm{s}}
\def\myhour{\mathrm{h}}

\def\mysecond{\mathrm{s}}
\def\mymeter{\mathrm{m}}

\def\mymin{\mathrm{min}}
\def\per{/}
\def\km{\mathrm{km}}
\def\myhour{\mathrm{h}}
\def\mygram{\mathrm{g}}

\def\nox{\mathrm{NO_{x}}}

\numberwithin{equation}{section}
\newtheorem{remark}{Remark}
\newtheorem{proposition}{Proposition}
\numberwithin{remark}{section}

\makeatletter
\newcommand{\mylabel}[2]{#2\def\@currentlabel{#2}\label{#1}}
\makeatother

\title{\Large\textbf{Estimate of traffic emissions through multiscale \\second order models 
with heterogeneous data}}

\author{\normalsize{Caterina Balzotti}\thanks{SISSA - Scuola Internazionale Superiore di Studi Avanzati, Mathematics Area, mathLab, Via Bonomea, 265, 34136, Trieste, Italy (\href{mailto:cbalzott@sissa.it}{cbalzott@sissa.it})
}
\and {\setcounter{footnote}{3}\normalsize{Maya Briani}\thanks{Istituto per le Applicazioni del Calcolo, 
Consiglio Nazionale delle Ricerche, Via dei Taurini 19, 00185, Rome, Italy (\href{mailto:m.briani@iac.cnr.it}{m.briani@iac.cnr.it})}
}
}
\date{\vspace{-0.5cm}}

\begin{document}

\maketitle

\begin{abstract}
In this paper we propose a multiscale traffic model, based on the family of Generic Second Order Models, which integrates multiple trajectory data into the velocity function. This combination of a second order macroscopic model with microscopic information allows us to reproduce significant variations in speed and acceleration that strongly influence traffic emissions. We obtain accurate approximations even with a few trajectory data. The proposed approach is therefore a computationally efficient and highly accurate tool for calculating macroscopic traffic quantities and estimating emissions. 
\end{abstract}

\begin{description}
\item[\textbf{Keywords.}] Second order traffic models; heterogeneous data; emissions; road networks. 
\item[\textbf{Mathematics Subject Classification.}]  35L65, 35F25, 90B20, 62P12.
\end{description}

\section{Introduction}\label{sec:intro}
In this paper, we focus on the development of models specifically designed to take advantage of the availability of heterogeneous data. By heterogeneous data we mean not only data coming from different sources, but especially data coming from different scales of observation and different modes of monitoring. We refer in particular to Lagrangian data, which provide information on the trajectories followed by vehicles, and to Eulerian data, which measure the transit of cars from fixed locations. In the case of vehicles, the Lagrangian data are typically GPS data (i.e.\ trajectory data, from which the instantaneous speed is derived), while the Eulerian data come from fixed sensors placed along the road, capable of counting cars and measuring their speed. 

Alongside this analysis we consider the problem of estimating emissions from vehicular traffic on complex networks.
The continuous traffic growth is in fact associated with negative environmental effects, which are related to both air quality and climate change. In order to assess the impact of traffic emissions on the environment and human health, an accurate estimate of their rates and location is required.

\subsection{Related work}
In this article we mainly follow the approach of \cite{colombo2016M3AS}, where the authors propose a new traffic model that integrates position and velocity information of a single tracked vehicle into the velocity function of the Lighthill-Whitham-Richards (LWR) model \cite{LighthillWhitham1955, Richards1956}. LWR is a first order (i.e.\ a single equation) model that describes traffic dynamics in a road through the density of vehicles and their average speed. We extend the ideas of \cite{colombo2016M3AS} to the case of several vehicle trajectory data and to the family of Generic Second Order Models (GSOM) introduced in \cite{lebacque2007TTT}. GSOM encompasses the majority of second order (i.e.\ two equations) traffic models that, unlike first order ones, are able to reproduce bounded traffic accelerations \cite{lebacque2002TT}. Our choice leads to a multiscale type model that exploits the best features of the macroscopic and microscopic approaches. Multiscale models has been already considered in several papers. We refer, for instance, to \cite{colombo2015M3AS, garavello2017NoDEA} where the macroscopic LWR model is merged with the classical microscopic \emph{follow the leader} model. In \cite{cristiani2019DCDSB} the authors propose a multiscale approach obtained by coupling a first order macroscopic model with a second order microscopic one that is used only under specific traffic conditions.
The interested reader can find other examples of multiscale models in \cite{bourrel2003TRR, garavello2013NHM, herty2012KRM, lattanzio2010M3AS, leclercq2007TRB}. Hereafter we cite some works closer to our goals.
In \cite{chalons2017IFB}, the LWR model is combined with an ordinary differential equation representing the trajectory of a slow vehicle acting as a moving bottleneck. This approach can be considered as a way to include real trajectory data in a macroscopic traffic model and is therefore comparable with our scopes.
We also refer to \cite{CanepaClaudel2017} and the references therein, to mention a robust method of traffic estimation involving mixed fixed and mobile sensor data using the Hamilton-Jacobi equations. 

Once the traffic state variables have been estimated, they can be used as  input for the so-called emission models, which evaluate the mass of pollutants emitted. In this respect, in \cite{Jamshidnejad2017TRC,zegeye2013elsevier} the authors provide general frameworks to integrate macroscopic traffic flow models and microscopic emission models. In \cite{alvarez2017JCAM,alvarez2018MCRF} the traffic modeling relies on the LWR model and a reaction-diffusion model describes the spread of carbon monoxide in the air with a source term associated with traffic dynamics. In the recent work \cite{garcia2020AX} the authors analyze a reaction-diffusion model, based on LWR traffic dynamics, to control nitrogen oxides emissions. In \cite{bayen2014} the authors suggest a new methodology to estimate in real-time the emission rates of pollutants and describe their diffusion in air. Furthermore, in \cite{balzotti2021DCDS} the authors propose a computational tool to estimate pollutant emissions due to vehicular traffic using second order traffic models. This approach approximates emissions well when a large amount of data is available to feed the traffic model. The present work extends this method by including microscopic data in the second order macroscopic traffic model, and provides a good estimate of pollutant emissions even when few data is available.

\subsection{Background and motivations}
In order to exploit as much information as possible, we consider traffic models that take into account different and heterogeneous traffic data available along a road. 
A single source of data is generally not sufficient to estimate and forecast traffic volumes on a road.
In Figure \ref{fig:esDati} on the left we provide an example of the partial information coming from GPS data; speed is observed for a limited number of vehicles and only at specific points in time and space. At the same time, in Figure \ref{fig:esDati} on the right we show an example of flux data coming from a fixed sensor; 
this type of data is not enough to correctly calculate traffic quantities such as densities, as it only provides average speed values or noisy, time-sampled flow information.
\begin{figure}[h!]
\centering
\includegraphics[width=0.3\columnwidth]{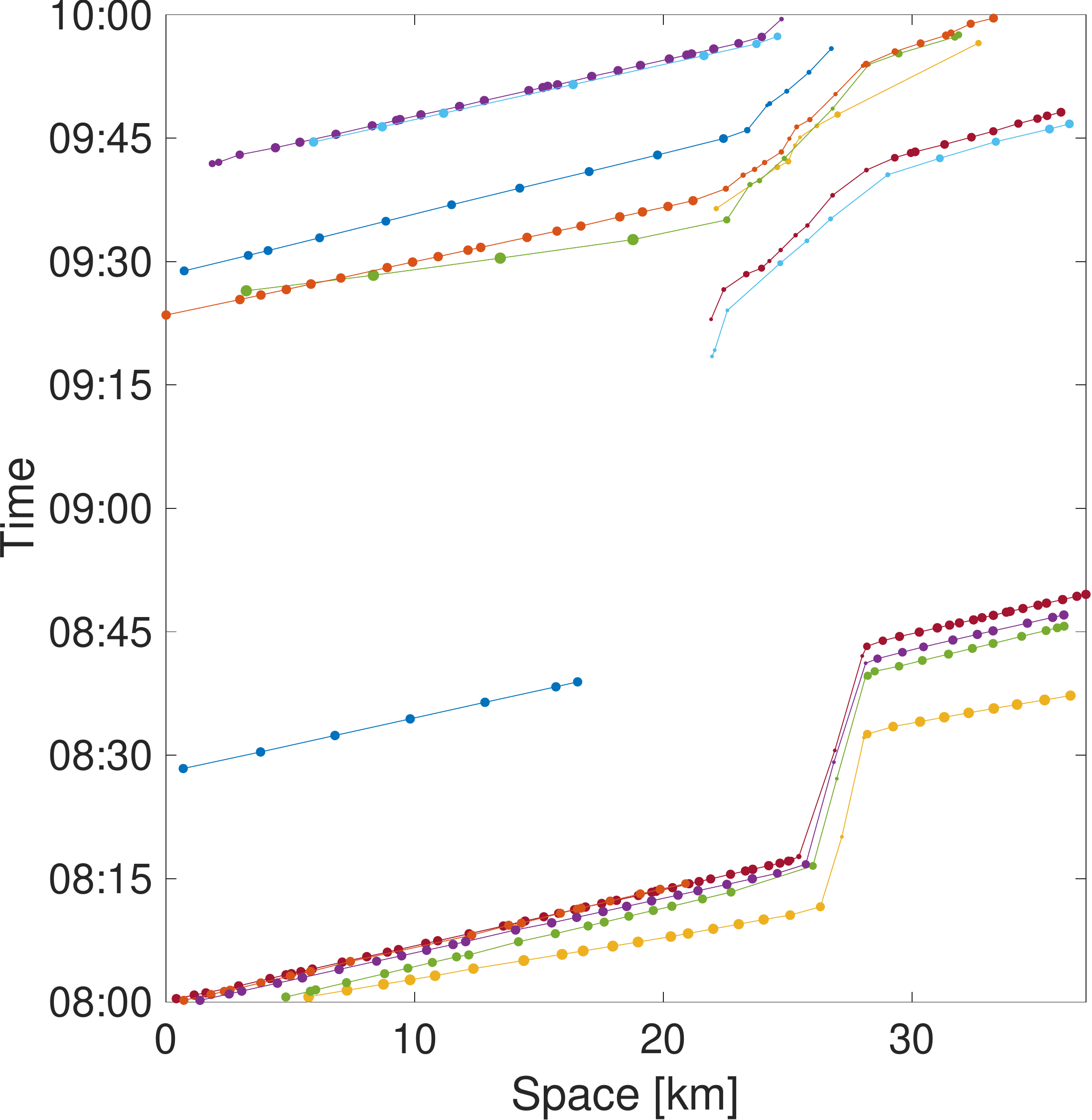}\quad
\includegraphics[width=0.305\columnwidth]{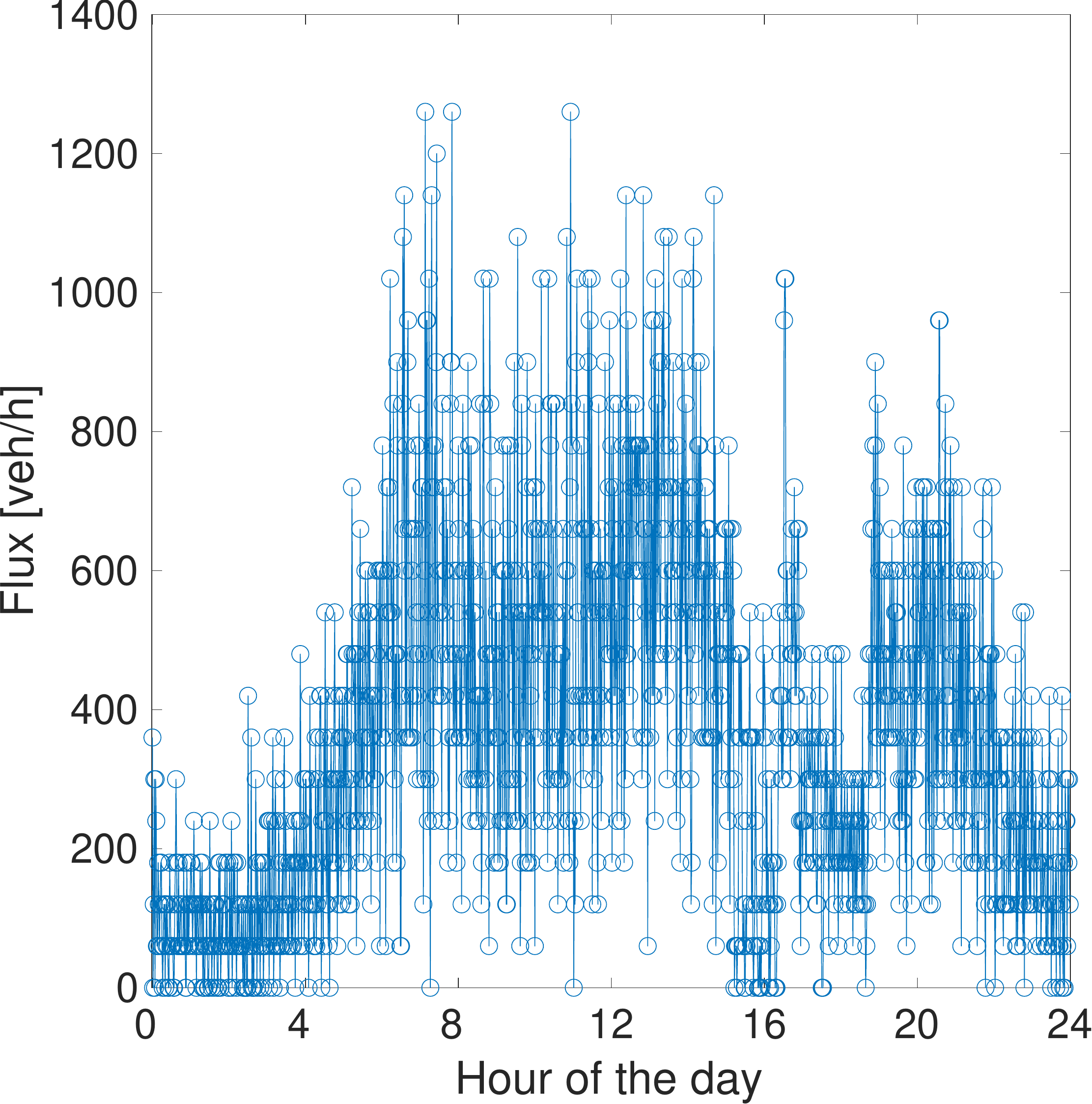}
\caption{A sample of a GPS trajectory dataset (left) and of flux data coming from a fixed sensor (right).
The data was provided by Autovie Venete S.p.A and are not publicly available.}
\label{fig:esDati}
\end{figure}

To deal with these two types of data, we have followed the approach proposed in \cite{colombo2016M3AS}, which is an effective and efficient way of coupling macroscopic and microscopic variables to estimate traffic volumes. 
This approach allows us to perform data fusion directly at the model level: Eulerian flow data measured by stationary sensors is used as boundary condition for the differential equations describing the traffic dynamics,
while the Lagrangian data from the GPS sensor is used to correct macroscopic quantities in real time. 
More precisely, let $p=p(t)$ be the position of a single vehicle along the road at time $t$. Starting from the conservation law of vehicles density $\rho=\rho(x,t)$,
\begin{equation}\label{eq:model1ord}
\de_{t}\rho+\de_{x}(\rho \uu) = 0,
\end{equation}
the measured trajectory $p=p(t)$ is directly encoded in the time- and space-dependent speed law $\uu$ as 
\begin{equation*}
	\uu(x,t,\rho;p) = 
	\disp\chi(x-p(t))\frac{2 \dot{p}(t)u(\rho)}{\dot{p}(t)+u(\rho)}+\big(1-\chi(x-p(t))\big)u(\rho),
\end{equation*}
if $\dot{p}(t)+u(\rho)\neq0$, otherwise $\uu=0$.
Here, $u=u(\rho)$ is a given speed function depending only on the density and $\chi(\xi)$ is a smooth, non-negative function such that $\chi(\xi)=1$ for $|\xi|\leq \ell$ and $\chi(\xi)=0$ when $|\xi|\geq L$, for two fixed constant $\ell,L$ with $\ell<L$. Thus, when the point $(x,t)$ is sufficiently far from the position of the vehicle $p=p(t)$ ($\chi=0$), then $\uu$ is defined by the velocity function $u(\rho)$ as in the LWR model, otherwise it is given by the harmonic mean between $u$ and the velocity $\dot p(t)$ of the vehicle.


We extend this idea to the case of more trajectory data available. We propose two ways to incorporate the Lagrangian data of $N$ vehicles. In the first one, the velocity function $\uu$ takes into account only the speed $\dot p_k(t)$ of the $\kappa$-vehicle closest to the point $x$ at time $t$; in the second one, it depends on the mean speed of the set of vehicles closest to point $x$. The difference between the two approaches only shows up when several vehicles with different speeds are close to each other. In both cases, the resulting flow function $\ff=\rho\uu$ is strictly concave with respect to $\rho$. We do not investigate the model from an analytical point of view.  We need $\uu$ to inherit reasonable properties of $u$ and $\dot p$, such as positivity and boundedness, and we still refer to \cite{colombo2016M3AS} for analytical aspects.

In absence of real trajectory data, the macroscopic traffic model introduced above can be coupled with a microscopic one.
Lagrangian data, in fact, can be generated from a microscopic model and then included in the velocity function $\uu$ as previously described. One therefore obtains a multiscale type model capable of describing some traffic phenomena generally captured only by a detailed, microscopic description of the dynamics.

The proposed approaches can be easily applied to other traffic models. We can thus incorporate the information from $N$ Lagrangian data into the family of Generic Second Order Model (GSOM) \cite{lebacque2007TTT}
described by 
 \begin{equation*}
	\begin{split}
		&\begin{cases}
			\de_t\rho+\de_x(\rho \vv) = 0\\
			\de_t w+\vv\de_x w = 0,
		\end{cases}
	\end{split}
	\label{eq:GSOM}
\end{equation*}
where $\rho(x,t)$ is the density of vehicles, $w(x,t)$ is a vehicle/driver property or invariant, which is conserved along trajectories, and $\vv$ is the velocity field. 
By taking into account the Lagrangian information,  $\vv=\vv(x,t,\rho,w;\bp)$ is defined as
\begin{align*}
	\vv(x,t,\rho,w;\bp) = 
	\disp\chi(x-p_{\kappa}(t))\frac{2 \dot{p}_{\kappa}(t)v(\rho,w)}{\dot{p}_{\kappa}(t)+v(\rho,w)}+\big(1-\chi(x-p_{\kappa}(t))\big)v(\rho,w)
\end{align*}
if $\dot{p}_{\kappa}(t)+v(\rho,w)\neq0$, otherwise $\vv=0$. Here $\bp=(p_1(t),\ldots,p_N(t))$ is the vector of positions of the $N$ vehicles and $\kappa=\kappa(x,t)$ is the index of the closest vehicle to point $x$ at time $t$. The function $v=v(\rho,w)$ is a given analytical function which describes the macroscopic velocity field in the \textit{native} (without tracking vehicles) second order model.

\subsection{Emissions estimate} The use of a second order model leads to good approximations of the acceleration of vehicles and consequently to the estimate of traffic emissions at ground level, that is one of the main goal of our work. Most emission models are based on both vehicle speed and acceleration, see for instance \cite{Barth2000,Smit2010} and references therein. Here we explore the use of two types of formula which compute the emissions associated with the motion of vehicles. The models we consider have been introduced in \cite{panis2006elsevier} and \cite{ahn1999TRB}, respectively. 
In both formulas the contribution of the Lagrangian data is incorporated in the terms of speed and acceleration evaluated by the traffic model. Specifically, by computing the acceleration function $a=a(x,t)$ as the time material derivative of the new function $\vv$ introduced above it directly depends on the acceleration $\ddot{p}_{\kappa}(t)$ of the nearest vehicle. 

We compare the two emission models and we show how the integration of real data affects their results. We observe that the integration of trajectory data into the macroscopic traffic model increases the order of accuracy of the emission estimate, even when there are few data available. In particular, the formula proposed in \cite{panis2006elsevier} gives better results.

We conclude our study with a \textit{real-life application} using trajectory and fixed sensors
data provided by Autovie Venete S.p.A. on the Italian A4 (Trieste-Venice) highway.
With this test 
we link heterogeneous traffic source data to emission estimates along a road network, and at the same time we provide an approximation of the source term that feeds air pollutant diffusion and chemical reaction models.  
The numerical results show how the GPS data influences the solution of the traffic model and gives good reproductions of the emission peaks at the macroscopic scale.

\subsection{Main goal}
In summary, we propose a second order traffic model that returns macroscopic traffic quantities by incorporating microscopic information. Microscopic trajectories are included in the definition of the velocity field in order to perturb the velocity and acceleration values at the macroscopic level. 
This methodology combines the computational efficiency of a macroscopic model with the accuracy of a microscopic representation. 
This makes it particularly suitable as an input for estimating the mass of emitted pollutants when an aggregate description is required.
With a few Lagrangian trajectories, it is in fact possible to reproduce significant emission variations at the macroscopic scale. 
The procedure is very flexible and can be used with real measurements or with vehicle trajectories generated by Lagrangian models.

\subsection{Paper organization}
In Section \ref{sec:Ndata} we propose two possible extensions of the first order LWR model to integrate Lagrangian data from $N$ vehicles. A numerical test is proposed to show the differences between the two approaches. In Section \ref{sec:modello2}, we apply the ideas given in the previous section to the family of GSOM. Then, we compute the acceleration as the material derivative of the velocity function, making explicit its dependence on the acceleration of the single vehicles. In Section \ref{sec:emiModels} we describe two models to estimate the emissions produced by vehicular traffic.
In Section \ref{sec:numtraffic} and \ref{sec:emissioni}, we propose numerical tests to show how the integration of Lagrangian data into the GSOM impacts on the traffic dynamic and on the estimate of emissions, respectively. We conclude with Section \ref{sec:autovieData}, which describes the use of GPS and fixed sensors data provided by Autovie Venete S.p.A. on a portion of the Italian highway network. Finally, in Appendix \ref{sec:proof} we report two technical proofs.

\section{A macroscopic first order model embedding Lagrangian data}\label{sec:Ndata}

Assume to know the trajectory of $N$ vehicles, and let $\bp =(p_1(t),\ldots,p_N(t))$ be the vector of their positions at time $t$. Following \cite{colombo2016M3AS}, we assume $p_i=p_i(t)$, to be continuous and smooth functions such that $\dot{p}_i\geq 0$ for $i=1,\cdots,N$ and a.e.\ $t\in\R^+$. We propose two approaches to incorporate information from the vehicles into the velocity function of the conservation law \eqref{eq:model1ord}.

\begin{description}[before={\renewcommand\makelabel[1]{##1}}]
\item[\mylabel{vicina}{\textnormal{(CV)}}] \textbf{Closest Vehicle}.
We define $\uu$ in \eqref{eq:model1ord} as
\begin{align}\label{eq:utrue}
	\uu(x,t,\rho;\bp) & = \begin{cases}\disp
	\chi(x-p_{\kappa}(t))\frac{2 \dot{p}_{\kappa}(t)u(\rho)}{\dot{p}_{\kappa}(t)+u(\rho)}+\big(1-\chi(x-p_{\kappa}(t))\big)u(\rho) &\text{if $(\dot{p}_{\kappa},u)\neq(0,0)$}\\
		0 &\text{otherwise,}
	\end{cases}
\end{align}
where $p_\kappa(t)$ is the position of the $\kappa$-vehicle closest to point $x$. With a slight abuse of notation, we have 
\begin{equation}\label{eq:kvicino}
	\kappa = \kappa(x,t) = \arg\min\{|x-p_{i}(t)|, i=1,\dots N\}. 
\end{equation}

\item[\mylabel{media}{\textnormal{(ACVs)}}] \textbf{Average on the Closest Vehicles}.
For each vehicle position $p_i(t)$, $i=1,\ldots,N$, we set 
\begin{align*}
	\uu_{i}(x,t,\rho;\bp)  &= \begin{cases}\disp
	\chi(x-p_{i}(t))\frac{2 \dot{p}_{i}(t)u(\rho)}{\dot{p}_{i}(t)+u(\rho)}+\big(1-\chi(x-p_{i}(t))\big)u(\rho) &\text{if $(\dot{p}_{i},u)\neq(0,0)$}\\
		0 &\text{otherwise.}
	\end{cases}
\end{align*}
We introduce the function $\phi(x,t)$ that counts the number of vehicles whose position at time $t$ is close to point $x$, i.e.\ such that $\chi(x-p_i(t))>0$,
\begin{equation}\label{eq:phi}
	\phi(x,t) = \#\{i\in\{1,\dots,\ncar\} \,:\,\chi(x-p_{i}(t))>0 \}.
\end{equation}
We then define the function $\uu$ in \eqref{eq:model1ord} as the average value of the speeds of the closest vehicles, that is
\begin{align}\label{eq:umedio}
	\uu(x,t,\rho;\bp) & = \begin{cases}
	\disp\frac{1}{\phi(x,t)}\disp\sum_{i=1}^{\phi(x,t)}\uu_{i}(x,t,\rho;\bp) & \quad\text{if } \phi(x,t)>0
	\\
	u(\rho) & \quad\text{if } \phi(x,t)=0.
	\end{cases}
\end{align}

\end{description}

\noindent The two approaches are different only when two or more trajectories are very close to each other, otherwise they coincide.

We use the same assumptions made in \cite{colombo2016M3AS},  namely: $\rho\in[0,\rhomax]$, $u=u(\rho)$ is a regular and non increasing function with $u(\rhomax)=0$, $f(\rho)=\rho u(\rho)$ is strictly concave in $\rho$ and, for each $z>0$ the function ${\rho z u}/{(z+u)}$ is strictly concave in $\rho$; the flux function $\ff(x,t,\rho) = \rho\uu(x,t,\rho;\bp)$, with $\uu$ in \eqref{eq:utrue} or \eqref{eq:umedio}, is strictly concave with respect to $\rho$. Thus, there exists a unique density value $\sigma(x,t)$ where the flux reaches its maximum $\ff(x,t,\sigma) = \ff^{\max}(x,t)$ and we can define the \textit{sending} $S=S(x,t,\rho)$ and \textit{receiving} $R=R(x,t,\rho)$ functions in the standard way \cite{lebacque1996},
\begin{equation}\label{eq:SRfunct}
	S = \begin{cases}
		\ff(x,t,\rho) &\quad\text{if $\rho\leq\sigma(x,t)$}\\
		\ff^{\max}(x,t) &\quad\text{if $\rho>\sigma(x,t)$}
	\end{cases}
	\qquad
	R = \begin{cases}
		\ff^{\max}(x,t) &\quad\text{if $\rho\leq\sigma(x,t)$}\\
		\ff(x,t,\rho) &\quad\text{if $\rho>\sigma(x,t)$.}
	\end{cases}
\end{equation}
The critical density $\sigma(x,t)$ is defined by $\de_{\rho}\ff(x,t,\sigma)=0$. 
In the \ref*{vicina} case, for $(x,t)$ such that $\chi(x-p(t))\neq 0$ or $\chi(x-p(t))\neq 1$, it is implicitly defined by the non-linear relation 
\begin{equation}\label{eq:eqSigma}
\begin{split}
	2\dot{p}(t)\,\chi(x-p(t))\Big(\dot{p}(t)u(\sigma)+u^{2}(\sigma)+\rho\dot{p}(t)u_{\rho}(\sigma)\Big)&\\
	+\Big(1-\chi(x-p(t))\Big)\Big(u(\sigma)+\rho u_{\rho}(\sigma)\Big)\Big(\dot{p}(t)+u(\sigma)\Big)^{2}&=0.
\end{split}
\end{equation}
Also in the \ref*{media} case the computation of $\sigma$ requires the solution of a non-linear problem,
\begin{equation}\label{eq:eqSigma2}
\begin{split}
	\sum_{i=1}^{\phi(x,t)}\Big[2\dot{p_{i}}(t)\,\chi(x-p_{i}(t))\Big(\dot{p_{i}}(t)u(\sigma)+u^{2}(\sigma)+\rho\dot{p_{i}}(t)u_{\rho}(\sigma)\Big)&\\
	+\Big(1-\chi(x-p_{i}(t))\Big)\Big(u(\sigma)+\rho u_{\rho}(\sigma)\Big)\Big(\dot{p_{i}}(t)+u(\sigma)\Big)^{2}\Big]&=0.
\end{split}
\end{equation}
In both cases, for each point $(x,t)$, the calculation of $\sigma=\sigma(x,t)$ must be done numerically.
\begin{remark}\label{remark:sigma}
Numerical solvers for nonlinear equations such as \eqref{eq:eqSigma} or \eqref{eq:eqSigma2} have a high computational cost. Since the critical density $\sigma$ is the maximum point of the flow function $\ff$, one can reduce the complexity of the computation by sampling the values of $\ff=\ff(\cdot,\cdot,\rho)$ into a vector and applying a search for the maximum value of the vector that returns the corresponding density index.
\end{remark}

To highlight the differences between the two proposed models \ref*{vicina} and \ref*{media}, in Figure \ref{fig:chi} we plot different curves $\ff(x,t,\rho)$ with one or more trajectory data. We set 
$u(\rho)=\umax(\rhomax-\rho)/\rhomax$ and 
\begin{equation}\label{eq:chi}
	\chi(\xi) = \begin{dcases}
		0 &\quad\text{if $|\xi|>L$}\\
		(\xi+L)/(L-\ell) &\quad\text{if $-L\leq \xi<-\ell$}\\
		1 &\quad\text{if $-\ell\leq \xi\leq\ell$}\\
		(\xi-L)/(\ell-L) &\quad\text{if $\ell< \xi\leq L$}.
	\end{dcases}
\end{equation}
In Figure \ref{fig:chi} on the left we fix the time $t$ and we plot the curves $\ff(x,t,\rho)$ in the case of a single trajectory data. In this case the two approaches \ref*{vicina} and \ref*{media} coincide, and  we observe that the critical density $\sigma(x,t)$ increases with $\chi(x-p(t))$ while $\ff^{\max}(x,t)$ decreases.
In Figure \ref{fig:chi} on the right, we compare the flux function $\ff(x,t,\rho)$ given by the \ref*{vicina} and \ref*{media} approaches assuming to have three vehicles travelling with different speed and close enough to influence the velocity function \eqref{eq:umedio}. The solid line represents the flux function $\ff(x,t,\rho)$ with the \ref*{media} approach on the cell $x$ where all the three vehicles contribute to the velocity computation, i.e.\ $\phi(x,t)=3$. The dotted lines, instead, represent the flux function with the \ref*{vicina} model in the cells $x$ where $\chi(x-p_{\kappa}(t))=1$, $\kappa\in\{1,2,3\}$. We observe that $\sigma(x,t)$ increases with the decrease of vehicle velocity. Indeed, for small velocities the vehicles need more time to fill the road, thus we obtain higher values of the critical density.
\begin{figure}[h!]
\centering
\includegraphics[width=0.3\columnwidth]{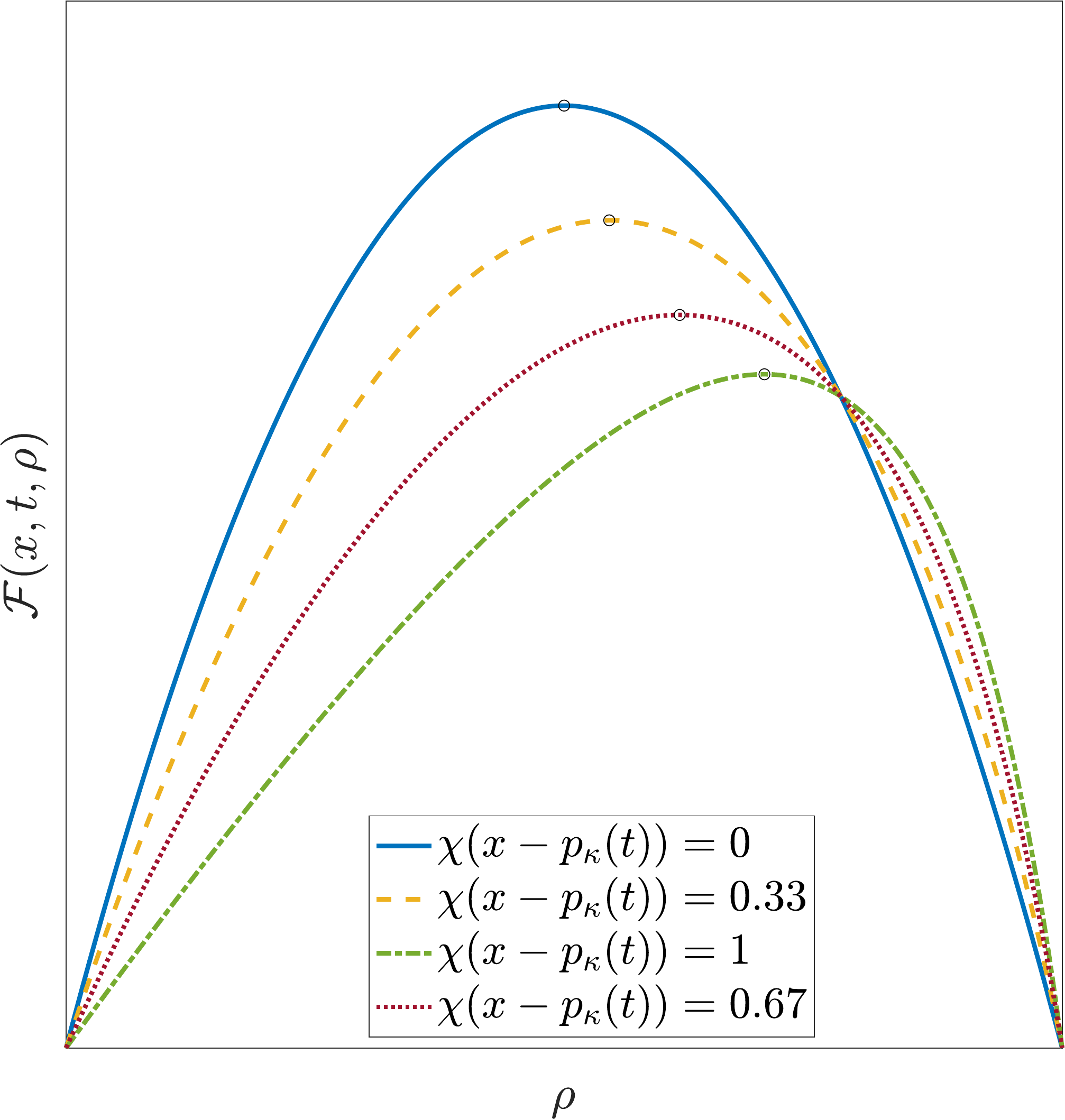}\qquad
\includegraphics[width=0.3\columnwidth]{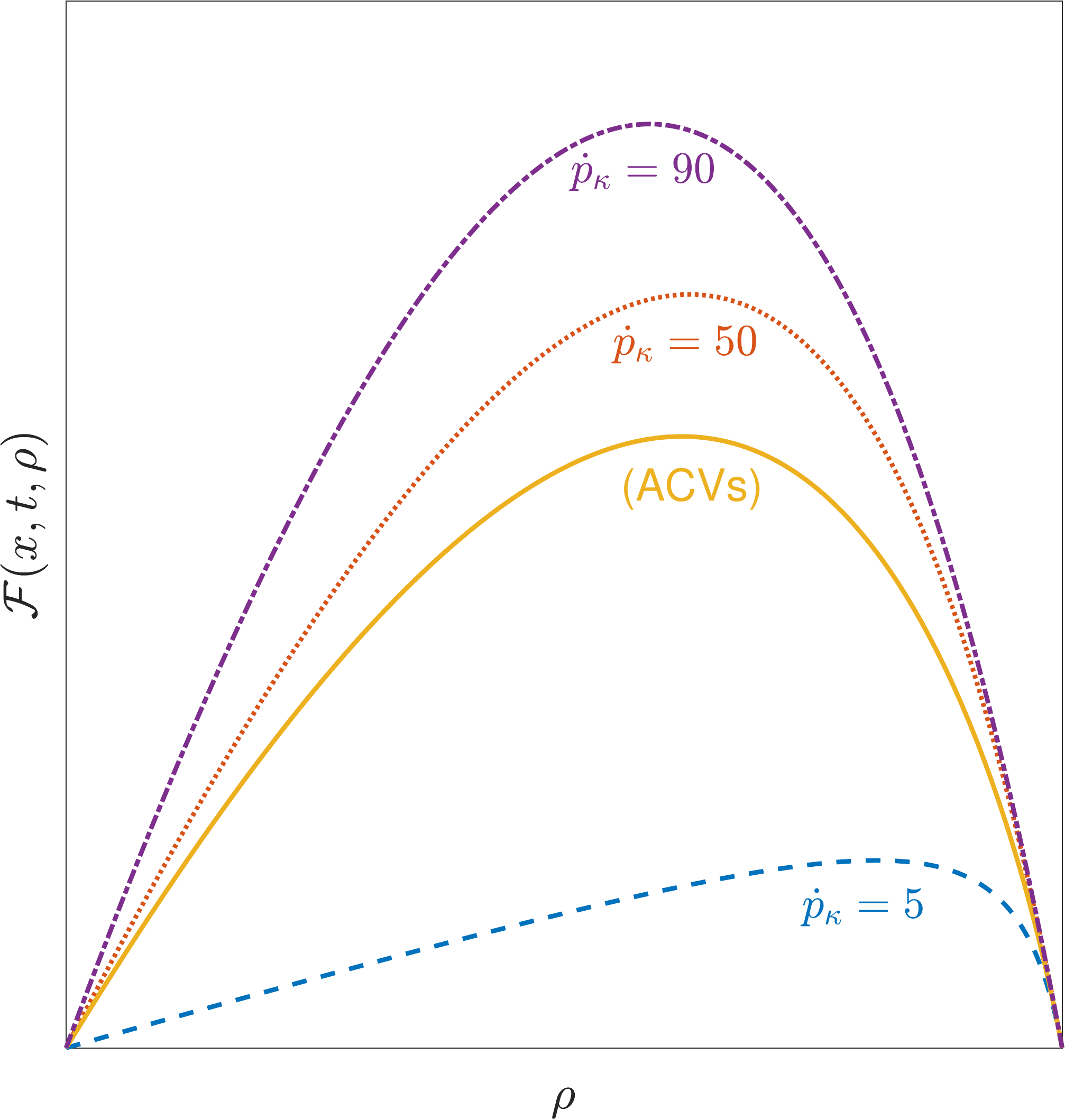}
\caption{Left: flux function $\ff(x,t,\rho)$ at fixed $t$ as $\rho$ and $x$ change; the circles represent $(\sigma(x,t),\ff^{\max}(x,t))$ as $x$ changes. Right: comparison of the flow function given by the \ref*{vicina} and \ref*{media} approaches, assuming to have three vehicles travelling with different speed and close enough to have $\phi(x,t)>1$; the solid line represents the flux obtained by the \ref*{media} formula; the dotted lines represent the flux obtained with the \ref*{vicina} formula with closest vehicle moving with velocity $\dot{p}_\kappa=90, 50, 5\,\kmh$.}
\label{fig:chi}
\end{figure}
%

\medskip

Let us consider now a numerical grid on a road $[a,b]$. Let $N_{x}$ be the number of cells $[x_{j-1/2},x_{j+1/2})$ of size $\dx$, and $N_t+1$ the number of intervals of length $\dt$ into which we divide the period of time $[0,T]$. Let $\rho^{n}_{j}=\rho(x_{j},t^{n})$ be the density of vehicles into the cell $x_{j}$ at time $t^{n}$, defined as the cell average
\begin{equation*}
	\rho^{n}_{j} =\frac{1}{\dx} \int_{x_{j-1/2}}^{x_{j+1/2}}\rho(x,t^{n})dx.
\end{equation*}

To approximate the model \eqref{eq:model1ord} we use the Cell Transmission Model (CTM) \cite{daganzo1994TRB}. 
The numerical scheme has the following structure
\begin{equation}\label{eq:schema}
	\rho^{n+1}_{j} = \rho^{n}_{j}-\frac{\dt}{\dx}(F^{n}_{j+1/2}-F^{n}_{j-1/2}),
\end{equation} 
with the numerical flux $F^{n}_{j+1/2}$ defined as
\begin{equation}\label{eq:numFlux1}
	F^{n}_{j+1/2}=\min\{S(x_j,t^n,\rho^{n}_{j}),R(x_{j+1},t^n,\rho^{n}_{j+1})\},
\end{equation}
where $S$ and $R$ are the sending and receiving functions introduced in \eqref{eq:SRfunct}. At each time step $n$ and for each cell centered in $x_j$, the critical densities $\sigma(x_j,t^n)$ must be computed numerically. As observed in Remark \ref{remark:sigma}, 
the computational costs are reduced by calculating the values $\ff^n_j(\rho)=\ff(x_j,t^n,\rho)$ as $\rho\in[0,\rhomax]$ varies and applying a search for the maximum of these values to compute the index $l$ such that $\sigma(x_j,t^n)\approx \rho_l$. For instance, on the MATLAB platform the \textnormal{\texttt{max}} function allows to efficiently compute the index of the maximum of the sampled flux values and is significantly faster than the function \textnormal{\texttt{fsolve}} required to solve non-linear equations as \eqref{eq:eqSigma} or \eqref{eq:eqSigma2}.

With standard arguments it is possible to prove the following properties of the scheme \eqref{eq:schema}-\eqref{eq:numFlux1}. Both results given in the next Proposition \ref{prop:monotonia1} and \ref{prop:CFL1} apply to the \ref*{vicina} and \ref*{media} models. 
For completeness, in Appendix \ref{sec:proof} we give the details of the proofs for the \ref*{vicina} approach.

\begin{proposition}\label{prop:monotonia1}
The approximate solution $\boldsymbol\rho^n=(\ldots, \rho^n_{j-1},\rho^n_j,\rho^n_{j+1},\ldots)$ constructed via the scheme  \eqref{eq:schema}-\eqref{eq:numFlux1} satisfies the bounds
$$
0 \leq \rho^n_j \leq \rhomax \quad \mbox{ for all } j\in\Z,\ n\in\N
$$
if the following Courant-Friedrichs-Levy (CFL) condition holds:
\begin{equation}\label{eq:CFL_flux}
	\disp\frac{\dt}{\dx}\disp\max_{(x,t,\rho)}\big|\partial_\rho\ff(x,t,\rho)\big|\leq 1. 
\end{equation}
\end{proposition}


\begin{proposition}\label{prop:CFL1}
The stability of the scheme is guaranteed by the condition
\begin{equation}\label{eq:CFL1}
	\dt\leq 
	\disp\frac{\dx}{\max\big\{\sup_{(x,t)}\dot{p}_{\kappa(x,t)}(t), \umax\big\}}
\end{equation}
where $\dot{p}_\kappa$ is the velocity of the $\kappa$-th vehicle, with $\kappa$ defined in \eqref{eq:kvicino}, and $\umax$ is the maximum velocity value given by the analytical model.
\end{proposition}

\begin{remark} If all the trajectories on the road have a speed always lower than the parameter $\umax=u(0)$ of the analytical model, then we use the standard CFL condition $\dt\leq \dx/\umax$. 
\end{remark}

\medskip

Now we propose an example to show the differences between the two approaches \ref*{vicina} and \ref*{media}.
Recall that $u(\rho) =\umax(\rhomax-\rho)/\rhomax$. We fix the following parameters: $\rhomax=100\,\vehkm$, $\umax=90\,\kmh$, $a=0$, $b = 3\,\km$, $\dx = 0.1\,\km$, $T=1\,\mymin$. The time step $\dt = 0.2\,\second$ is chosen to satisfy the CFL condition in \eqref{eq:CFL1}. The function $\chi$ in \eqref{eq:utrue} is defined in \eqref{eq:chi}
with $\ell=2\dx$ and $L=6\dx$. We test two different initial data, namely a rarefaction wave and a shock one, given by 
\begin{equation*}
	\rho_{0}(x) = \begin{cases}
		45 &\quad\text{if $x<1.5$}\\
		30 &\quad\text{if $x\geq 1.5$}
	\end{cases}\qquad\text{and}\qquad
	\rho_{0}(x) = \begin{cases}
		20 &\quad\text{if $x<1.5$}\\
		40 &\quad\text{if $x\geq 1.5$}.
	\end{cases}
\end{equation*}

Since the two definitions of $\uu$ in \eqref{eq:utrue} and \eqref{eq:umedio} differ for $\phi(x,t)>1$ with $\phi$ in \eqref{eq:phi}, we fix the trajectories of three vehicles as 
$$ p_{1}(t) = 1+10t, \quad p_{2}(t)=1.001+25t \ \mbox{ and }\  p_{3}(t)=1.002+50t,$$
 so that $\phi(x_{k},t)>1$ for some cells $x_{k}$ and time $t$. In Figure \ref{fig:test1medio} we compare the results at different times starting from the rarefaction wave on the top and the shock one on the bottom. Let us consider the rarefaction case: the main differences between the solutions to \eqref{eq:model1ord} with $\uu$ in \eqref{eq:utrue} and \eqref{eq:umedio} can be observed in plots \protect\subref{fig:medio1}, \protect\subref{fig:medio2} and \protect\subref{fig:medio3}, i.e.\ when the three vehicles are really close to each other ($\phi(x,t)>1$). Indeed, the solutions related to $\uu$ in \eqref{eq:utrue} (red-solid lines) show oscillations with higher picks than those related to \eqref{eq:umedio} (blue-dotted lines), where the velocity value is averaged. At the final time of the simulation, plot \protect\subref{fig:medio4}, the two solutions are quite similar, since the three vehicles are far enough from each other ($\phi(x,t)\leq1$ for each $x$) and the two definitions of $\uu$ in \eqref{eq:utrue} and \eqref{eq:umedio} coincide. A similar analysis holds for the shock wave initial data.
Therefore, the main difference between the two dynamics (red-solid and blue-dotted lines) concerns the height of the peaks of the oscillations caused by the individual behavior of the tracked vehicles.

\begin{figure}[h!]
\centering
\subfloat[][Rarefaction $t=3\dt$]{\label{fig:medio1}
\includegraphics[width=0.233\columnwidth]{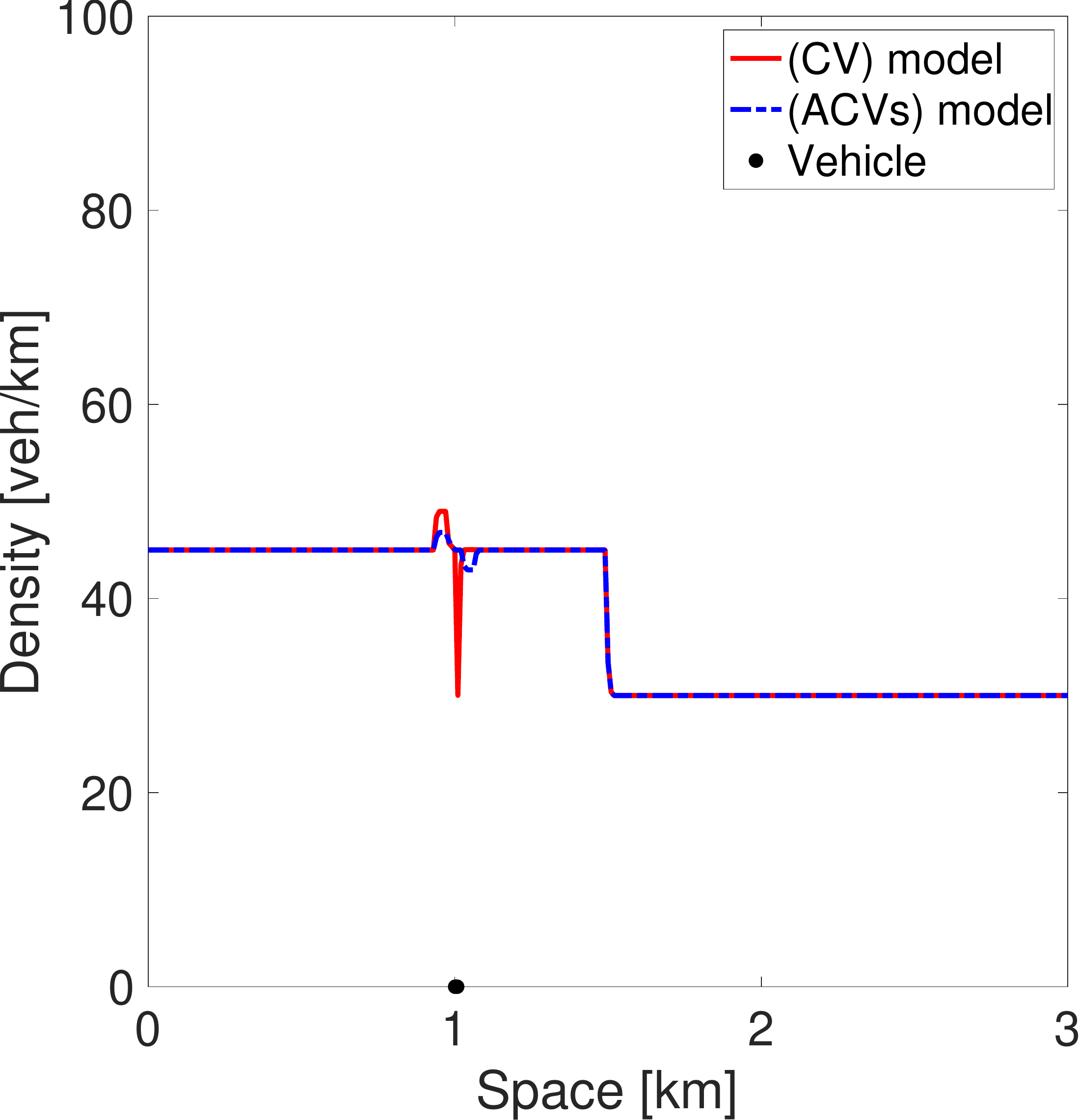}
}
\subfloat[][Rarefaction $t=10\dt$]{\label{fig:medio2}
\includegraphics[width=0.233\columnwidth]{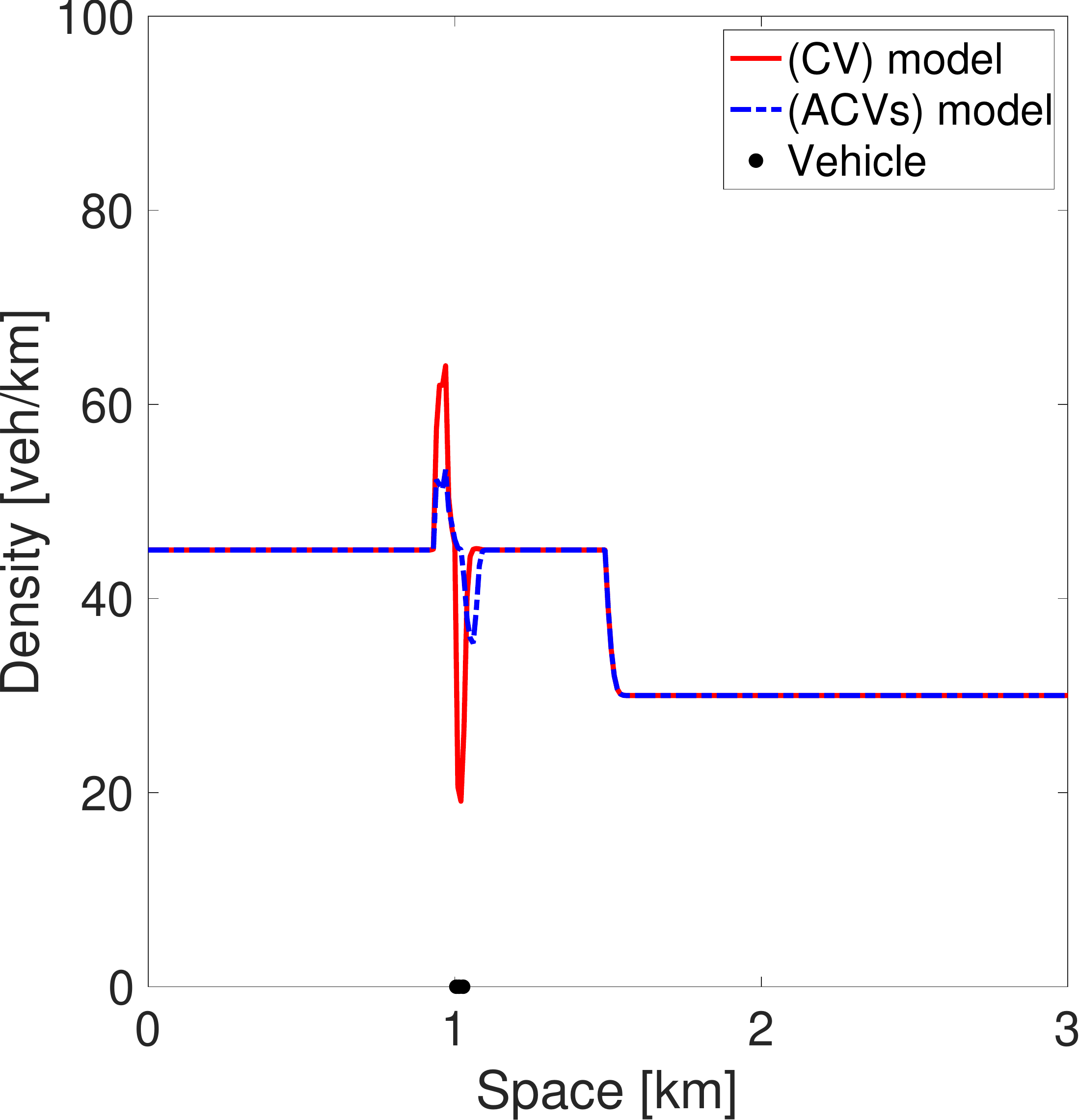}
}
\subfloat[][Rarefaction $t=40\dt$]{\label{fig:medio3}
\includegraphics[width=0.233\columnwidth]{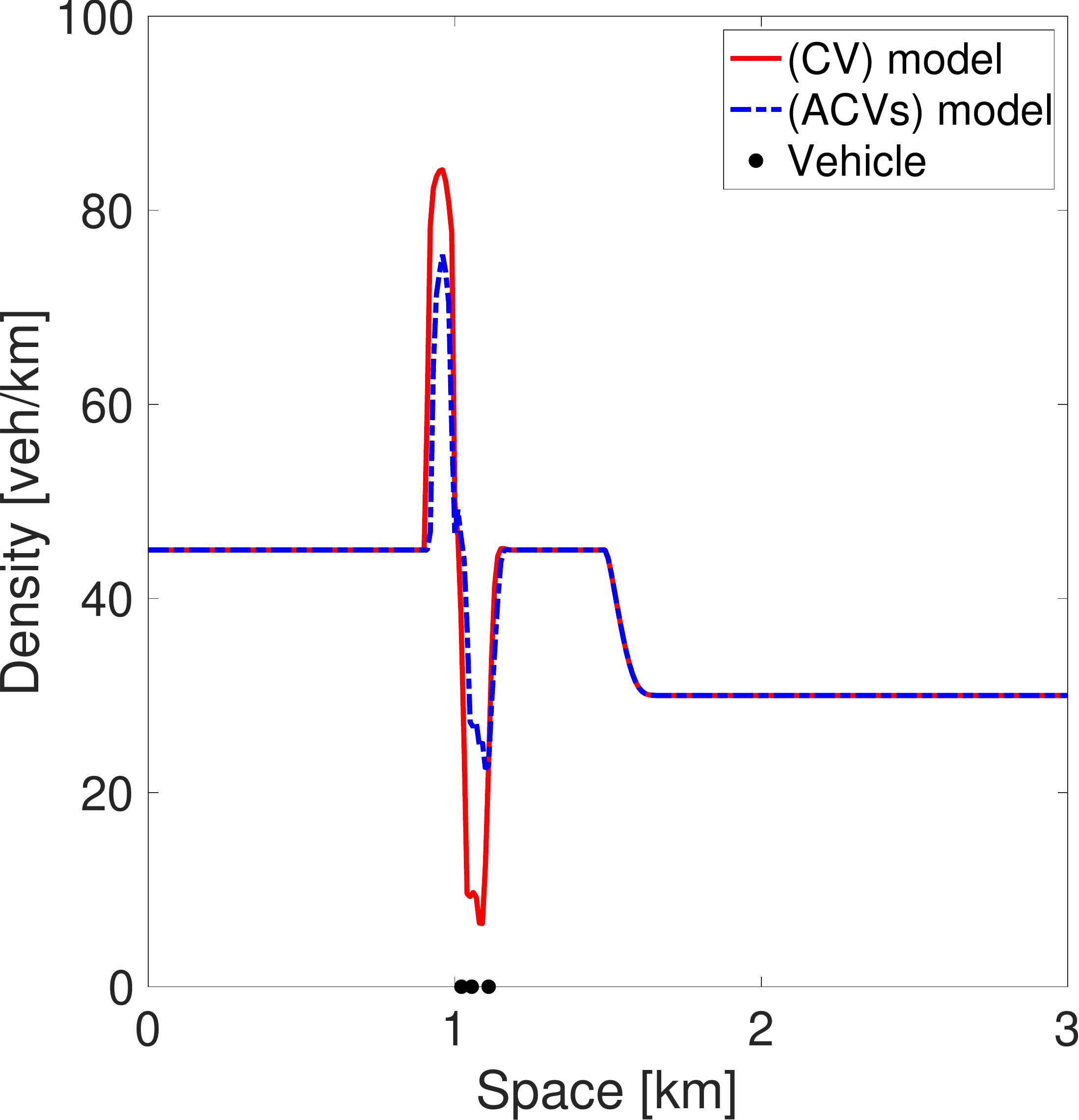}
}
\subfloat[][Rarefaction $t=T$]{\label{fig:medio4}
\includegraphics[width=0.233\columnwidth]{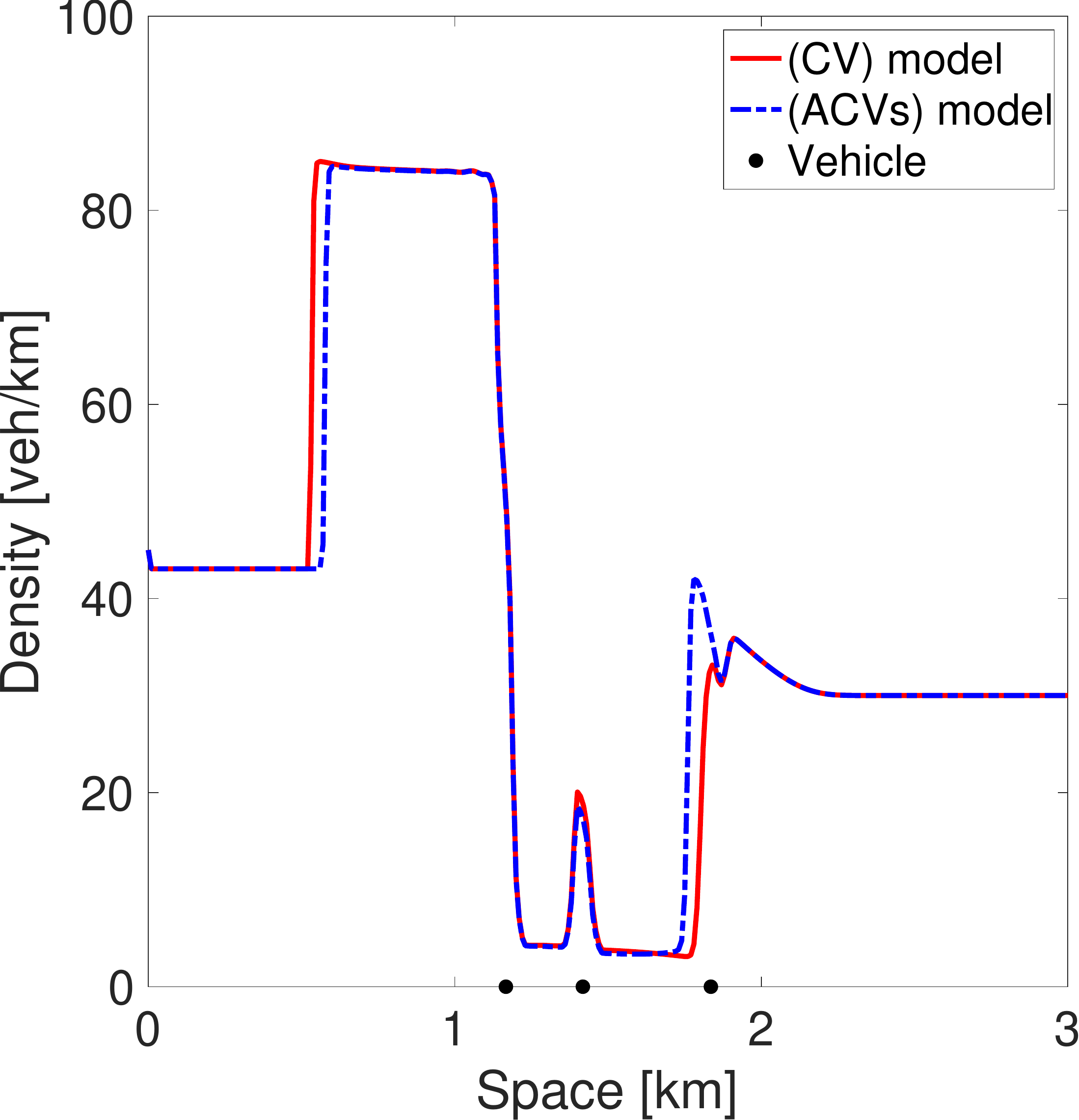}
}
\\
\subfloat[][Shock $t=3\dt$]{\label{fig:medio1s}
\includegraphics[width=0.233\columnwidth]{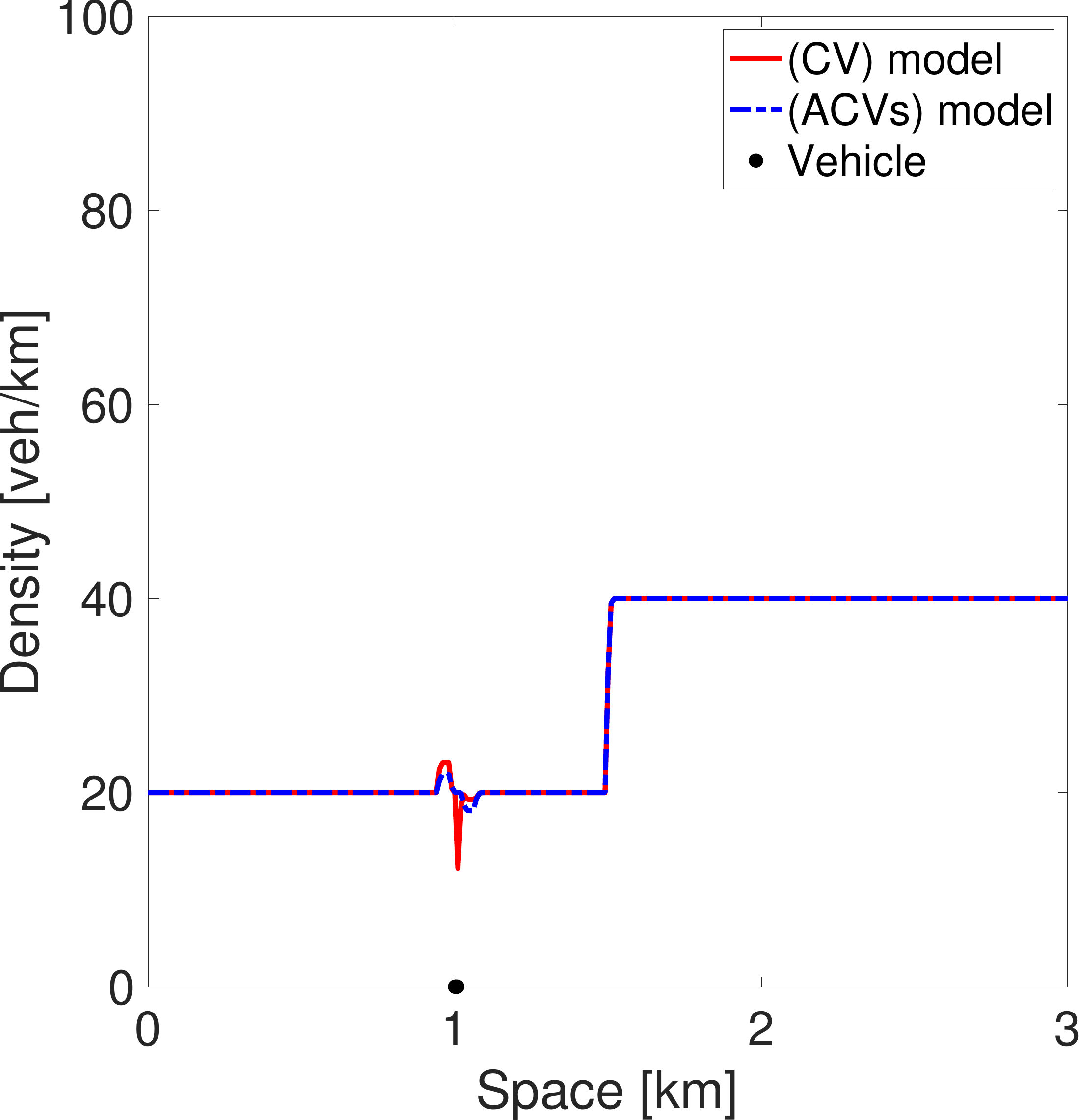}
}
\subfloat[][Shock $t=10\dt$]{\label{fig:medio2s}
\includegraphics[width=0.233\columnwidth]{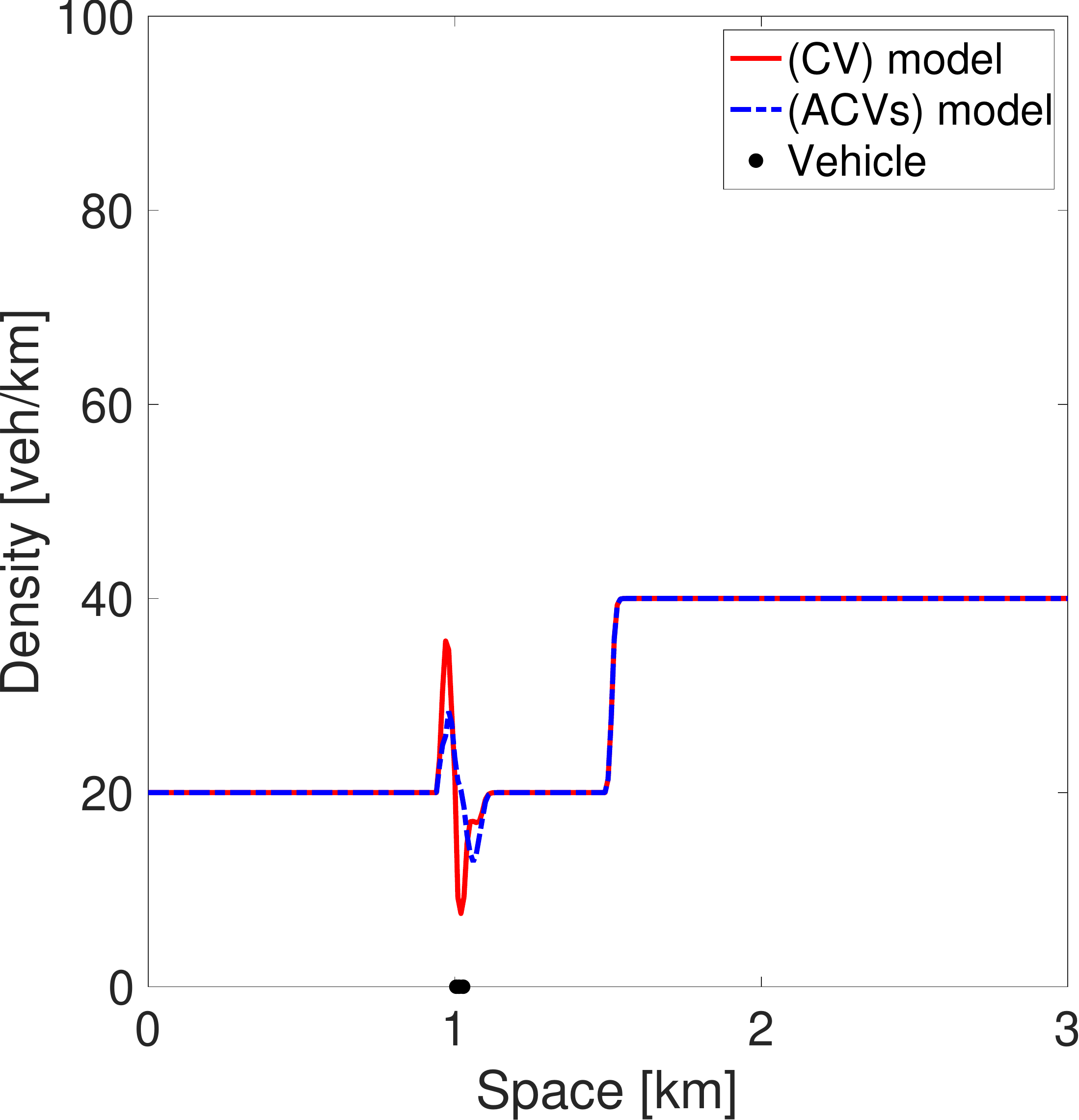}
}
\subfloat[][Shock $t=40\dt$]{\label{fig:medio3s}
\includegraphics[width=0.233\columnwidth]{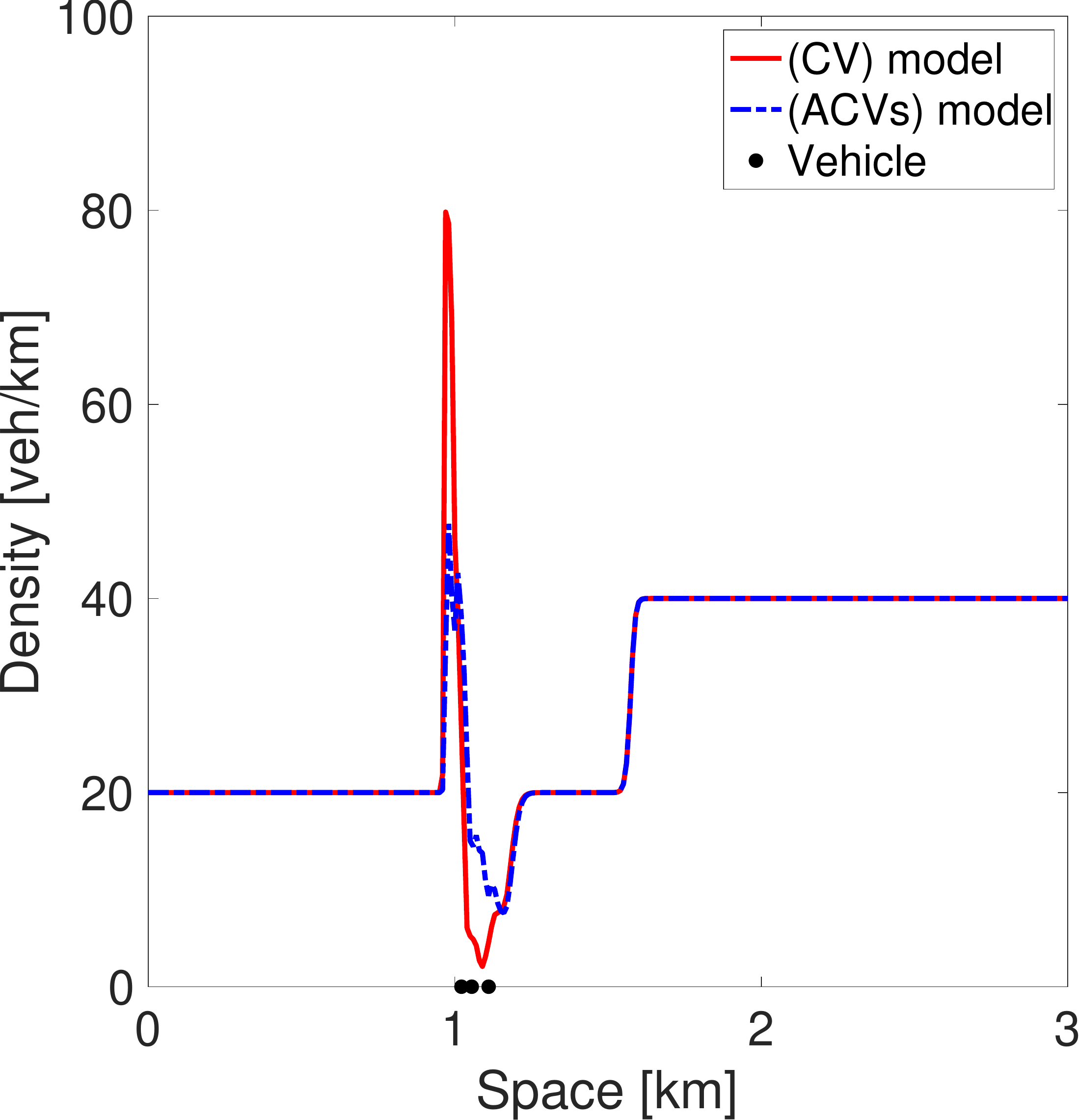}
}
\subfloat[][Shock $t=T$]{\label{fig:medio4s}
\includegraphics[width=0.233\columnwidth]{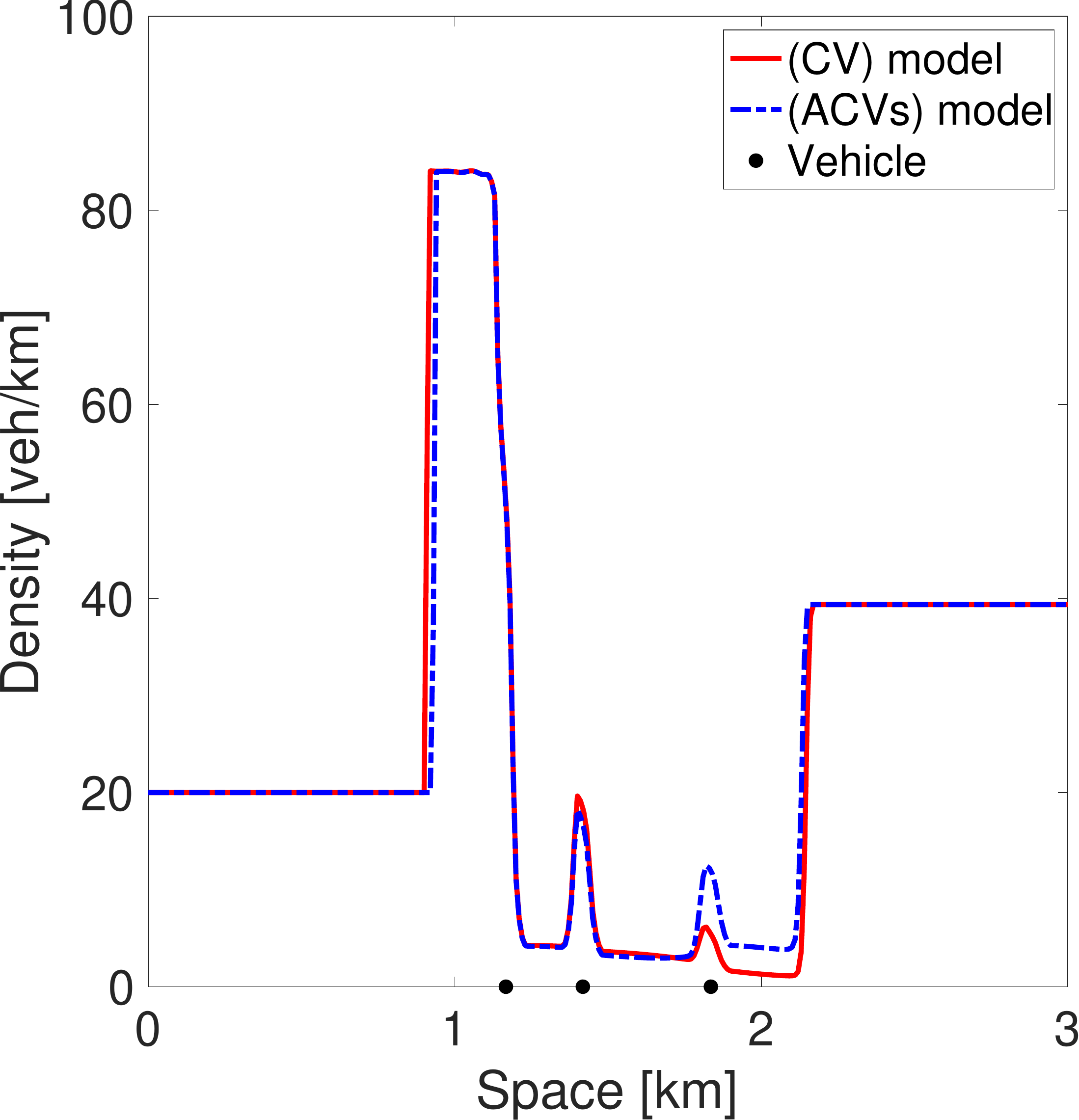}
}
\caption{Comparison between model \eqref{eq:model1ord} with velocity function $\uu$ defined in \eqref{eq:utrue} and \eqref{eq:umedio}, red-solid and blue-dotted lines respectively, at different times of the simulation starting from a rarefaction wave (top) and shock one (bottom).}
\label{fig:test1medio}
\end{figure}

\bigskip

We conclude this section by highlighting another important feature of the model described above.
The proposed approach also allows the coupling of a macroscopic model with a microscopic one.
Put simply, trajectory data can be generated from a microscopic model, and then included in the velocity function $\uu$ following the \ref*{vicina} or \ref*{media} methodology.
By coupling, for example, the first order macroscopic model \eqref{eq:model1ord} with a second order microscopic one of the follow-the-leader type, the resulting \textit{multiscale model} would have the following form
\begin{equation}\label{eq:micromacro}
\left\{\hspace{-0.15cm}\begin{array}{ll}
\de_{t}\rho+\de_{x}(\rho\,\uu(x,t,\rho;\bp)) = 0 &
\smallskip\\
\dot p_i(t) = V_i(t) & i=1,\ldots,N
\smallskip\\
\dot V_i(t) = A(p_i(t), p_{i+1}(t), V_i(t), V_{i+1}(t)) & i=1,\ldots,N-1
\smallskip\\
\dot V_N(t) = 0, &
\end{array}\right.
\end{equation}
where the $N$-th vehicle is the leader which has a special dynamics, and the function $A$ represents the acceleration to be defined, see for instance \cite{zhao2017TRB}.

We propose a numerical test, to show how the model \eqref{eq:micromacro} can reproduce typical traffic phenomena. In Figure \ref{fig:microSim} on the left, we generate trajectories by the second order microscopic model used in \cite{cristiani2019DCDSB} (equations (1)-(6)-(7)), which is specifically designed to reproduce stop \& go waves. Starting from a constant initial density along a road, in Figure \ref{fig:microSim} on the right we observe how the microscopic dynamics cause a variation on the macroscopic density, leading to
the reproduction of the stop \& go phenomenon at the macroscopic level. In this simulation, we have used the \ref*{vicina} approach. 
For completeness, in Figure \ref{fig:microSim2} we superpose the density profile obtained with the \ref*{vicina} method to the \ref*{media} one. The two profiles are very similar because the not-perturbed vehicles generated by the microscopic dynamics are all moving at the same speed.

The advantage of using a model of the type proposed in \eqref{eq:micromacro} is clear from Figure \ref{fig:microSim3}, where a small sample of microscopic data is sufficient to describe a stop \& go wave that is difficult to simulate through
a macroscopic model (particularly of the first order). 

\begin{figure}[h!]
\centering
\includegraphics[width=0.3\columnwidth]{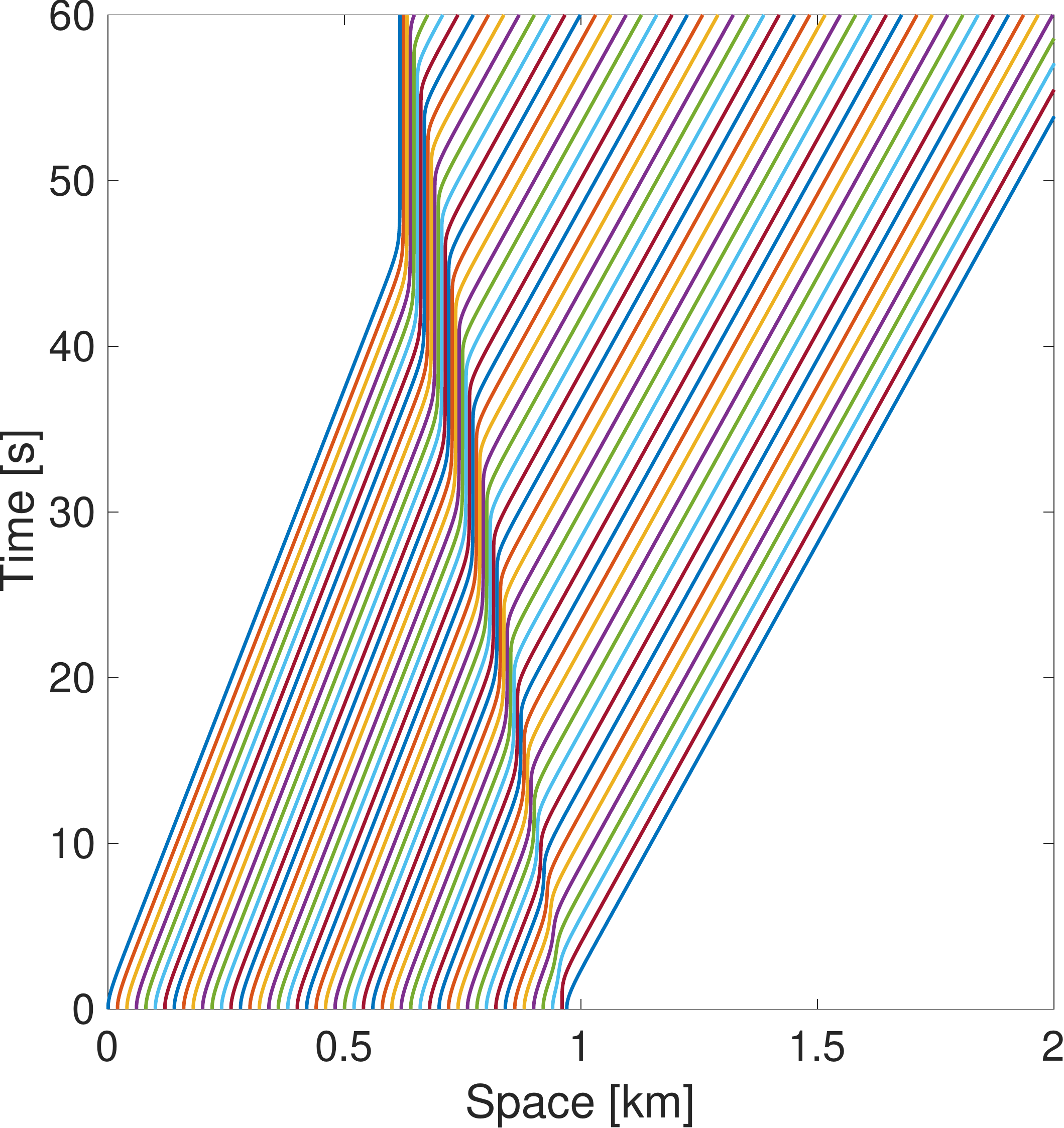}\qquad
\includegraphics[width=0.31\columnwidth]{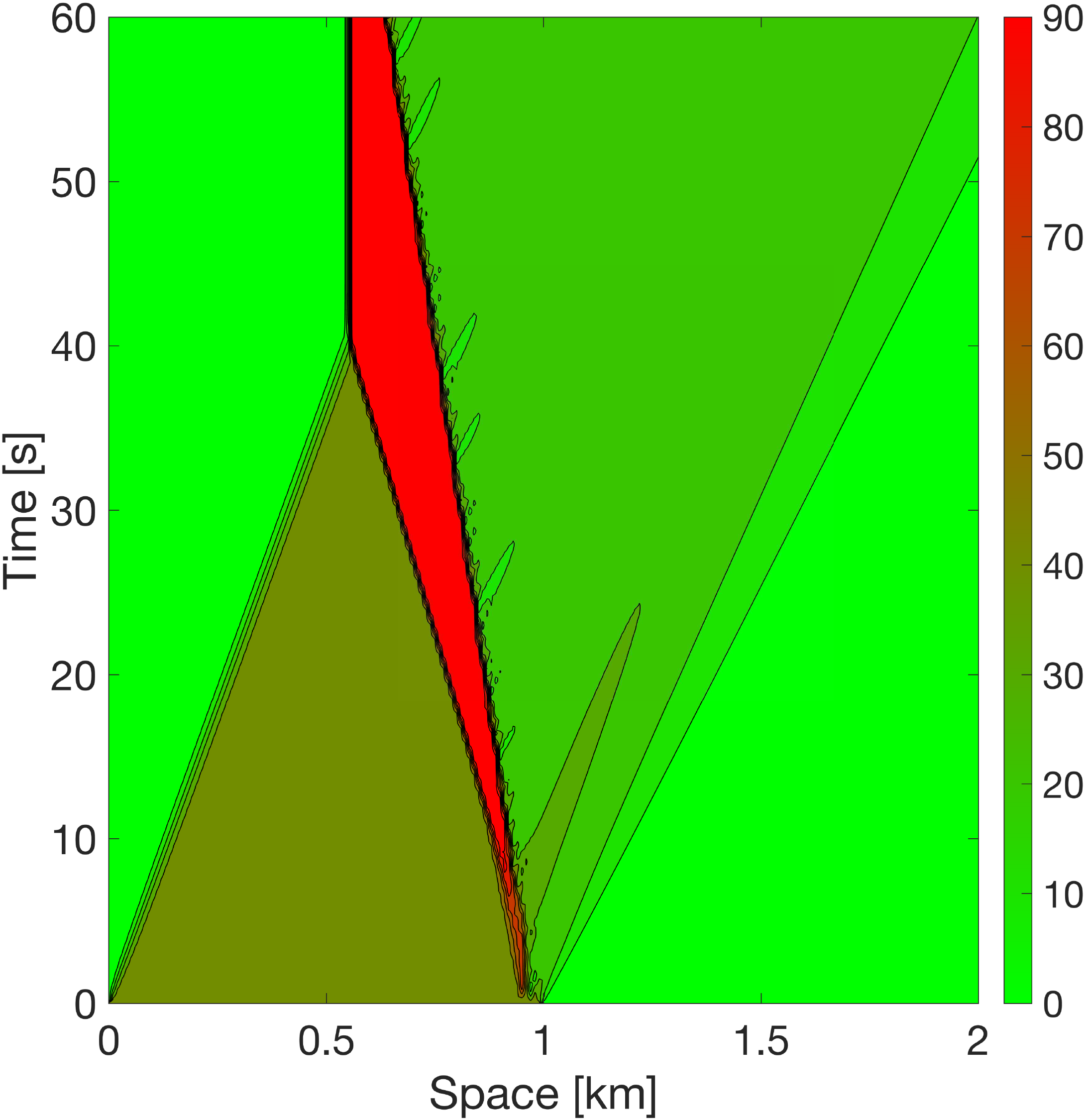}
\caption{Left: Trajectories generated by the second order microscopic model used in \cite{cristiani2019DCDSB}.  Right: Evolution in space and time of the macroscopic density $\rho=\rho(x,t)$ obtained by the multiscale model \eqref{eq:micromacro}-\eqref{eq:utrue} with the $N=50$ trajectories given on the left.}
\label{fig:microSim}
\end{figure}
\begin{figure}[h!]
\centering
\includegraphics[width=0.3\columnwidth]{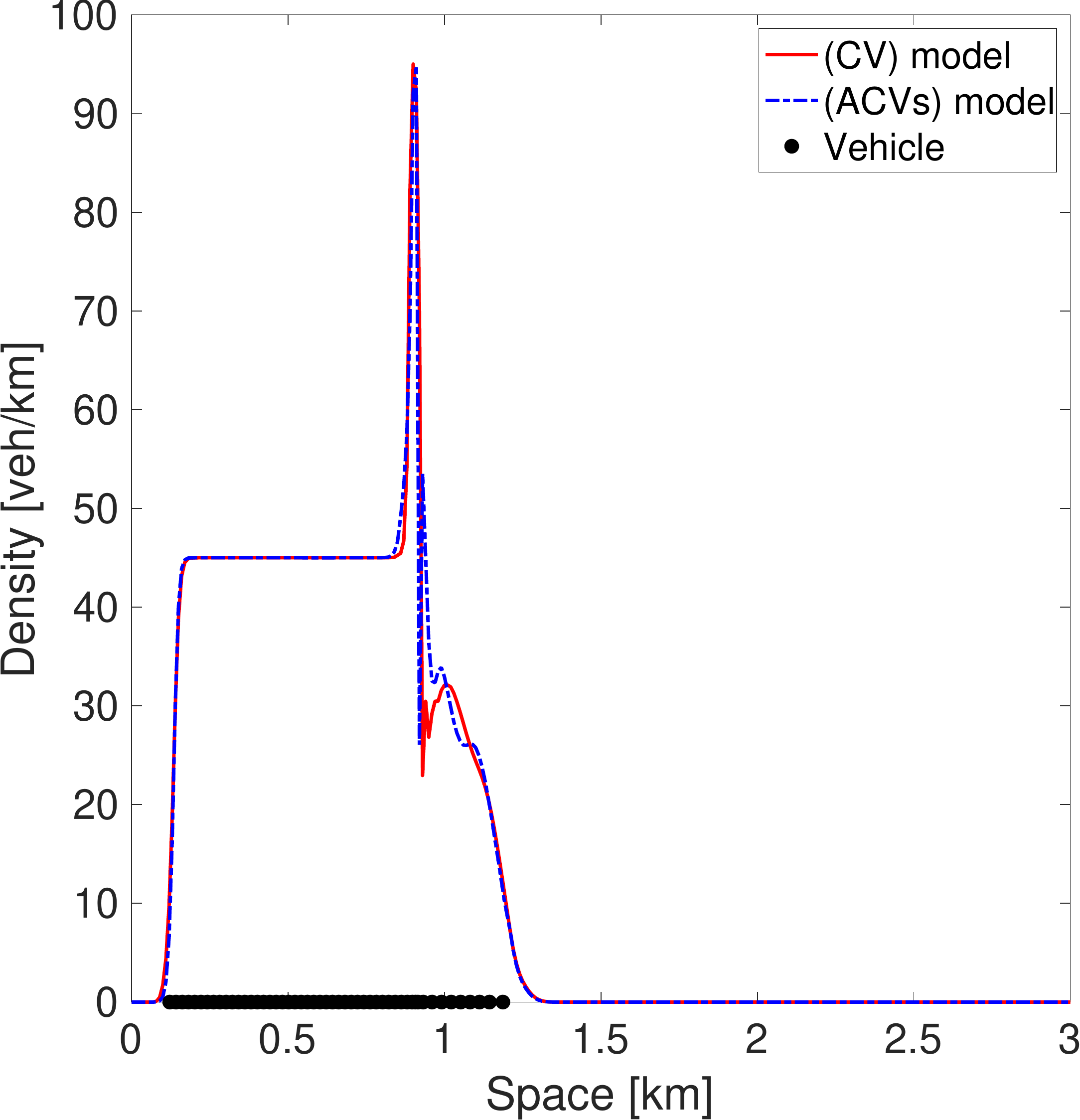}\qquad
\includegraphics[width=0.3\columnwidth]{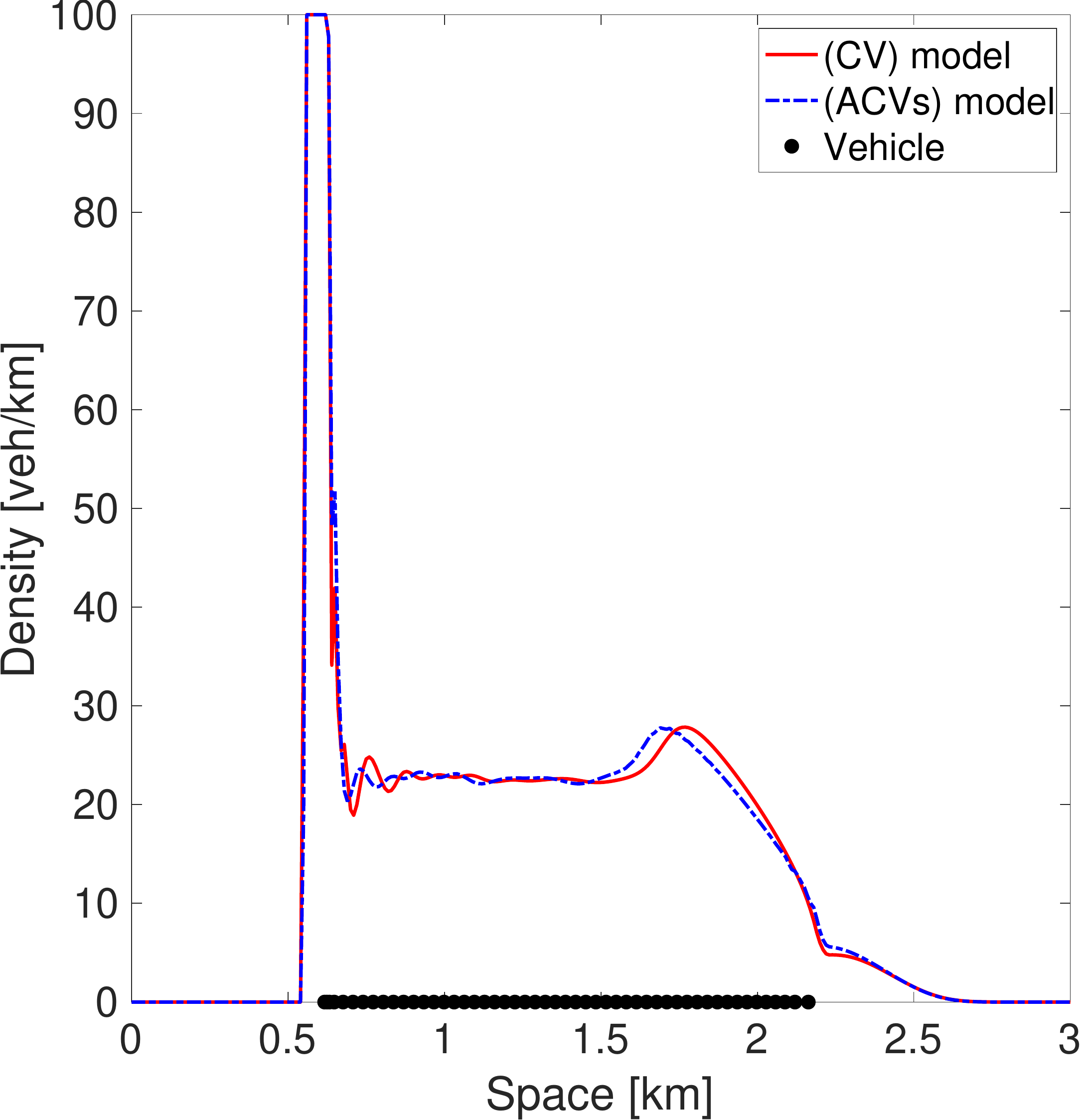}
\caption{Evolution of the density $\rho=\rho(x,t)$ along a road at time $t=50\dt$ (left) and $t=T$ (right), obtained by solving \eqref{eq:micromacro} with $\uu$ defined in \eqref{eq:utrue} and \eqref{eq:umedio}, represented by red-solid and blue-dotted lines, respectively. }
\label{fig:microSim2}
\end{figure}
\begin{figure}[h!]
\centering
\includegraphics[width=0.3\columnwidth]{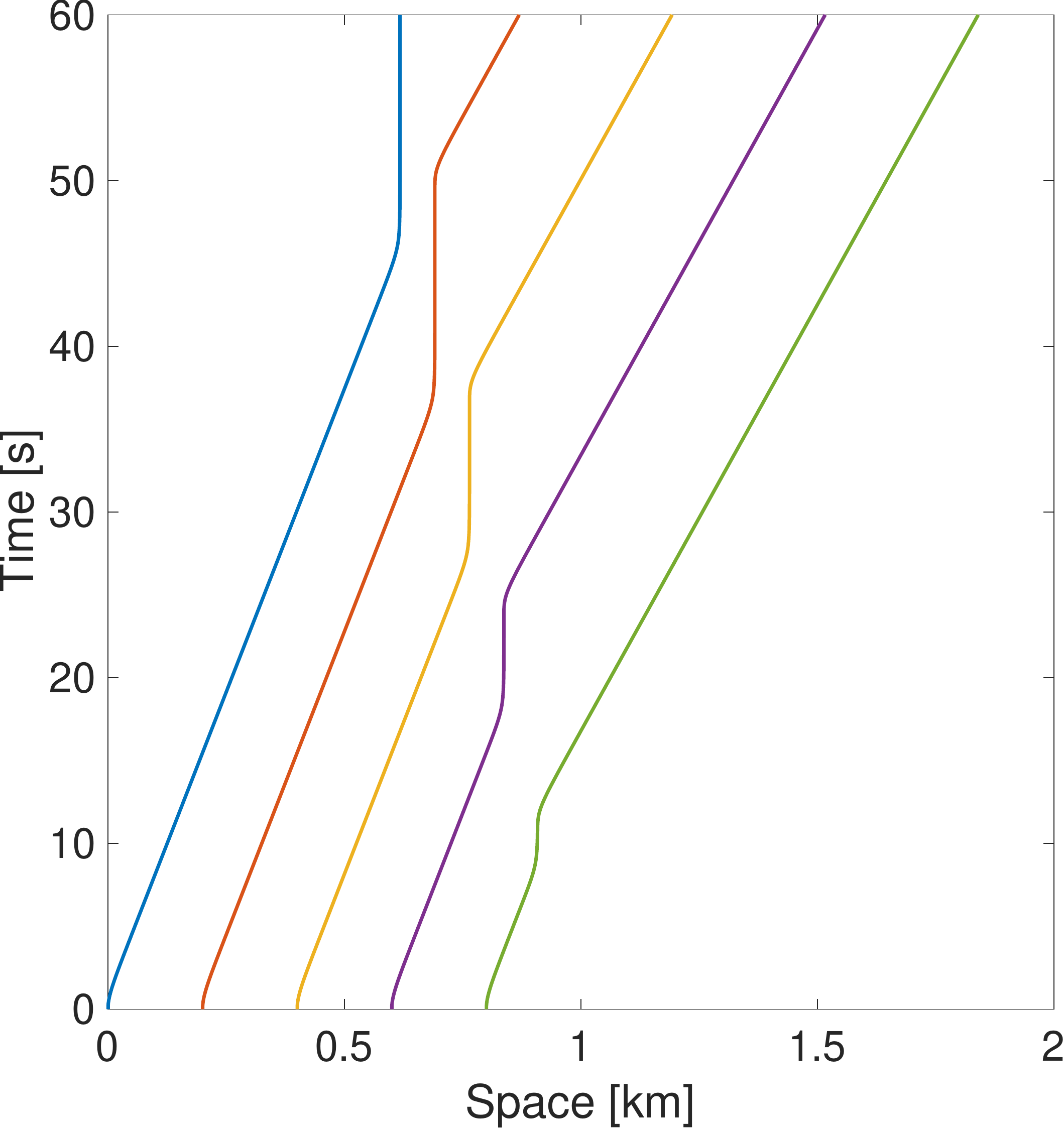}\qquad
\includegraphics[width=0.31\columnwidth]{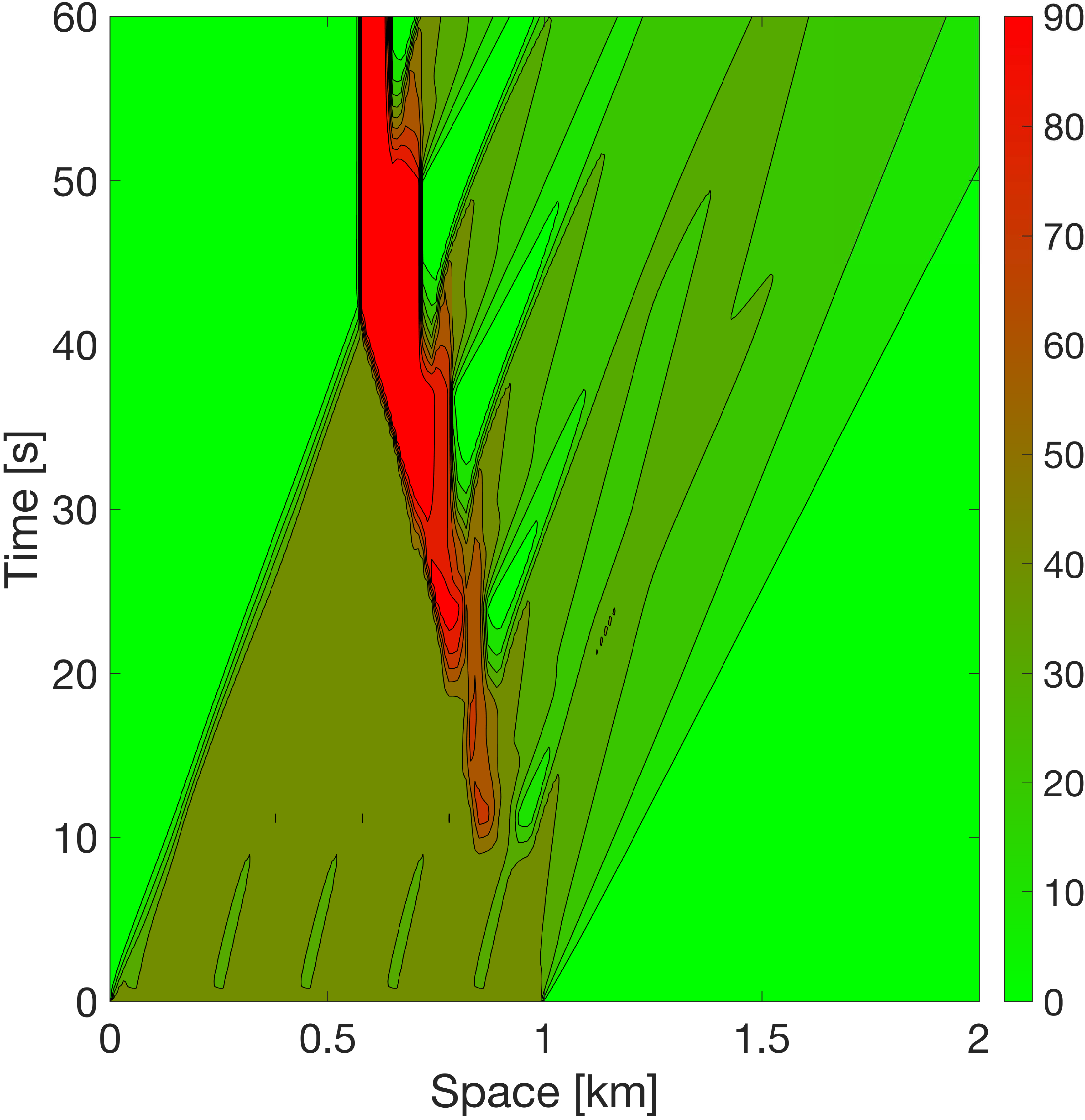}
\caption{Left: A sample of the trajectories drawn in Figure \ref{fig:microSim}-left.  Right: Evolution in space and time of the density $\rho=\rho(x,t)$ obtained by solving \eqref{eq:micromacro}-\eqref{eq:utrue} with only the $N=5$ trajectories given on the left.}
\label{fig:microSim3}
\end{figure}

\section{A macroscopic second order model embedding Lagrangian data} \label{sec:modello2}
In this section we apply the ideas introduced in Section \ref{sec:Ndata} to the Generic Second Order Models (GSOM) \cite{lebacque2007TTT}, described by 
\begin{align}
	\begin{split}
		&\begin{cases}
			\de_t\rho+\de_x(\rho \vv) = 0\\
			\de_t w+\vv\de_x w = 0,
		\end{cases}
	\end{split}
	\label{eq:modello2}
\end{align}
where $\rho(x,t)$ is the density of vehicles and $w(x,t)$ is a property of vehicles advected by the velocity function $\vv$. 
Similarly to the first order \ref*{vicina} case, given $N$ vehicles with known trajectory $\bp=(p_1(t),\ldots,p_N(t))$, the velocity function $\vv=\vv(x,t,\rho,w;\bp)$ is defined as previously given in the introduction by
\begin{equation}\label{eq:vel2ord}
	\vv(x,t,\rho,w;\bp) = \disp\chi(x-p_{\kappa}(t))\frac{2 \dot{p}_{\kappa}(t)v(\rho,w)}{\dot{p}_{\kappa}(t)+v(\rho,w)}+\big(1-\chi(x-p_{\kappa}(t))\big)v(\rho,w) 
\end{equation}
if $\dot{p}_{\kappa}(t)+v(\rho,w)\neq0$ and $\vv(x,t,\rho,w;\bp) = 0$ if $\dot{p}_{\kappa}(t)+v(\rho,w)=0$, with $\kappa=\kappa(x,t)$ given in \eqref{eq:kvicino}. 
Following \cite{FanSunPiccoliSeiboldWork2017}, we assume that the flux function $Q(\rho,w)=\rho v(\rho,w)$ satisfies these properties: $Q(\rho,w)\in \cuno$ and $Q(\rhomax,w) = 0$ for each $w$; the flux is strictly concave with respect to $\rho$ and it is non-decreasing with respect to $w$ for each $\rho$.
The flux function $Q$ defines the velocity function as $v(\rho,w) = Q(\rho,w)/\rho$. The assumptions on $Q$ imply that the function $v$ is regular and it is non-increasing with respect to $\rho$. 

As a consequence of the properties of $Q$ and $v$, the flux function $\qq(x,t,\rho,w)=\rho\vv(x,t,\rho,w;\bp)$ with $\vv$ in \eqref{eq:vel2ord} is strictly concave with respect to the density $\rho$. 
Thus, as in the first order case, there exists a unique density value $\sigma=\sigma(x,t,w)$ where $\qq$ reaches its maximum $\qq^{\max}(x,t,w)=\qq(x,t,\sigma,w)$ with respect to $\rho$ and we can extend the definition of sending and receiving functions in \eqref{eq:SRfunct} as 
\begin{equation*}
	S = \begin{cases}
		\qq(x,t,\rho,w) & \text{if $\rho\leq\sigma(x,t,w)$}\\
		\qq^{\max}(x,t,w) & \text{if $\rho>\sigma(x,t,w)$}
	\end{cases}
	\qquad 
	R = \begin{cases}
		\qq^{\max}(x,t,w) & \text{if $\rho\leq\sigma(x,t,w)$}\\
		\qq(x,t,\rho,w) & \text{if $\rho>\sigma(x,t,w)$,}
	\end{cases}
\end{equation*}
with $\sigma(x,t,w)$ critical density obtained by numerically solving the non-linear equation $\de_{\rho}\qq(x,t,\sigma,w)=0$, for each triple of values $(x,t,w)$, see Remark \ref{remark:sigma}.

\medskip 

Under the same numerical setting introduced in Section \ref{sec:Ndata}, for $j=0,\ldots,N_x-1$, $n=0,\ldots,N_t$, we extend the scheme
\eqref{eq:schema} to 
\begin{equation}\label{eq:schema2}
\begin{cases}
	\rho^{n+1}_{j} = \rho^{n}_{j}-\displaystyle\frac{\dt}{\dx}(Q^{n}_{j+1/2}-Q^{n}_{j-1/2})\medskip\\
	w^{n+1}_{j} = w^{n}_{j}-\displaystyle\frac{\dt}{\dx}\,\vv^n_j\, (w^{n}_{j}-w^{n}_{j-1}),
\end{cases}
\end{equation}
with $\vv^n_j = \vv(x_j,t^n,\rho^n_j,w^n_j;\bp)$ and 
\begin{equation*}\label{eq:numflux}
	Q^{n}_{j-1/2}=\min\{S(x_{j-1},t^n,\rho^{n}_{j-1},w^{n}_{j-1}),R(x_j,t^n,\rho^{n}_{j},w^{n}_{j})\}.
\end{equation*}
Since the variable $w$ is advected forward by the motion of vehicles travelling with positive velocity, the second equation has been approximated with an \textit{up-wind} finite difference scheme. 

With computations similar to the first order case given in Proposition \ref{prop:monotonia1} and \ref{prop:CFL1}, the 
following CFL condition
\begin{equation*}
	\dt\leq\frac{\dx}{\max_{(x,t,\rho,w)}\vv(x,t,\rho,w)}
\end{equation*}
guarantees the stability of the scheme \eqref{eq:schema2}. Specifically, we get 
\begin{equation}\label{eq:CFL2}
	\dt\leq 
	\disp\frac{\dx}{\max\big\{\sup_{(x,t)}\dot{p}_{\kappa(x,t)}(t), \vmax\big\}},
\end{equation}
where $\dot{p}_\kappa$ is the velocity of the $\kappa$-th vehicle, with $\kappa$ defined in \eqref{eq:kvicino}, and $\vmax$ is the maximum velocity value given by the analytical model, i.e. $\vmax=\max_w v(0,w)$. 

\smallskip

The finite volume formulation given in \eqref{eq:schema2} allows for easy handling of flow and velocity data from fixed sensors. In order to exploit this information, the sensor data are used as boundary conditions on the incoming side of a road. Usually, sensor data are aggregated per minute
 and need to be interpolated to be available at each numerical time step $t^n$. Let $\qq^{n}_{\sens}$ and $\vv^n_{\sens}$ be the sensor flux and speed data after interpolation, respectively.
There are two main ways to include the real data into the numerical scheme \eqref{eq:schema2}: i) replace $Q^{n}_{-1/2} = \qq^{n}_{\sens}$ in the first cell $[x_{-1/2},x_{1/2}]$ making sure that the sensor datum is an admissible value with respect to the analytical bounds; ii) fix the density value $\rho^n_{0}=\qq^{n}_{\sens}/\vv^n_{\sens}$ and use it in the calculation of $Q^{n}_{-1/2}$.
In both cases, once the density at the left boundary has been computed, we estimate $w^{n}_{0}$ such that $v(\rho^{n}_{0},w^{n}_{0})=\vv^{n}_{\sens}$.

In the following, we employ the first approach i) that uses the sensor data in an almost pure form and avoids the additional approximation needed to calculate the density.

\subsection{Acceleration estimate}\label{sec:acc}
One of the main advantages of second order models is their greater accuracy in approximating velocity, which allows improvements in the estimate of vehicle acceleration.

Let us consider the model \eqref{eq:modello2}, with the velocity function $\vv$ in \eqref{eq:vel2ord}.
We compute the acceleration function $a=a(x,t)$ as the material derivative ($ D/Dt = d/dt + v \cdot d/dx$) of $\vv=\vv(x,t,\rho(x,t),w(x,t);\bp)$ with respect to the time $t$:
\begin{equation*}
	a(x,t)=\frac{D\vv}{Dt} = \vv_t + \vv_{\rho}\rho_{t} + \vv_{w}w_{t} + \Big(\vv_x + \vv_{\rho}\rho_{x}+ \vv_w w_x\Big)\vv(x,t).
\end{equation*}
Therefore, from \eqref{eq:modello2},
\begin{align}\label{eq:acc}
	 a(x,t) 
	& = \vv_{\rho}(\rho_{t}+\vv\rho_{x})+\vv_{w}(w_{t}+\vv w_{x})+\vv_{t}+\vv \vv_{x} = -\rho\vv_{\rho} \vv_{x}+\vv_{t}+\vv \vv_{x}.
\end{align}
For $(\dot{p}(t),v(\rho,w))\neq(0,0)$ we have
\begin{align*}
	\vv_{\rho} &= \chi(x-p(t))\frac{2v_{\rho}\dot{p}^{2}}{(\dot{p}+v)^{2}}+(1-\chi(x-p(t)))v_{\rho}\\
	\vv_{x} & =  \chi^{\prime}(x-p(t))\frac{v(\dot{p}-v)}{\dot{p}+v}\\
	\vv_{t} & = \chi^{\prime}(x-p(t))\dot{p}v\frac{v-\dot{p}}{\dot{p}+v}+\chi(x-p(t))\frac{2\ddot{p}v^{2}}{(\dot{p}+v)^{2}},
\end{align*}
where $\chi^\prime(\cdot)$ is the first derivative of $\chi$ with respect to its argument. Let us observe that the acceleration formula directly depends on the acceleration $\ddot{p}(t)$ of the vehicle that is influencing the speed function of the model.

\section{Emissions estimate from traffic quantities}\label{sec:emiModels}

Here we focus on the estimation of pollutants production at ground level due to vehicular traffic, whose impact on air quality is a long-standing and complex problem. Following \cite{balzotti2021DCDS}, we propose a computational approach that combines the traffic simulation model with an emission one.
We focus our attention on $\nox$ emissions, which are dangerous to human health and are precursors of the ozone, that also has negative effects on health \cite{seinfeld2016JWS, wayne1991CP}.

Once approximated the density, velocity and acceleration of vehicles, we use them to estimate the emissions produced by vehicular traffic. Here we consider two microscopic emission models and we follow \cite{balzotti2021DCDS} to extend them to macroscopic variables. 

Let us consider a vehicle $i$ moving at speed $v_{i}$ and subject to a certain acceleration $a_{i}$ at time $t$. The first microscopic emission model considered (see \cite{panis2006elsevier}), estimates the emissions associated with the motion of the vehicle as
\begin{equation}\label{eq:emissioniMicro}
	E_{i}(t) = \max\{E_{0},f_1 + f_2 v_{i}(t) + f_3 v_{i}(t)^2 + f_4 a_{i}(t) + f_5 a_{i}(t)^2 + f_6 v_{i}(t) a_{i}(t)\}
\end{equation}
where $f_{1}$--$f_{6}$ are coefficients that depend on the type of vehicle, fuel and pollutant considered. In Table \ref{tab:panis} we collect the coefficients associated with $\nox$ emissions for a petrol car, and we refer to \cite[Table 2]{panis2006elsevier} for an exhaustive list of coefficients related to other types of vehicles and pollutants.
\begin{table}[h]
\small
\centering
\begin{tabular}{ccccccc}\toprule
Vehicle mode & $f_1$ $\left[\frac{\mygram}{\mysecond}\right]$ & $f_2$ $\left[\frac{\mygram}{\mymeter}\right]$ & $f_3$ $\left[\frac{\mygram\, \mysecond}{\mymeter^{2}}\right]$ & $f_4$ $\left[\frac{\mygram\, \mysecond}{\mymeter}\right]$& $f_5$ $\left[\frac{\mygram\, \mysecond^{3}}{\mymeter^{2}}\right]$& $f_6$ $\left[\frac{\mygram \, \mysecond^{2}}{\mymeter^{2}}\right]$ \\
\midrule
If $a_i (t) \geq -0.5\,\mymeter\per\mysecond^2$ &  6.19e-04  & 8e-05  & -4.03e-06  & -4.13e-04  & 3.80e-04  & 1.77e-04\\
If $a_i (t) <-0.5\,\mymeter\per\mysecond^2$ &  2.17e-04  & 0  & 0  & 0  & 0  & 0\\\bottomrule
\end{tabular}
\caption{Coefficients in equation \eqref{eq:emissioniMicro} for $\nox$ emissions of an internal combustion engine car.}
\label{tab:panis}
\end{table}
\\
The second emission model is based on \cite{ahn1999TRB}. For
\begin{equation*}
	\boldsymbol{v}_{i}(t) = [1\quad v_{i}(t)\quad v_{i}^{2}(t)\quad v_{i}^{3}(t)]^{T}, \qquad 
	\boldsymbol{a}_{i}(t) = [1\quad a_{i}(t)\quad a_{i}^{2}(t)\quad a_{i}^{3}(t)]^{T},
\end{equation*}
the emissions associated with the vehicle $i$ are estimated by
\begin{equation}\label{eq:emissioniMicro2}
	E_{i}(t) = \exp(\boldsymbol{v}_{i}^{T}(t)\cdot P\cdot \boldsymbol{a}_{i}(t)),
\end{equation}
where $P$ is the following $4\times4$ matrix (see \cite[Appendix A1]{zegeye2011model})
\begin{equation*}
	P = 0.01\begin{bmatrix}
	-1488.31 & 83.4524 & 9.5433 & -3.3549\\
	15.2306 & 16.6647 & 10.1565 & -3.7076\\
	-0.1830 & -0.4591 & -0.6836 & 0.0737\\
	0.0020 & 0.0038 & 0.0091 & -0.0016
	\end{bmatrix}.
\end{equation*}

To extend these two models to the macroscopic case, we consider $M$ vehicles in a stretch of road moving at the same speed $\bar v(t)$ and subject to the same acceleration $\bar a(t)$. The emission rates associated with traffic in \eqref{eq:emissioniMicro} are modified in
\begin{equation}\label{eq:emissioni}
	E(t) = M\max\{E_{0},f_1 + f_2 \bar v(t) + f_3 \bar v(t)^2 + f_4 \bar a(t) + f_5 \bar a(t)^2 + f_6 \bar v(t) \bar a(t)\},
\end{equation}
while from \eqref{eq:emissioniMicro2} we obtain 
\begin{equation}\label{eq:emissioni2}
	E(t) = M \exp(\boldsymbol{\bar v}^{T}(t)\cdot P\cdot \boldsymbol{\bar a}(t)),
\end{equation}
with 
\begin{equation*}
	\boldsymbol{\bar v}(t) = [1\quad \bar v(t)\quad \bar v^{2}(t)\quad \bar v^{3}(t)]^{T}, \qquad 
	\boldsymbol{\bar a}(t) = [1\quad \bar a(t)\quad \bar a^{2}(t)\quad \bar a^{3}(t)]^{T}.
\end{equation*}

Hereafter we refer to the emission model \eqref{eq:emissioni} as the \emph{E-max-formula} and to \eqref{eq:emissioni2} as the \emph{E-exp-formula}. In Figure \ref{fig:confrontoMaxExp} we show a comparison between these two formulations at a microscopic level, observing that in this example their results are quite similar.
 The details of this numerical test are described in Section \ref{sec:emissioni}.

\begin{figure}[h!]
\centering
\includegraphics[width=0.3\columnwidth]{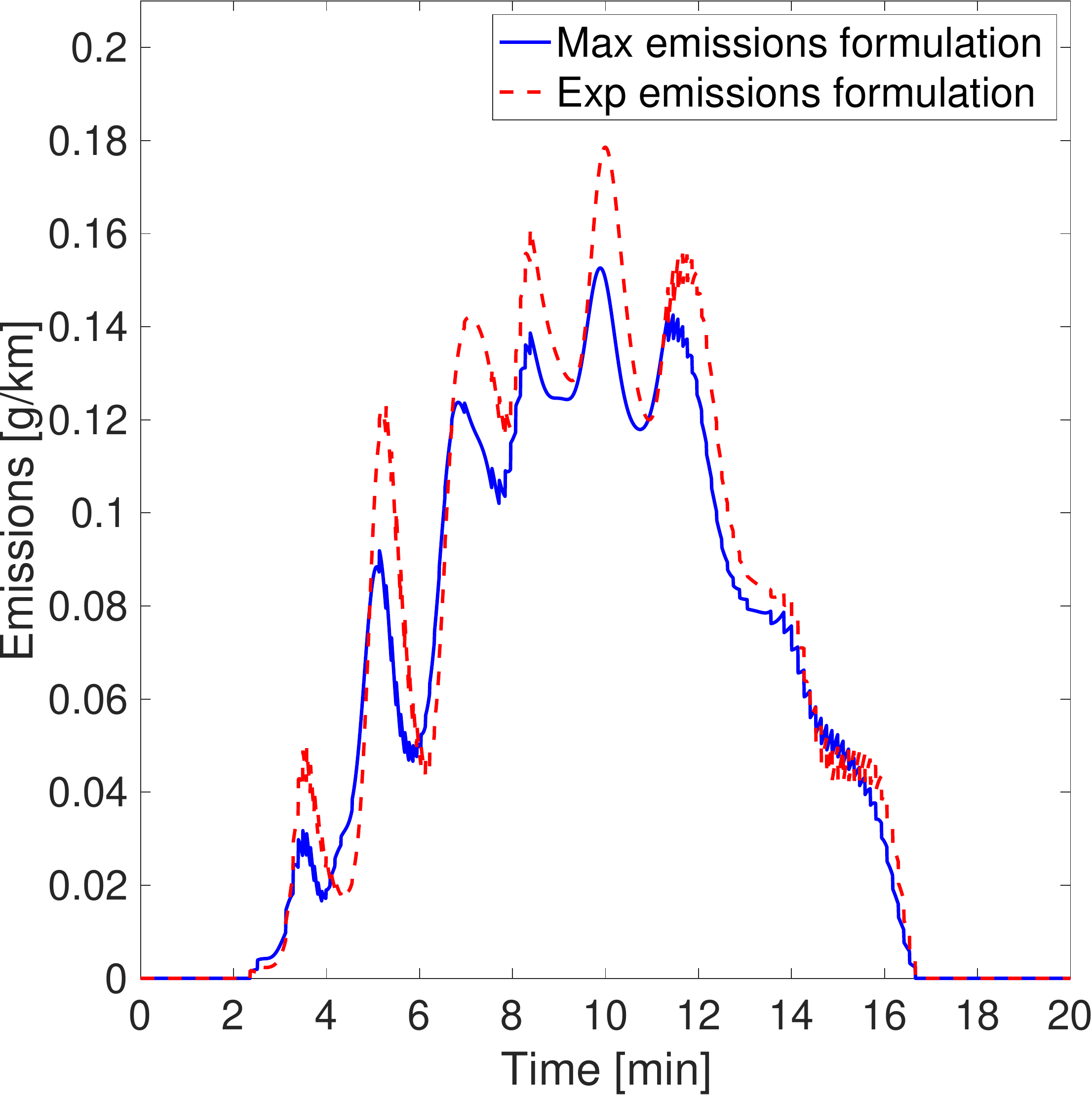}
\caption{Comparison between the microscopic \emph{E-max-formula} and the \emph{E-exp-formula} (see Section \ref{sec:emissioni}).}
\label{fig:confrontoMaxExp}
\end{figure}


\section{Numerical tests}\label{sec:test2O}
This section is devoted to the numerical tests. First, we focus on the second order model \eqref{eq:modello2}, then we analyze emission models and finally, we propose an application on a road network representing a portion of the Italian A4 motorway, combining real GPS data and fixed sensors.

\subsection{Traffic dynamics}\label{sec:numtraffic}
Let us begin with a test to illustrate the features of the second order traffic model \eqref{eq:modello2}-\eqref{eq:vel2ord}. First of all, we choose the Collapsed-Generalized-Aw-Rascle-Zhang (CGARZ) model \cite{FanSunPiccoliSeiboldWork2017} among the family of GSOM. 
In the CGARZ model the definition of the flow function is characterized by the distinction between the free and congested flow traffic regime. 
Hence, we define the flux function as
\begin{equation*}
	Q(\rho,w) = \begin{cases}
		Q_{f}(\rho) &\quad\text{if $\rho\leq\rhof$}\\
		Q_{c}(\rho,w) &\quad\text{if $\rho>\rhof$}
	\end{cases}
\end{equation*}
for a given density threshold $\rhof$ separating the two phases and
\begin{equation*}
Q_f(\rho) = g(\rho), \quad Q_c(\rho,w) = (1-\theta(w))f(\rho)+\theta(w)g(\rho),
\end{equation*}
where, following \cite{balzotti2021DCDS}, we set
\begin{equation}\label{eq:g}
	f(\rho) = \frac{\vmax}{\rhomax}\rhof(\rhomax-\rho), \quad
	g(\rho) = \frac{\vmax}{\rhomax}\rho(\rhomax-\rho), \quad
	\theta(w) = \frac{w-\wl}{\wr-\wl},
\end{equation}
for $\wl$ and $\wr$ minimum and maximum given value of $w$, respectively.
With these choices, the functions $Q(\rho,w)$ and $v(\rho,w)=Q(\rho,w)/\rho$ satisfy the hypotheses required by model \eqref{eq:modello2}.

The numerical simulations are performed with the CTM scheme \eqref{eq:schema2}.
We set $\rhomax=100\,\vehkm$, $\rhof=10\,\vehkm$, $\vmax=90\,\kmh$, $\wl=g(\rhof)$ and $\wr=g(\rhomax/2)$. The parameters $a$, $b$, $\dx$, $T$, $\dt$, the function $\chi$ and its parameters $\ell$ and $L$, are the same used for the test proposed in Section \ref{sec:Ndata}. 

We consider three vehicles and simulate their trajectory as follows
\begin{equation*}
p_{1}(t) = 0.5+15 t, \quad p_{2}(t) = 1+20 t, \quad p_{3}(t) = 1.2+22 t.
\end{equation*} 
We then fix the initial data
\begin{equation*}
	\rho_{0}(x) = 20\,\vehkm, \qquad w_{0}(x) = \begin{cases}
		\wr&\quad\text{if $x<1.5$}\\
		\wl&\quad\text{if $x\geq1.5$.}
	\end{cases}
\end{equation*}
Figure \ref{fig:test2} shows the final time of the simulation. We observe that the discontinuity in the variable $w$ generates a rarefaction wave on the right half of the road. The presence of the GPS trajectories, instead, produces non-classical shocks which are well captured by the CTM scheme. 

\begin{figure}[h!]
\centering
\includegraphics[width=0.3\columnwidth]{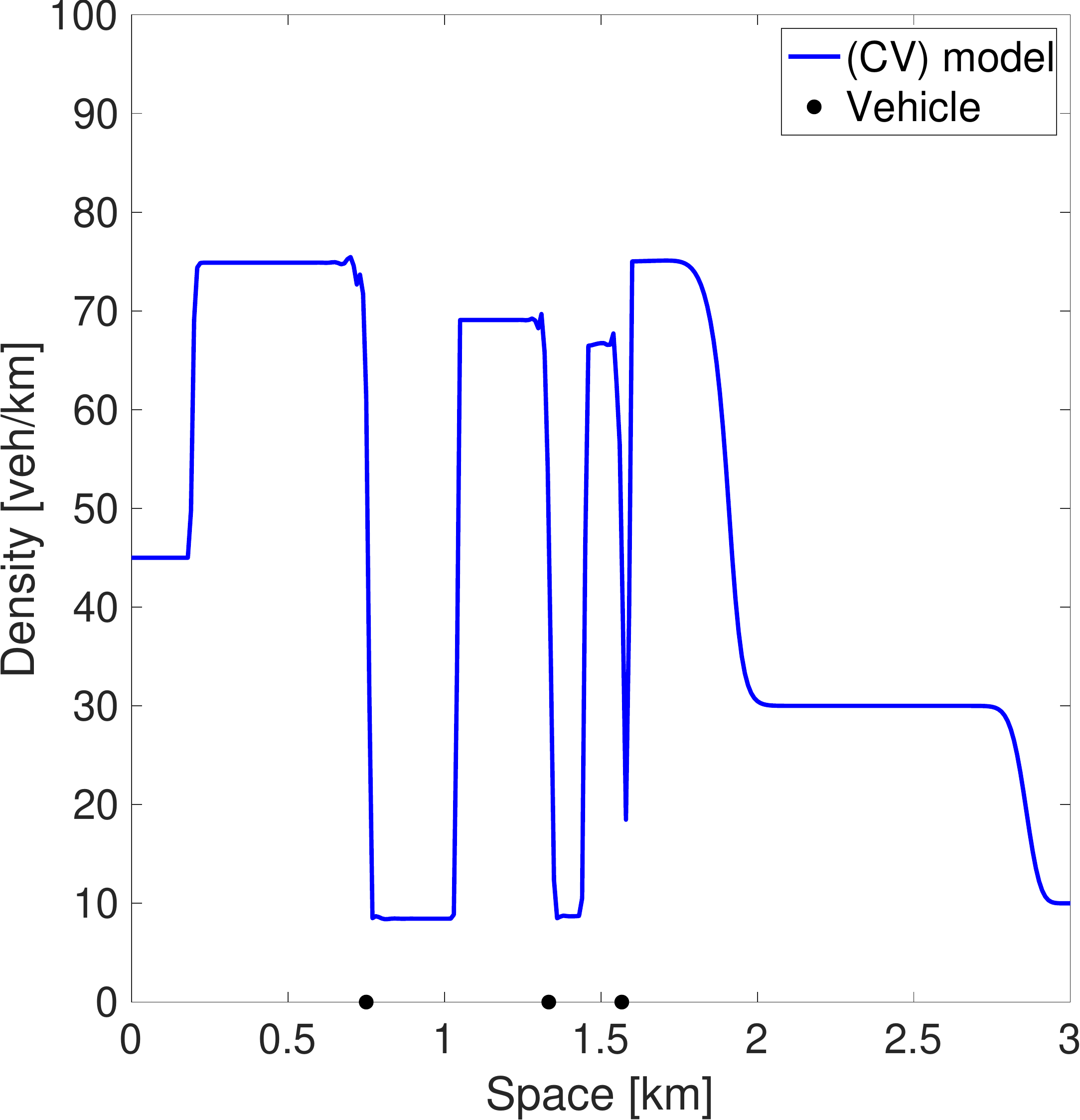}
\caption{Effects of monitored slow vehicles on the second order model \eqref{eq:modello2}, see Section \ref{sec:numtraffic}.}
\label{fig:test2}
\end{figure}

\subsection{Application to emissions}\label{sec:emissioni}
In this section we propose a test to show how the integration of GPS data impacts the estimate of emissions due to vehicular traffic. To this end, we consider $N$ vehicles moving according to the following equations during the time interval $[0,T]$:
\begin{equation}\label{eq:traiettorie}
\begin{split}
	x_{i}(t) &= c\vmax\Big(t-\frac{T}{k_{i}\pi}\cos\Big(\frac{k_{i}\pi t}{T}\Big)+\frac{T}{k_{i}\pi}\Big)+x_{0,i}\\
	v_{i}(t) &= c\vmax\Big(\sin\Big(\frac{k_{i}\pi t}{T}\Big)+1\Big) \\
	a_{i}(t) &= c\vmax\frac{T}{k_{i}\pi}\cos\Big(\frac{k_{i}\pi t}{T}\Big),
\end{split}
\end{equation}
with $c = 0.3$, $ k_{i}=20+5(i-1)/(N-1)$ and $x_{0,i}=1+0.05(i-1)$. 
With this choice, the vehicles are initially equally spaced every $50\,\meter$ and then are subjected to different accelerations and decelerations, which we expect will affect the emissions. We fix $T=20\,\mymin$ and the length of the road equal to $10\,\km$. Our analysis is then focused on a stretch of the road, specifically the interval $[4\,\km, 7\,\km]$, which is empty both at the beginning and at the end of the simulation, see Figure \ref{fig:traiettorie}.

\begin{figure}[h!]
\centering
\includegraphics[width=0.3\columnwidth]{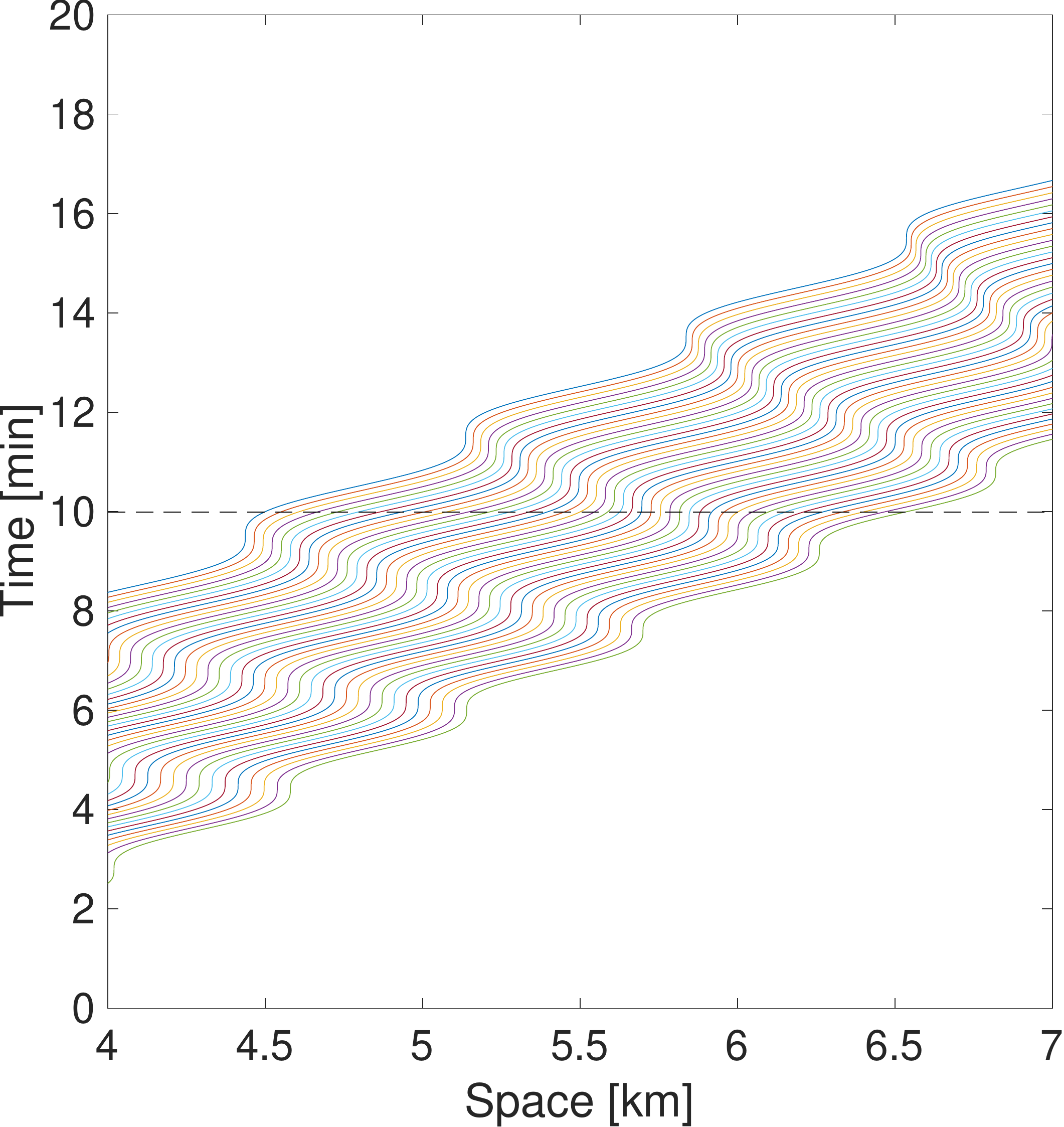}
\caption{Vehicle trajectories \eqref{eq:traiettorie} on a stretch of the road.}
\label{fig:traiettorie}
\end{figure}

In order to compare microscopic quantities with macroscopic ones, we first derive the macroscopic density and velocity from vehicle trajectories through a kernel density estimation (KDE). We use the Parzen-Rosenblatt window method \cite{parzen1962AMS,rosenblatt1956AMS}, which associates a density distribution with the position of the vehicles and then derives the global density by adding these distributions. More precisely, let $x_{i}(t)$ and $\nu_{i}(t)$ be the position and velocity of the $N$ vehicles, respectively. We define
\begin{align*}
\tilde\rho(x) = \sum_{i=1}^{N}\delta(x-x_{i}(t)),
\end{align*}
where $\delta$ is the Dirac delta function, so that
\begin{align*}
	\int_{\R}\tilde\rho(x)dx=\sum_{i=1}^{N}\int_{\R}\delta(x-x_{i}(t))dx=N.
\end{align*}
In order to recover the smooth density and velocity $\rho$ and $\nu$, we introduce the Gaussian kernel  
\begin{equation*}
K(x) = \frac{1}{2\pi h}\exp{\left(-\frac{x^{2}}{2h^{2}}\right)},
\end{equation*}
where $h$ is a smoothing parameters, which is chosen to obtain an almost constant density profile for equidistant vehicles. We then define
\begin{align*}
\rho(x,t) &= \int_{\R}K(x-\xi)\tilde\rho(\xi)d\xi = \sum_{i=1}^{N}K(x-x_{i}(t))\\
\nu(x,t) &= \frac{\sum_{i=1}^N \nu_i(t) K(x-x_i(t))}{\sum_{i=1}^N K(x-x_i(t))}.
\end{align*}
With this methodology we are able to reconstruct the initial density, $\rho_{0}=\rho(x,0)$, and velocity, $\nu_{0}=\nu(x,0)$, of the macroscopic model \eqref{eq:modello2}. Once $\rho_{0}$ and $\nu_{0}$ are known, we recover $w_{0}$ such that $v(\rho_{0},w_{0}) = \nu_{0}$.

The aim of this test is to compare the macroscopic emissions associated with the traffic model \eqref{eq:modello2} with the microscopic ones given by vehicle trajectories. Indeed, we estimate the traffic quantities $\rho^{n}_{j}$ and $v^{n}_{j}$ by means of the CTM scheme introduced in Section \ref{sec:modello2}, and the acceleration $a^{n}_{j}$ as described in Section \ref{sec:acc}. We then use these quantities in \eqref{eq:emissioni} and \eqref{eq:emissioni2} to obtain the macroscopic emissions with the two proposed \emph{E-max-formula} and \emph{E-exp-formula}, denoted by $E^{n}_{j}$ and $\widehat E^{n}_{j}$, respectively. At the same time, the microscopic velocity and acceleration given in \eqref{eq:traiettorie} are used in \eqref{eq:emissioniMicro} and \eqref{eq:emissioniMicro2} to recover the microscopic emissions $e^{n}_{i}$ and $\widehat e^{n}_{i}$ for each vehicle $i$. We then compute the total amount of emissions along the road at time $t^{n}$ as
\begin{equation*}
	E^{n}_{\tot} = \sum_{j=1}^{\nx}E^{n}_{j} \qquad\text{and}\qquad e^{n}_{\tot} = \sum_{i=1}^{N}e^{n}_{i},
\end{equation*}
similarly for $\widehat E^{n}_{\tot}$ and $\widehat e^{n}_{\tot}$.
Furthermore, we are interested in analyzing the effects on emissions of the number of tracked vehicles in model \eqref{eq:modello2}. To this end, let us introduce $\nn$ as the number of vehicles which influence the macroscopic traffic model. We perform three simulations: in the first one $\nn$ coincides with $N$, hence all the vehicles used for the microscopic dynamic influence the macroscopic model; then we consider only one vehicle out of two for model \eqref{eq:modello2}, hence $\nn=N/2$, and finally one in four, $\nn=N/4$. 

For our simulations, we fix a $10\,\km$ long road and a time horizon of 20 minutes. The model parameters are the same of the first numerical test of Section \ref{sec:numtraffic}. The bandwidth $h$ of the KDE is fixed as $h=100 \,\meter$. We use homogeneous Dirichlet boundary conditions on the left and homogeneous Neumann conditions on the right, which correspond to allow vehicles to leave the road. We then focus on the interval $[4\,\km, 7\,\km]$. 

On the top plots of Figure \ref{fig:testEmissioni} we compare the density of vehicles at $t=10\,\mymin$ computed with model \eqref{eq:modello2} and the one estimated through the KDE techniques. The black and fuchsia points represent vehicles position; in particular, the fuchsia points identify the vehicles which are tracked in model \eqref{eq:modello2}. From the plots we observe that the macroscopic model \eqref{eq:modello2} (blue-solid line) produces some peaks in correspondence of vehicles position, but the profile of the KDE density (red-dotted line) is quite well-captured. The accuracy of the results increases with the number of monitored vehicles $\nn$. The yellow line with circles represents the density associated with \eqref{eq:modello2} without considering the trajectory data, obtained by setting $\vv=v(\rho,w)$. We observe that the lack of GPS data produces a density profile completely different from the one recovered by microscopic data.

In the central (bottom) plots of Figure \ref{fig:testEmissioni} we compare $E^{n}_{\tot}$ ($\widehat E^{n}_{\tot}$), with and without trajectory data, and $e^{n}_{\tot}$ ($\widehat e^{n}_{\tot}$) during the whole simulation. We observe that the trend and the absolute values of emission rates using the \emph{E-max-formula} \eqref{eq:emissioni} are well captured by model \eqref{eq:modello2}, even when the number of tracked vehicles decreases. The total emission profile obtained from the model without GPS data is smoother than the other two since the dynamics do not take into account any vehicle acceleration or deceleration. On the other hand, the macroscopic \emph{E-exp-formula} \eqref{eq:emissioni2} well catches the behavior of emission rates but overestimates the absolute values.

\begin{figure}[h!]
\centering
\subfloat[][Density $\nn = 41$]{
\includegraphics[width=0.3\columnwidth]{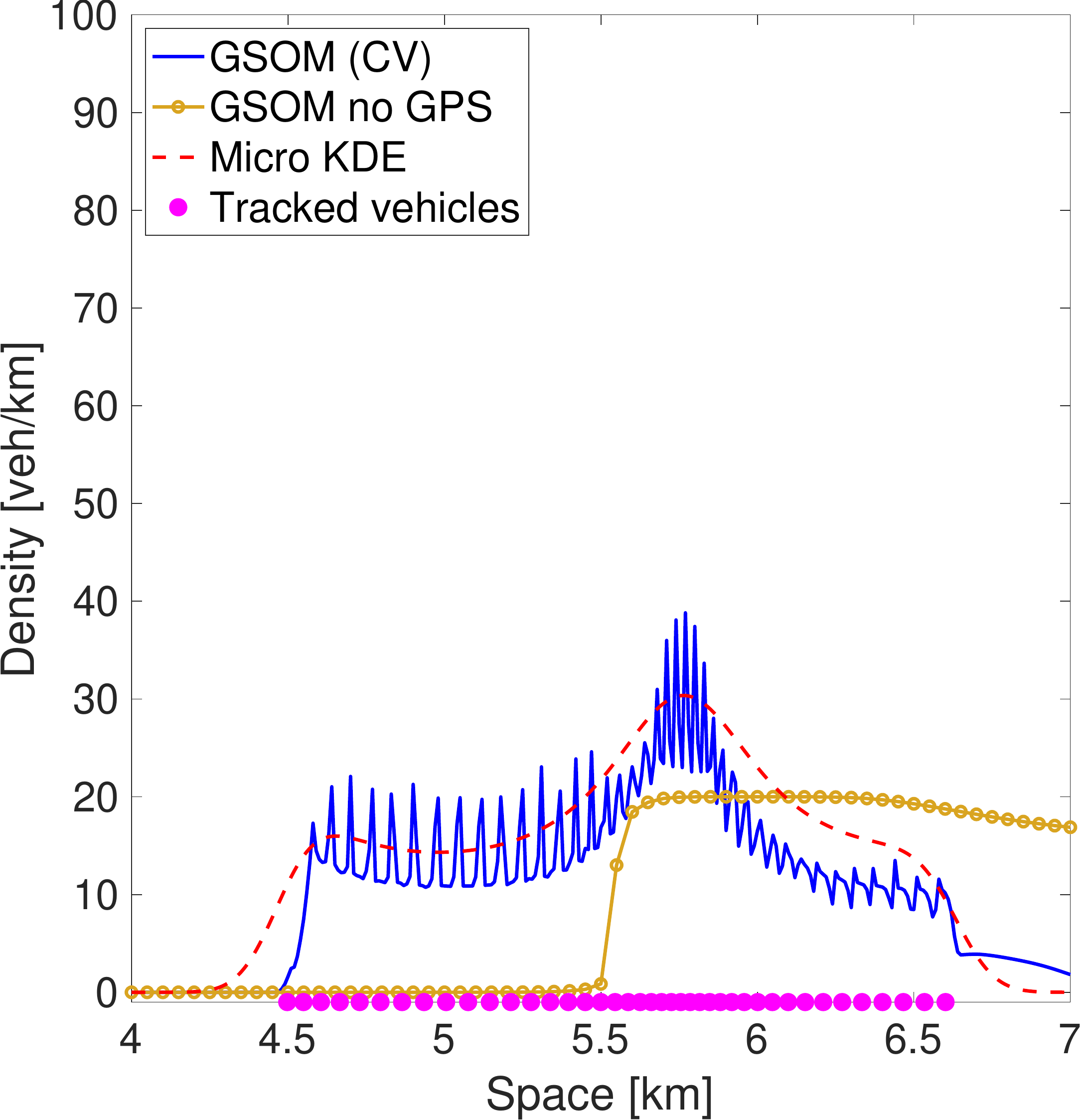}
}\,
\subfloat[][Density $\nn = 20$]{
\includegraphics[width=0.3\columnwidth]{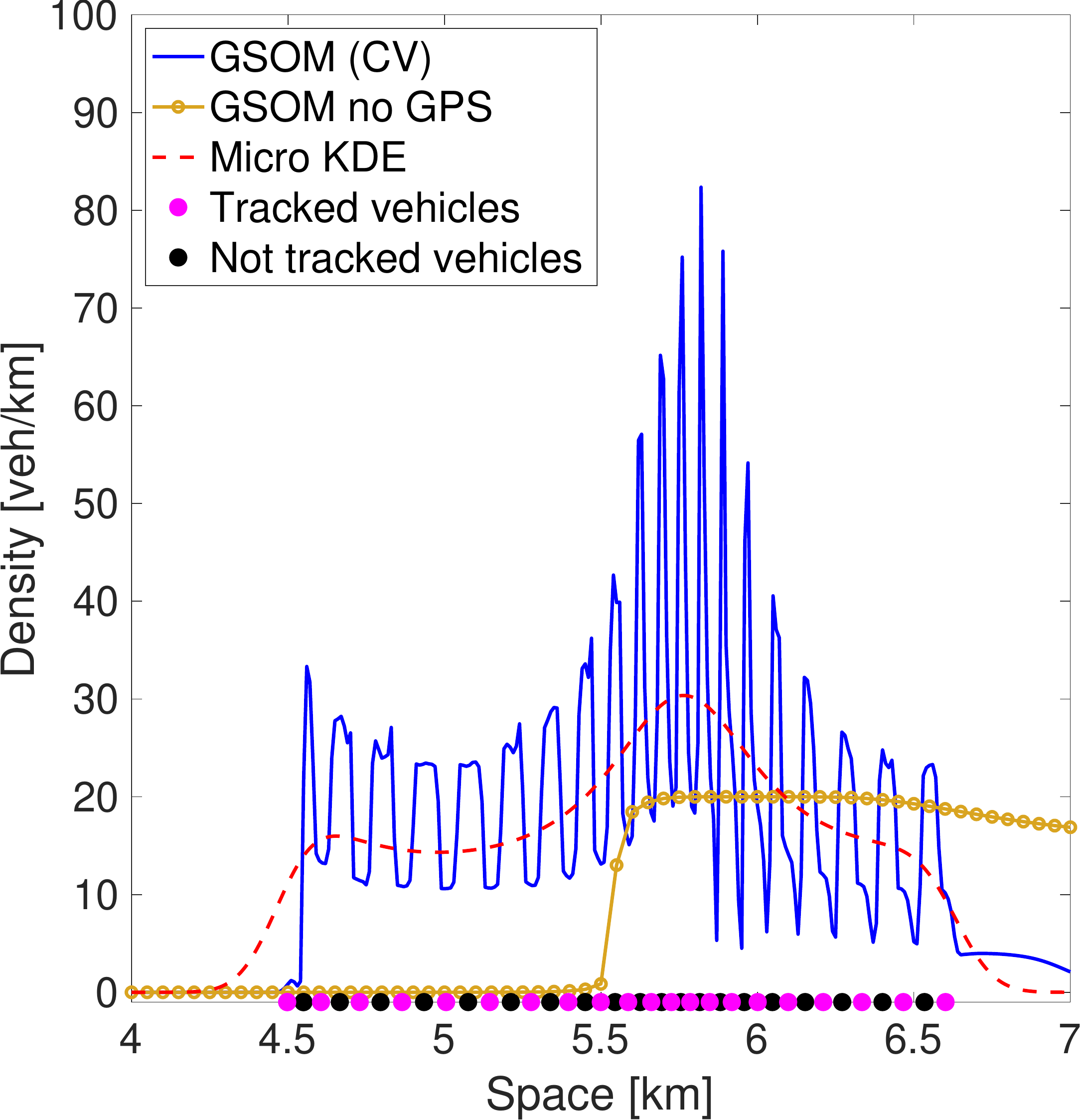}
}\,
\subfloat[][Density $\nn = 10$]{
\includegraphics[width=0.3\columnwidth]{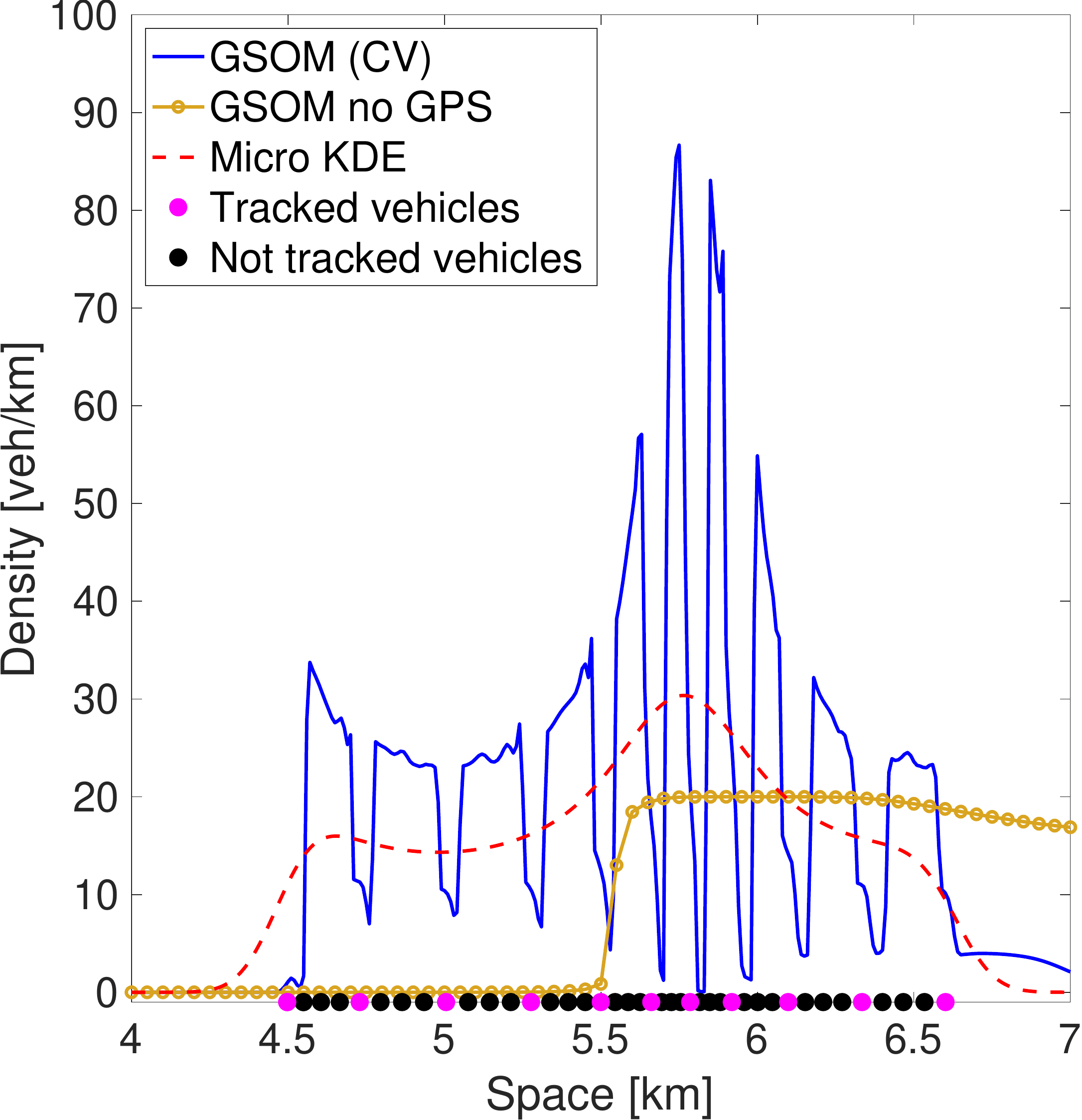}
}\\
\subfloat[][\emph{E-max-f} $\nn = 41$]{
\includegraphics[width=0.3\columnwidth]{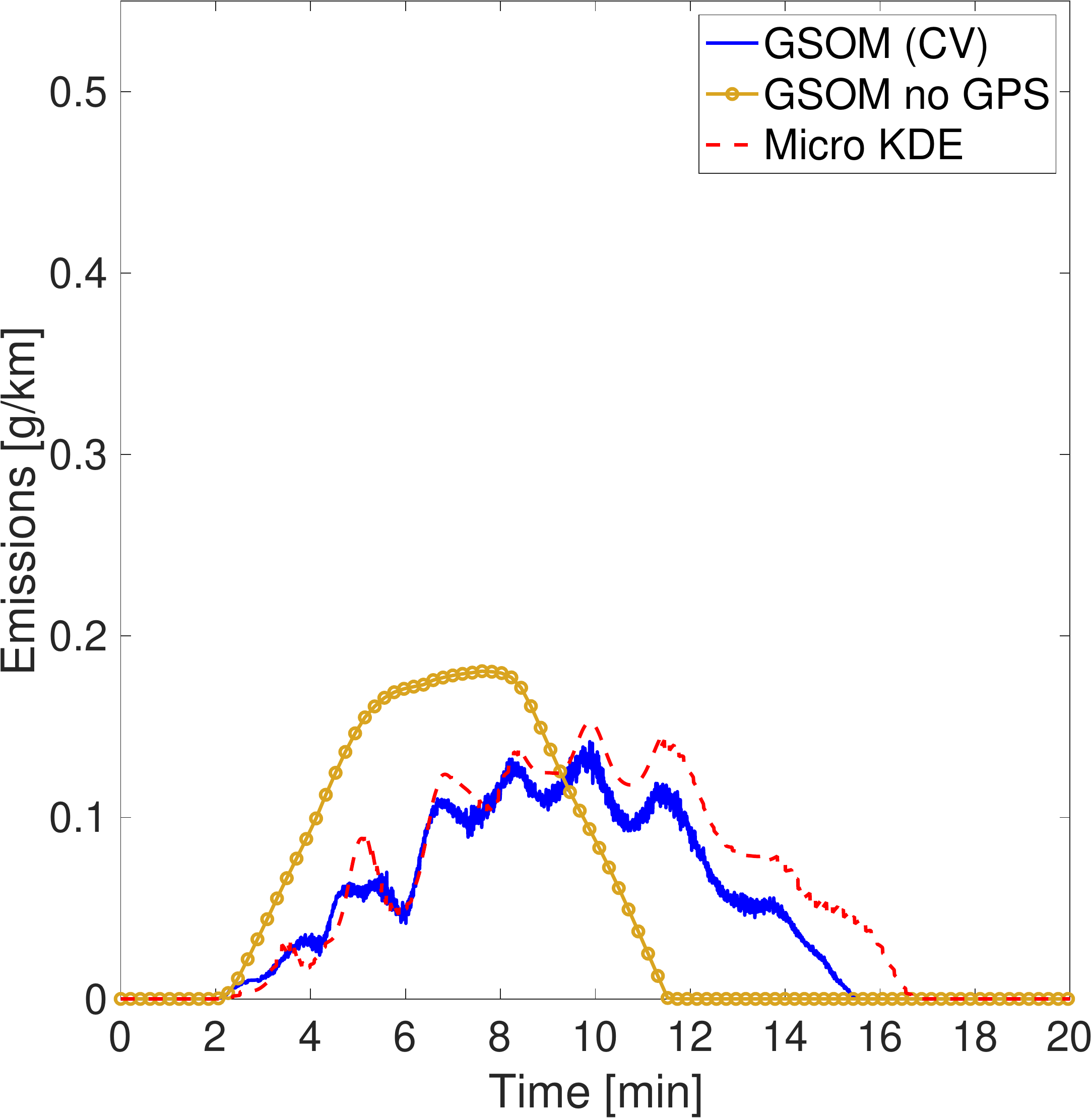}
}\,
\subfloat[][\emph{E-max-f} $\nn = 20$]{
\includegraphics[width=0.3\columnwidth]{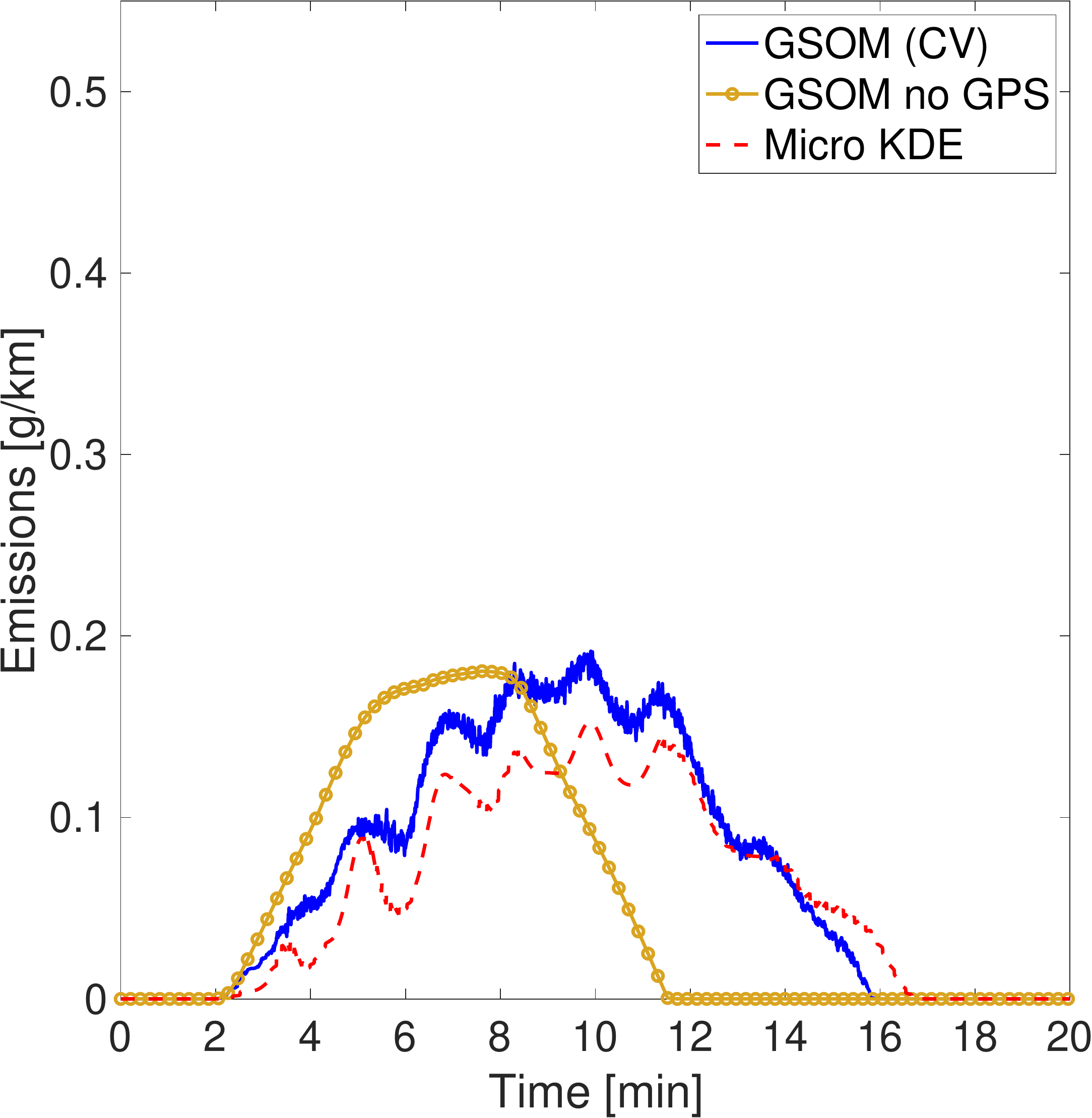}
}\,
\subfloat[][\emph{E-max-f} $\nn = 10$]{
\includegraphics[width=0.3\columnwidth]{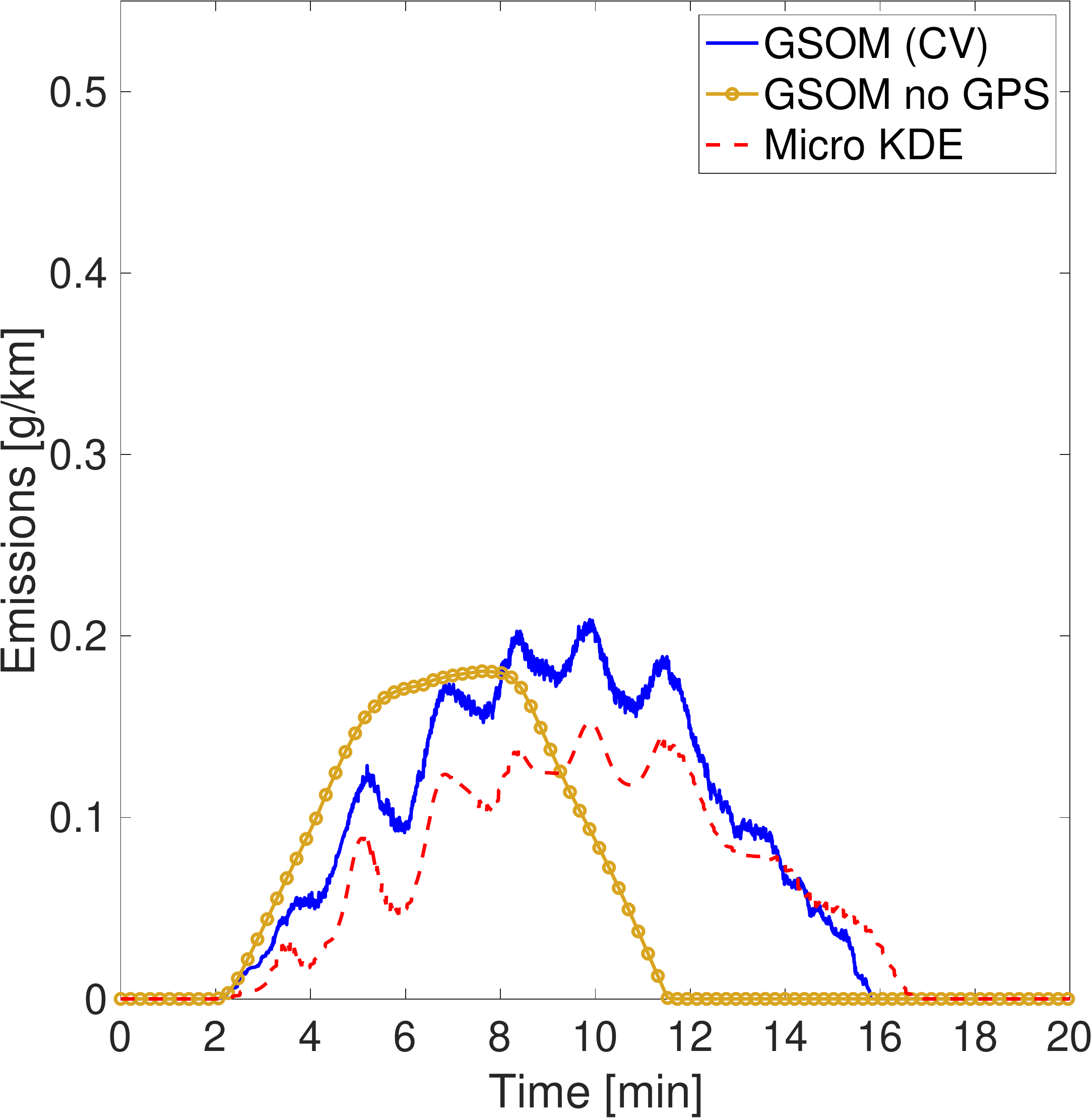}
}\\
\subfloat[][\emph{E-exp-f} $\nn = 41$]{
\includegraphics[width=0.3\columnwidth]{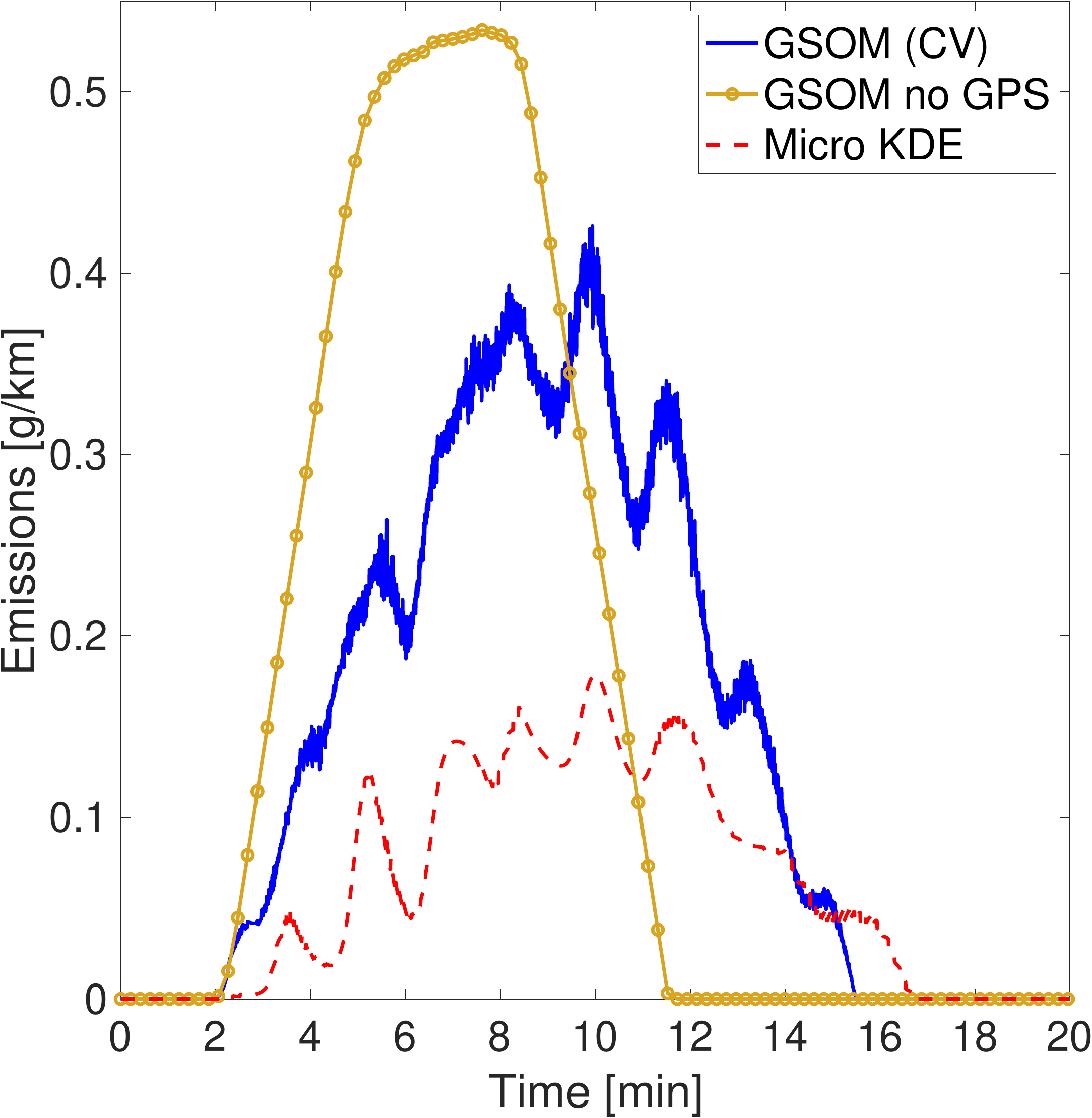}
}\,
\subfloat[][\emph{E-exp-f} $\nn = 20$]{
\includegraphics[width=0.3\columnwidth]{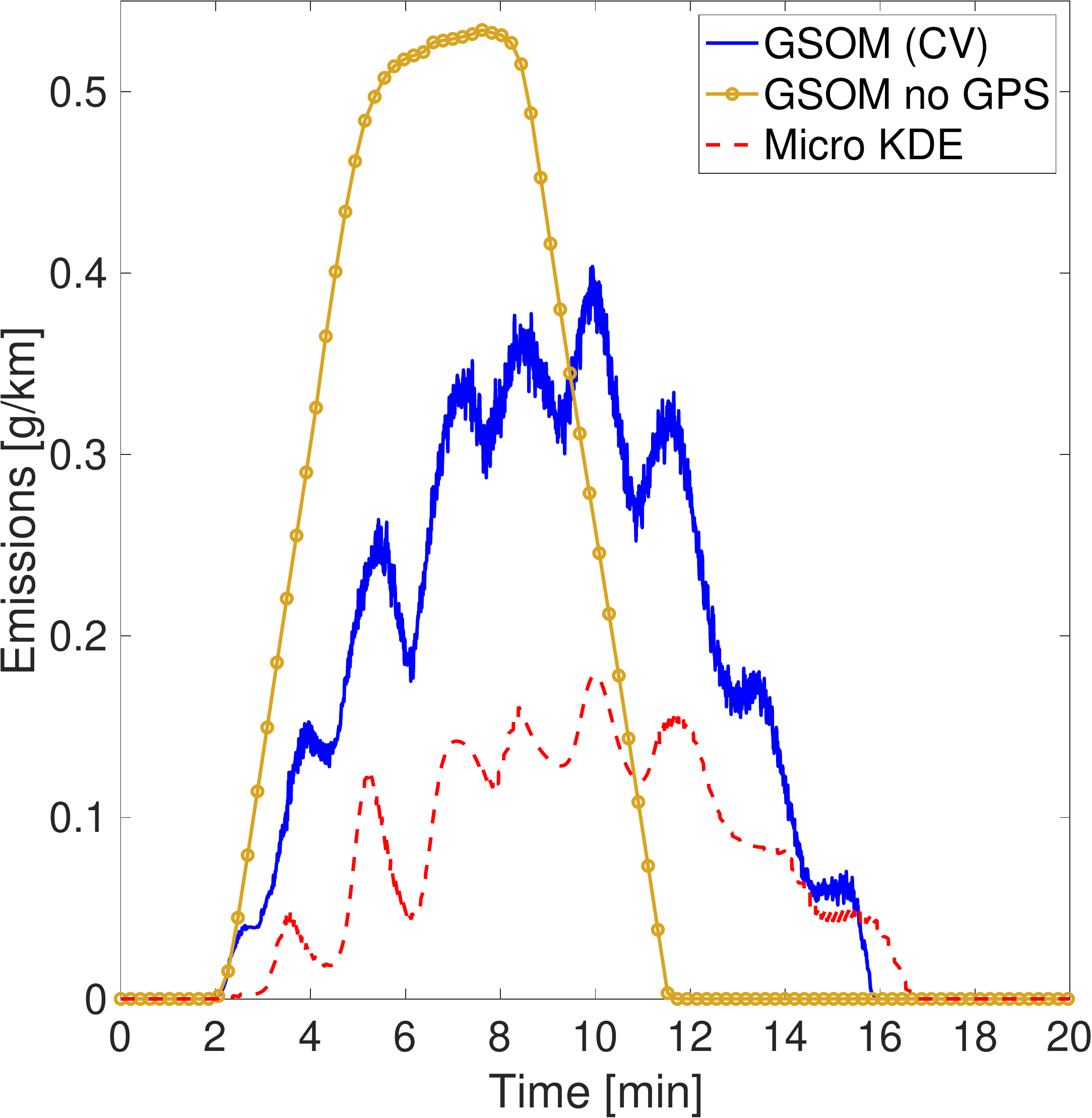}
}\,
\subfloat[][\emph{E-exp-f} $\nn = 10$]{
\includegraphics[width=0.3\columnwidth]{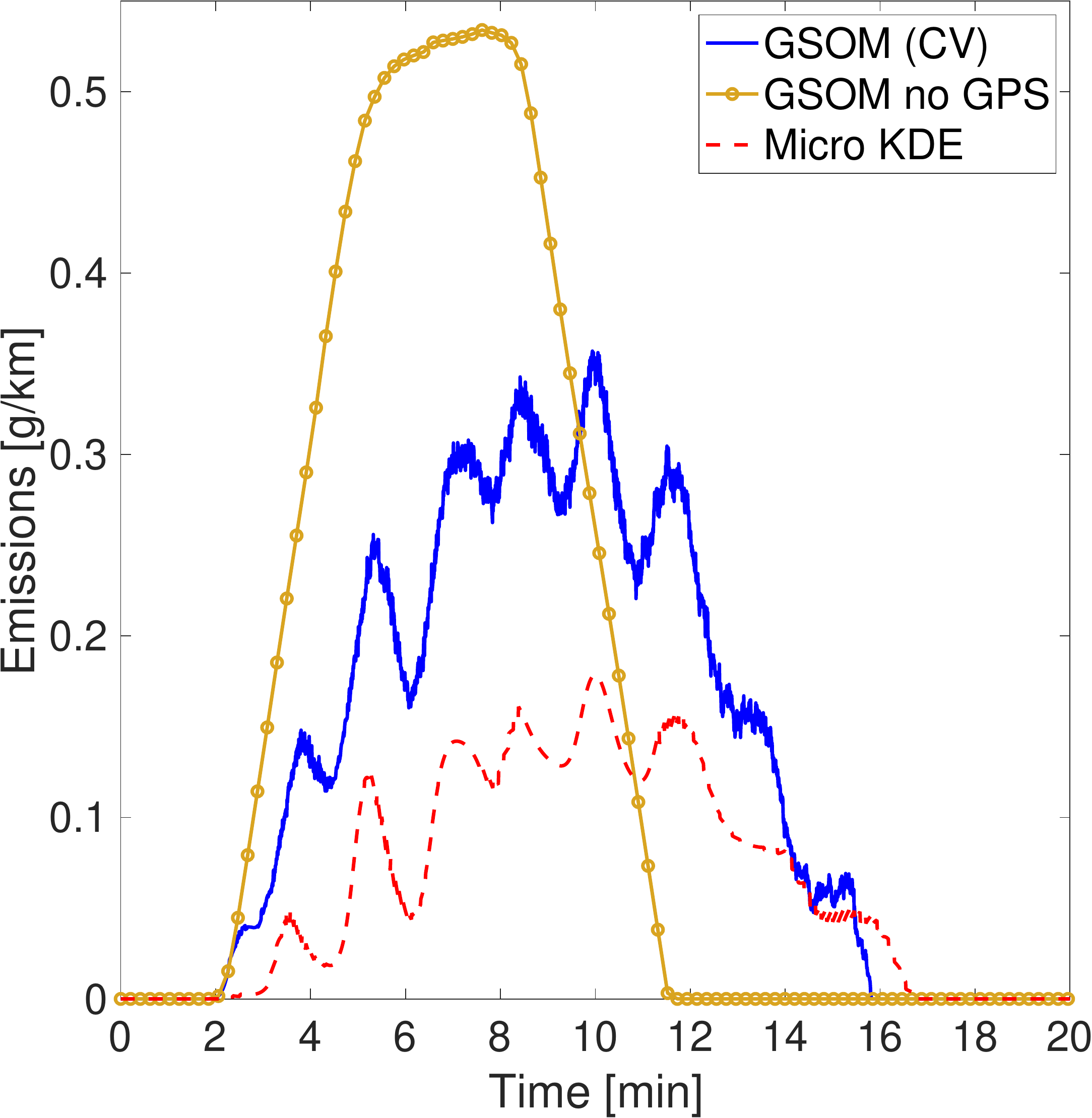}
}
\caption{Section \ref{sec:emissioni} tests. Density profile at $t=10\,\mymin$ (top)  
and total amount of emissions associated with traffic dynamics using the \emph{E-max-formula} \eqref{eq:emissioni} (center) and the \emph{E-exp-formula} \eqref{eq:emissioni2} (bottom) for different $\nn$.}
\label{fig:testEmissioni}
\end{figure}

Finally, in Table \ref{tab:confrontoEm} we estimate the absolute error  in $L^{1}$ norm at the final time of simulation, $\norm{E^{\nt}_{\tot}-e^{\nt}_{\tot}}_{1}$ and $\norm{\widehat E^{\nt}_{\tot}-\widehat e^{\nt}_{\tot}}_{1}$, with and without real trajectory data for $\nn=N,N/2,N/4$. Note that the error associated with the model without real data is not influenced by $\nn$, thus we report a unique value in the third and last column of the table.
For the \emph{E-max-formula}, the error slightly grows with the decrease of $\nn$ and we gain an order of magnitude using the model with real data. For the \emph{E-exp-formula} the error seems not to be affected by the number of tracked vehicles and is higher than the one obtained with the first emission formula. 

\begin{table}[h!]
\centering
\small
\begin{tabular}{ccccc}\toprule
\multirow{2}{*}{$\nn$} & \multicolumn{2}{c}{$\norm{E^{\nt}_{\tot}-e^{\nt}_{\tot}}_{1}$} & \multicolumn{2}{c}{$\norm{\widehat E^{\nt}_{\tot}-\widehat e^{\nt}_{\tot}}_{1}$} \smallskip\\
& With GPS & Without GPS & With GPS & Without GPS\\\midrule
41 & 4e-03 & - & 3e-02 & -\\
20 & 6e-03 & 2e-02& 3e-02 & 5e-02\\
10 & 8e-03 & - & 3e-02 & -\\\bottomrule
\end{tabular}
\caption{Absolute error in $L^{1}$ norm between the two proposed macroscopic and microscopic emission models with and without real trajectory data for different $\nn$.}
\label{tab:confrontoEm}
\end{table}

To sum up, we can give a good approximation of emissions even when few real trajectory data are available. Moreover, the \emph{E-max-formula} \eqref{eq:emissioni} allows for a better approximation of macroscopic emissions than the \emph{E-exp-formula} \eqref{eq:emissioni2} compared to the ground-truth emissions obtained using \eqref{eq:emissioniMicro} and \eqref{eq:emissioniMicro2}. Therefore, in the rest of the paper we will use only the \emph{E-max-formula}.

\begin{remark}
The results of the tests described above are not in contrast with those proposed in \cite[Section 3.1]{balzotti2021DCDS}, where the CGARZ model without GPS data is used with the NGSIM dataset \cite{TrafficNGSIM}. Indeed, the latter contains data for more than 5000 vehicles in 500 meters of road during 45 minutes of data recording. Hence, the large amount of real data allows quite accurate approximations of the emissions just using the CGARZ model with initial data and boundary conditions recovered from the real trajectories of vehicles.
\end{remark}

\subsection{A real-life application with fixed sensors and GPS data}\label{sec:autovieData}
To conclude, we propose an application of traffic model \eqref{eq:modello2} using real trajectory and fixed sensors data provided by Autovie Venete S.p.A. on the Italian A4 (Trieste-Venice) highway. 
The latter is a two lane motorway network travelled by light and heavy vehicles (cars and trucks). In Figure \ref{fig:rete} we draw a sketch of the network with three diverge and three merge junctions connecting six roads. The triangles represent the fixed sensors which record the flux of vehicles that enter the network.
Since we have access to more real truck trajectories than car ones, here we focus only on the dynamic of heavy vehicles and the goal is to apply the methodology described above to estimate the $\nox$ emission rates in a real-life scenario. The model can be used in the same way with real trajectory data of cars or other vehicle classes.
As already discussed at the end of Section \ref{sec:emissioni}, we use the \emph{E-max-formula} \eqref{eq:emissioni} for a more precise approximation of macroscopic emission rates. We use the coefficients $f_{1}-f_{6}$ in \eqref{eq:emissioni} calibrated for trucks, see \cite[Table 2]{panis2006elsevier}.

\begin{figure}[h!]
\centering
\begin{tikzpicture}[scale=0.75]
\draw[thick] (-1,0) -- node[below, xshift=0.3cm]{1} (3,0) node[currarrow,pos=0.6,xscale=0.7, sloped] {}; 
\draw[thick] (3,0) -- node[below, xshift=0.7cm]{2} (7,0) node[currarrow,pos=0.7,xscale=0.7, sloped] {};  
\draw[thick] (-1,1) -- node[above, xshift=-0.2cm]{4} (3,1) node[currarrow,pos=0.4,xscale=-0.7, sloped] {}; 
\draw[thick] (3,1) -- node[above, xshift=0.6cm]{3} (7,1) node[currarrow,pos=0.6,xscale=-0.7, sloped] {};  
\draw[thick] (2,5) -- node[left, yshift=1.1cm]{5} (2,1) node[currarrow,pos=0.3,xscale=0.7, sloped] {}; 
\draw[thick] (3,0) -- node[right, yshift=1.1cm]{6} (3,5) node[currarrow,pos=0.75,xscale=0.7, sloped] {};  
\draw[thick] (2,2) -- (4,0) node[currarrow,pos=0.2,xscale=0.7, sloped] {}; 
\draw[thick] (2,2) -- (2,1) node[currarrow,pos=0.5,xscale=0.7, sloped] {}; 
\filldraw[Red] (2,2) circle (1.5pt) node[left]{D3};

\draw[thick] (4,1) -- (3,2) node[currarrow,pos=0.5,xscale=-0.7, sloped] {}; 
\draw[thick] (4,1) -- (3,1) node[currarrow,pos=0.5,xscale=-0.7, sloped] {}; 
\filldraw[Red] (4,1) circle (1.5pt) node[below]{D2};

\draw[thick] (3,0) -- (3,1) node[currarrow,pos=0.5,xscale=0.7, sloped] {}; 
\draw[thick] (3,0) -- (4,0)node[currarrow,pos=0.5,xscale=0.7, sloped] {}; 
\filldraw[Red] (3,0) circle (1.5pt) node[below]{D1};

\filldraw[ForestGreen] (2,1) circle (1.5pt) node[below]{M2};
\filldraw[ForestGreen] (3,2) circle (1.5pt) node[right]{M3};
\filldraw[ForestGreen] (4,0) circle (1.5pt) node[below]{M1};

\node[isosceles triangle,isosceles triangle apex angle=60,draw,fill=orange!90,minimum size = 0.05cm, rotate=90] at (-1.2,-0.1){};
\node[isosceles triangle,isosceles triangle apex angle=60,draw,fill=orange!90,minimum size = 0.05cm, rotate=90] at (7.2,0.9){};
\node[isosceles triangle,isosceles triangle apex angle=60,draw,fill=orange!90,minimum size = 0.05cm, rotate=90] at (2,5.2){};

\end{tikzpicture}
\caption{Section \ref{sec:autovieData} test. Sketch of the highway network, where the roads are numbered from 1 to 6, the triangles represent the fixed sensors, the diverge junctions are represented by points D1, D2, D3 and the merge ones by points M1, M2, M3.}
\label{fig:rete}
\end{figure}
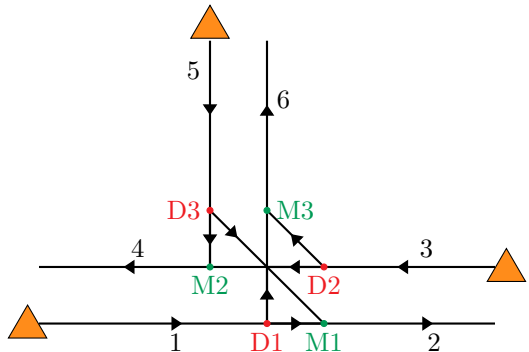

First of all we describe the main features of the network. We have three incoming roads (1, 3 and 5) and three outgoing roads (2, 4 and 6). The dynamic of vehicles along each road is described by the model \eqref{eq:modello2}. We use the notation $\rho_{r,j}^{n}$ to identify the density on road $r$ in cell $x_{j}$ at time $t^{n}$; similarly for the variable $w$.
A special treatment is required at junctions, which are divided into two types: the diverge junctions, that divide a road in two, and the merge ones, which join two roads into one. Diverge junctions are ruled by a distribution coefficient that specifies the percentage of vehicles which prefer one road rather than the other one. For instance, junction $\mathrm{D1}$ has a certain distribution parameter $\alpha_{\mathrm{D1}}$ defining the percentage of vehicles which continue their path from road 1 to road 2, while $1-\alpha_{\mathrm{D1}}$ represents vehicles which continue on road 6. 
Similarly for the other diverge junctions. Following \cite{balzotti2022AX}, merge junctions are characterized by a priority rule that establishes which of the two roads sends more flow of vehicles. We denote by $\beta_{\mathrm{M1}}$, $\beta_{\mathrm{M2}}$ and $\beta_{\mathrm{M3}}$ these parameters in $[0,1]$, and we calibrate them as a property of the network.

Now we describe the treatment of real data and we begin with data from GPS devices. In Figure \ref{fig:dati} we show an example of real trajectory data related to heavy vehicles on the six roads of the network. This data provides information on the position and velocity of vehicles at various times. The time interval of the recorded positions is not constant and can vary from a few seconds to many minutes. In order to integrate the trajectory data into the numerical scheme, once a vehicle is located into a monitored road, we need to know its position with respect to the time step $\dt$. To this end, we linearly interpolate the given data to reconstruct missing information when necessary. 
Furthermore, this data is used to estimate the distribution parameters $\alpha_{\mathrm{D1}}$, $\alpha_{\mathrm{D2}}$ and $\alpha_{\mathrm{D3}}$ of diverge junctions. Indeed, we use the trajectory data to reconstruct the main paths along the network to compute the percentage of vehicles that at the junction continue towards one road rather than the other one. The analysis of the paths of heavy vehicles during one month of data allowed us to estimate $\alpha_{\mathrm{D1}} = 0.78$, $\alpha_{\mathrm{D2}} = 0.78$ and $\alpha_{\mathrm{D3}} = 0.48$.

\begin{figure}[h!]
\centering
\subfloat[][Road 1]{
\includegraphics[width=0.3\columnwidth]{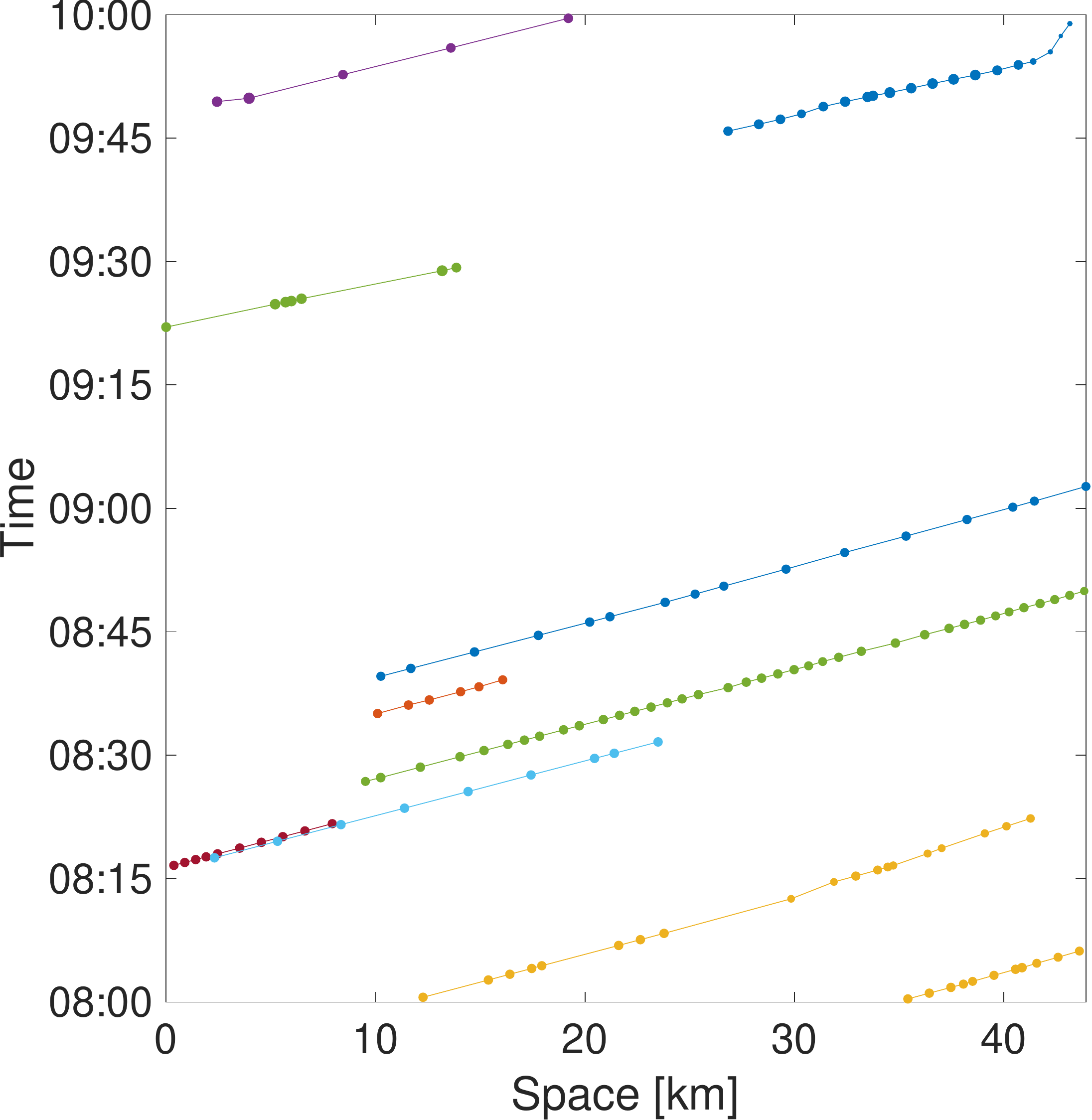}
}
\subfloat[][Road 2]{
\includegraphics[width=0.3\columnwidth]{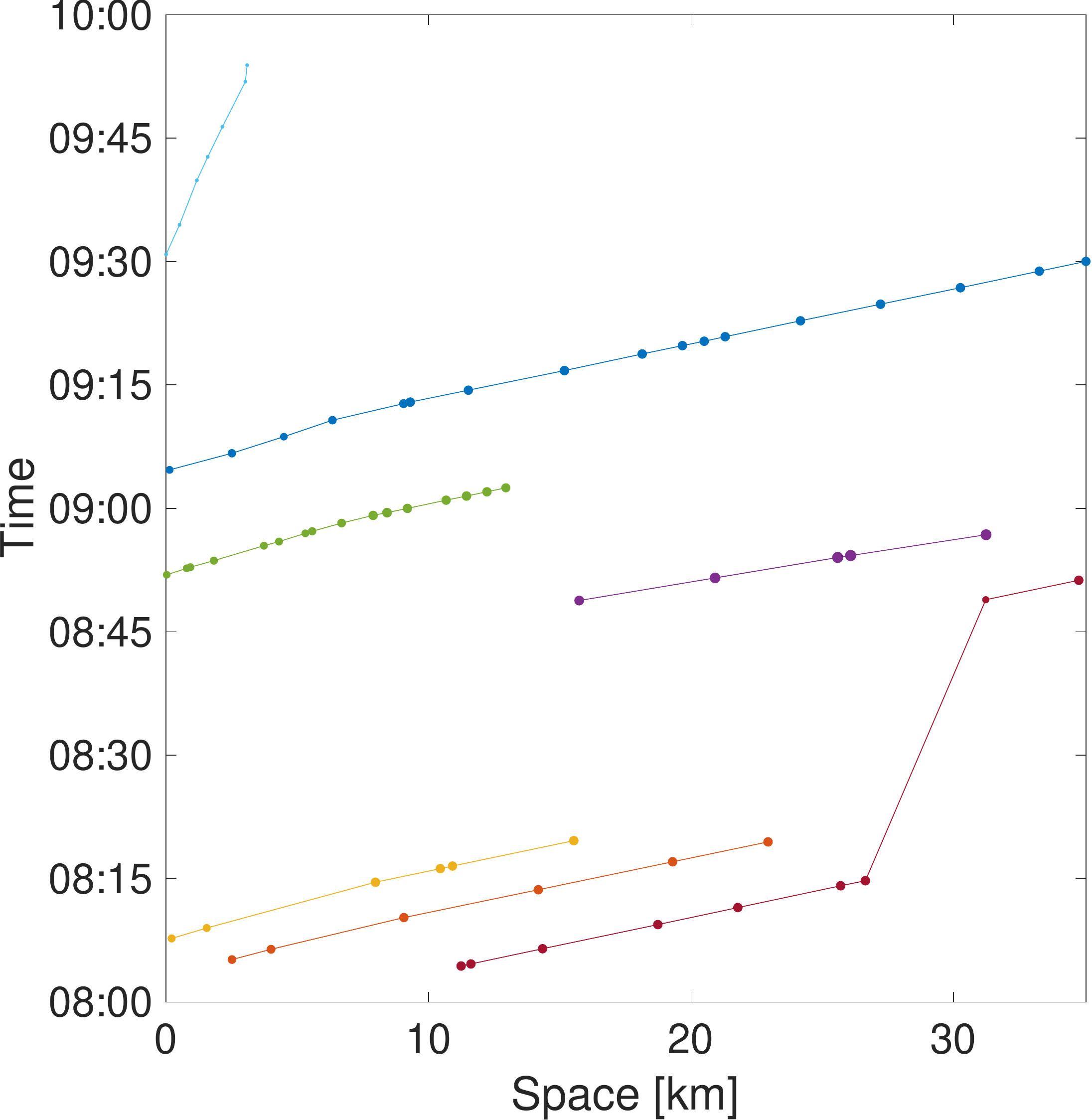}
}
\subfloat[][Road 3]{
\includegraphics[width=0.3\columnwidth]{grafici/camion_27082021_tratta5441.pdf}
}\\
\subfloat[][Road 4]{
\includegraphics[width=0.3\columnwidth]{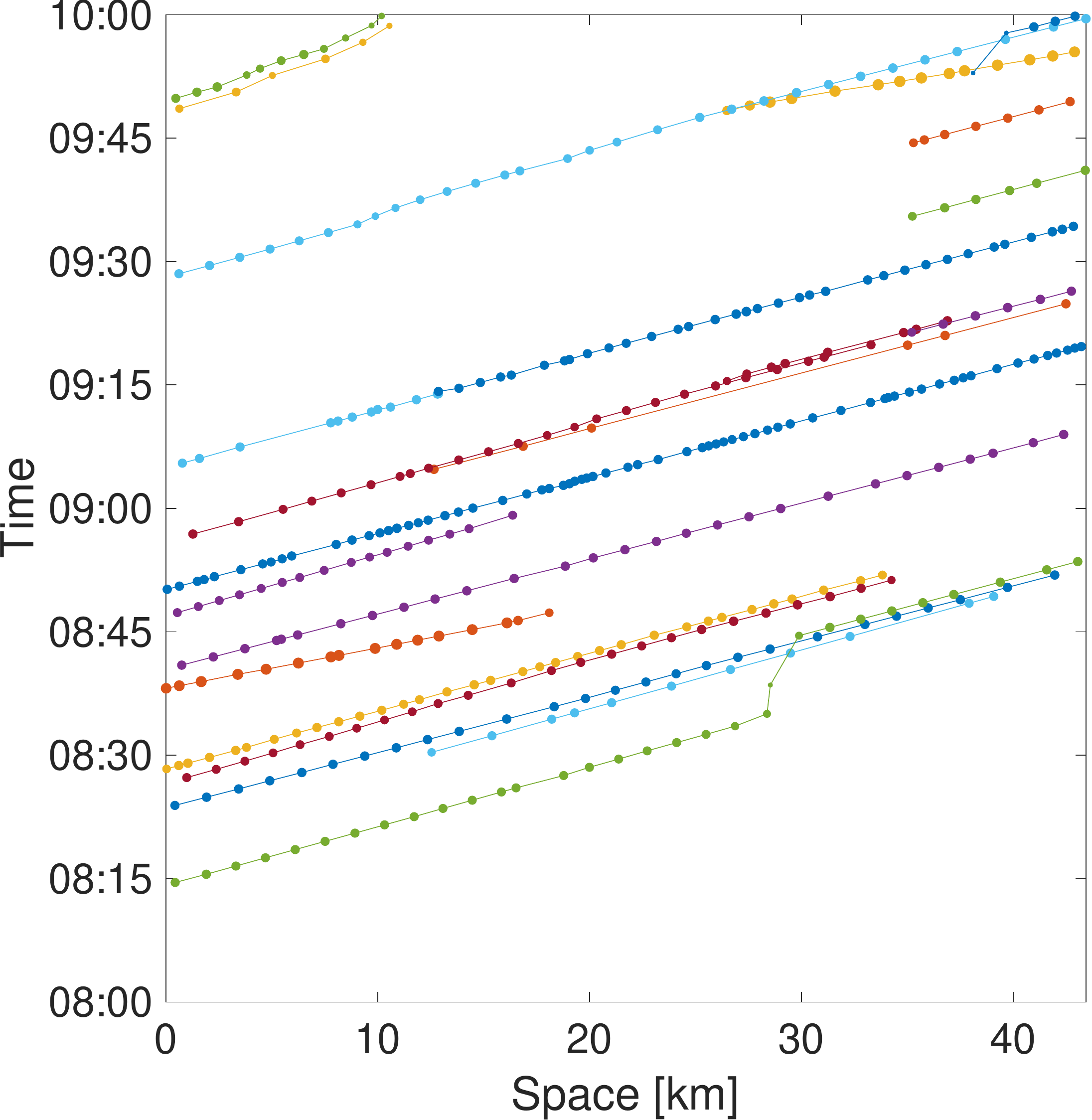}
}
\subfloat[][Road 5]{
\includegraphics[width=0.3\columnwidth]{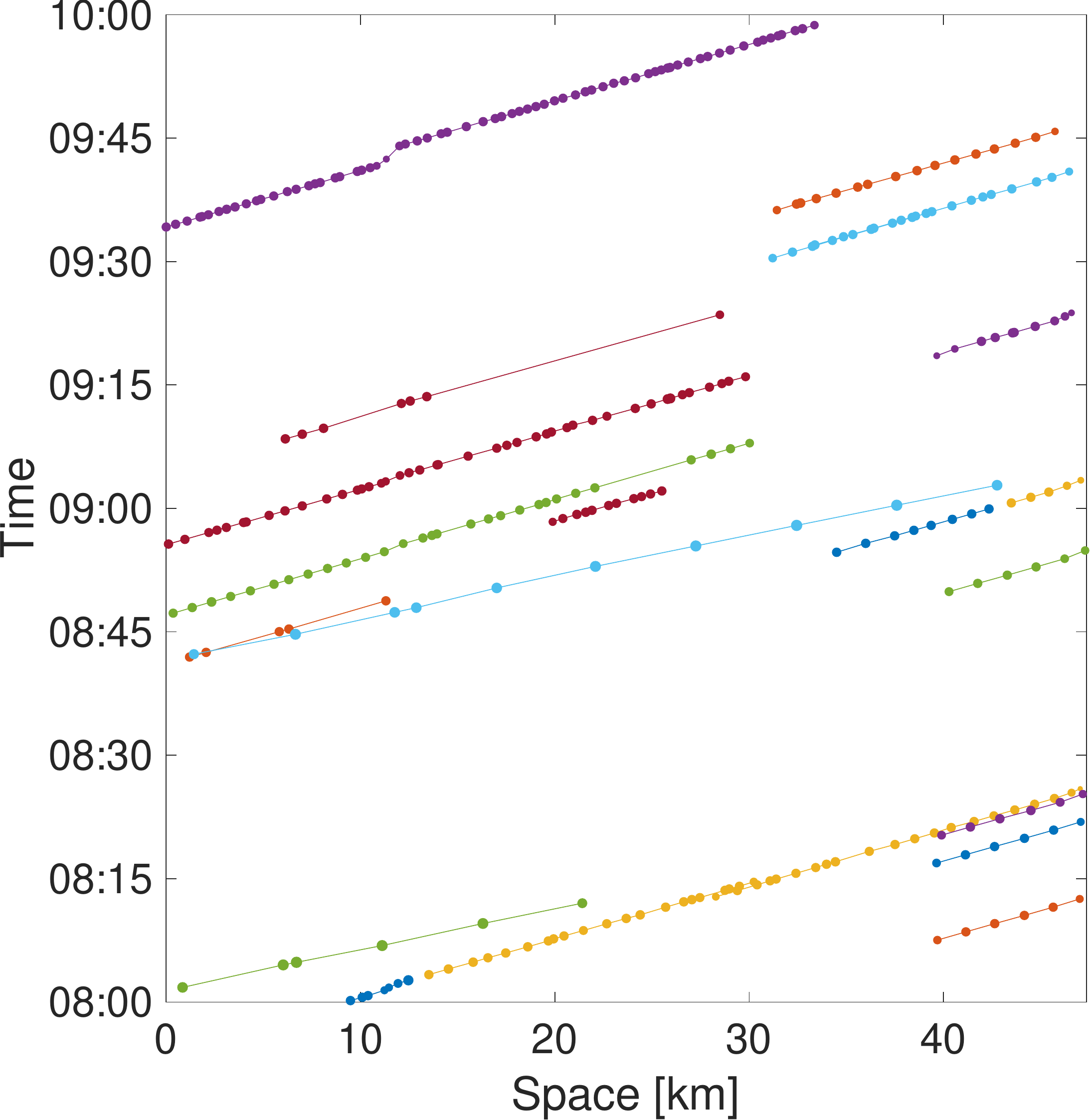}
}
\subfloat[][Road 6]{
\includegraphics[width=0.3\columnwidth]{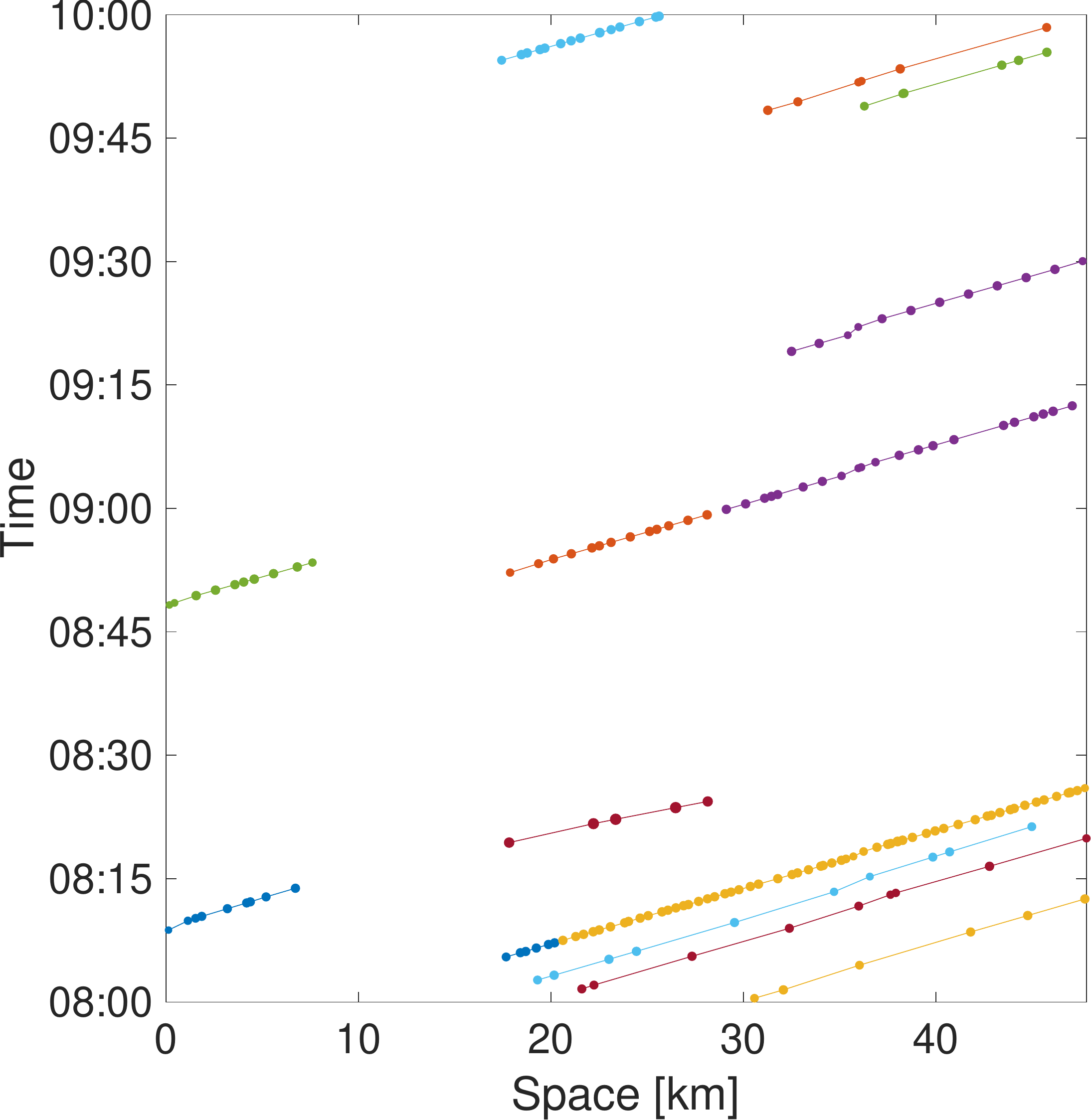}
}
\caption{Section \ref{sec:autovieData} test. Example of real trajectory data recorded on 27/08/2021. The size of the space-time circles is proportional to vehicles velocity. The data were provided by Autovie Venete S.p.A and are not publicly available.}
\label{fig:dati}
\end{figure}

In addition to GPS data, there are fixed sensors along the highway network that record the flow and speed of vehicles crossing it every minute. We denote the sensor data on road $r$ by $\qq_{r,\sens}$ and $\vv_{r,\sens}$. As mentioned ad the end of Section \ref{sec:modello2}, we use sensors data for the boundary conditions on the incoming side of the roads entering the network, i.e.\ roads 1, 3 and 5. Indeed, we modify the incoming numerical flux of scheme \eqref{eq:schema2} in the first cell of these roads. Specifically, since the sensor data is given every minute, we interpolate it in a piecewise constant way, thus we define $\qq^{n}_{r,\sens} = Q_{r,\sens}(\hat t)$, where $\hat t=\lfloor 60n\dt\rfloor$ is the minute corresponding to $t^{n}$. Then we compute the density $\rho^{n}_{r,0}$ by applying \eqref{eq:schema2} with $Q^{n}_{r,-1/2} = \qq^{n}_{r,\sens}$. Once computed the density, we estimate $w^{n}_{r,0}$ such that $v(\rho^{n}_{r,0},w^{n}_{r,0})=\vv^{n}_{r,\sens}$, with $V^{n}_{r,\sens} = \vv_{r,\sens}(\lfloor 60 n\dt\rfloor)$ velocity value captured by the sensor. In Figure \ref{fig:sensori} we show the variation in time of the flux of heavy vehicles for the three sensors.

\begin{figure}[h!]
\centering
\subfloat[][Sensor on road 1]{
\includegraphics[width=0.3\columnwidth]{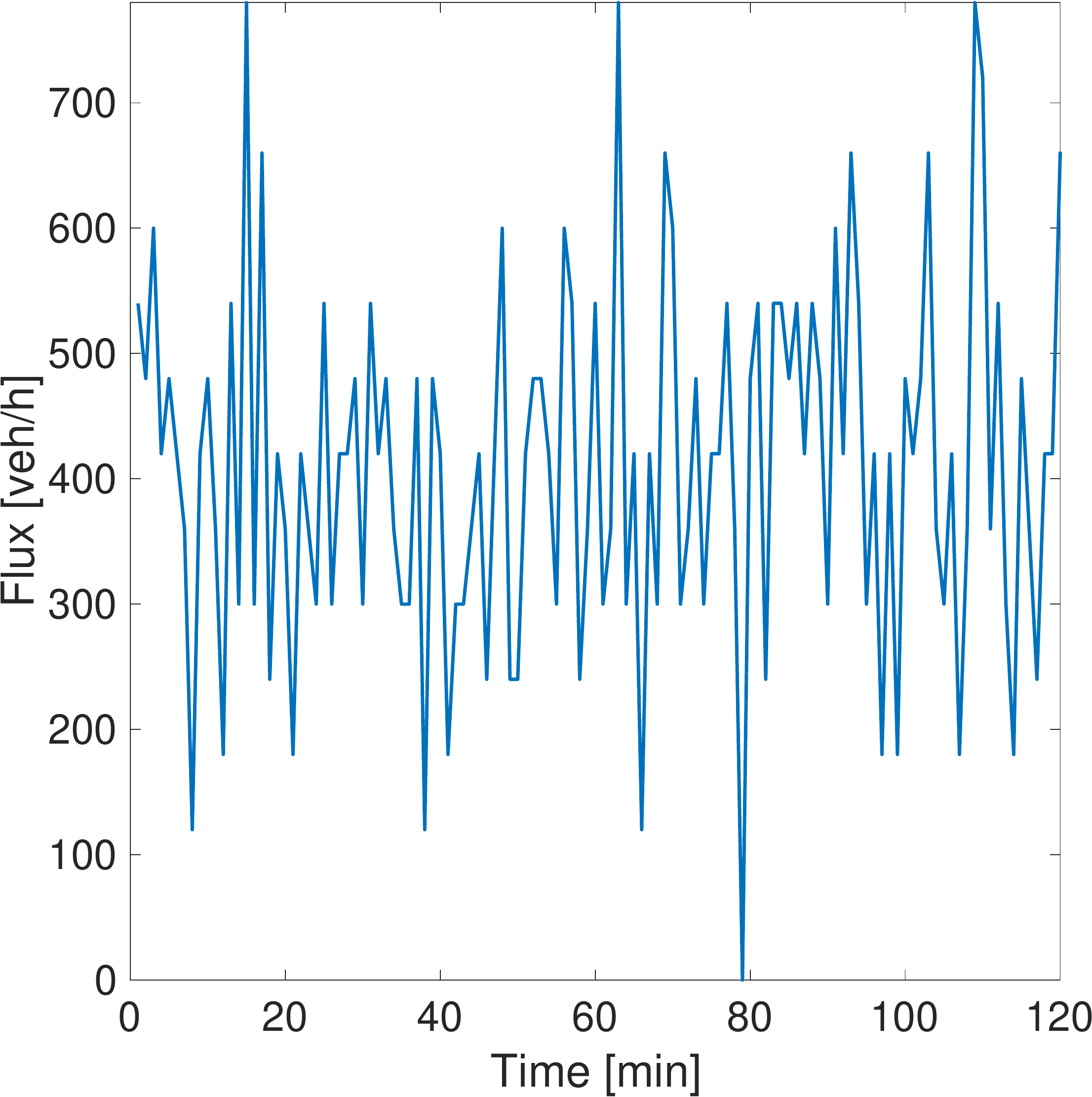}
}
\subfloat[][Sensor on road 3]{
\includegraphics[width=0.3\columnwidth]{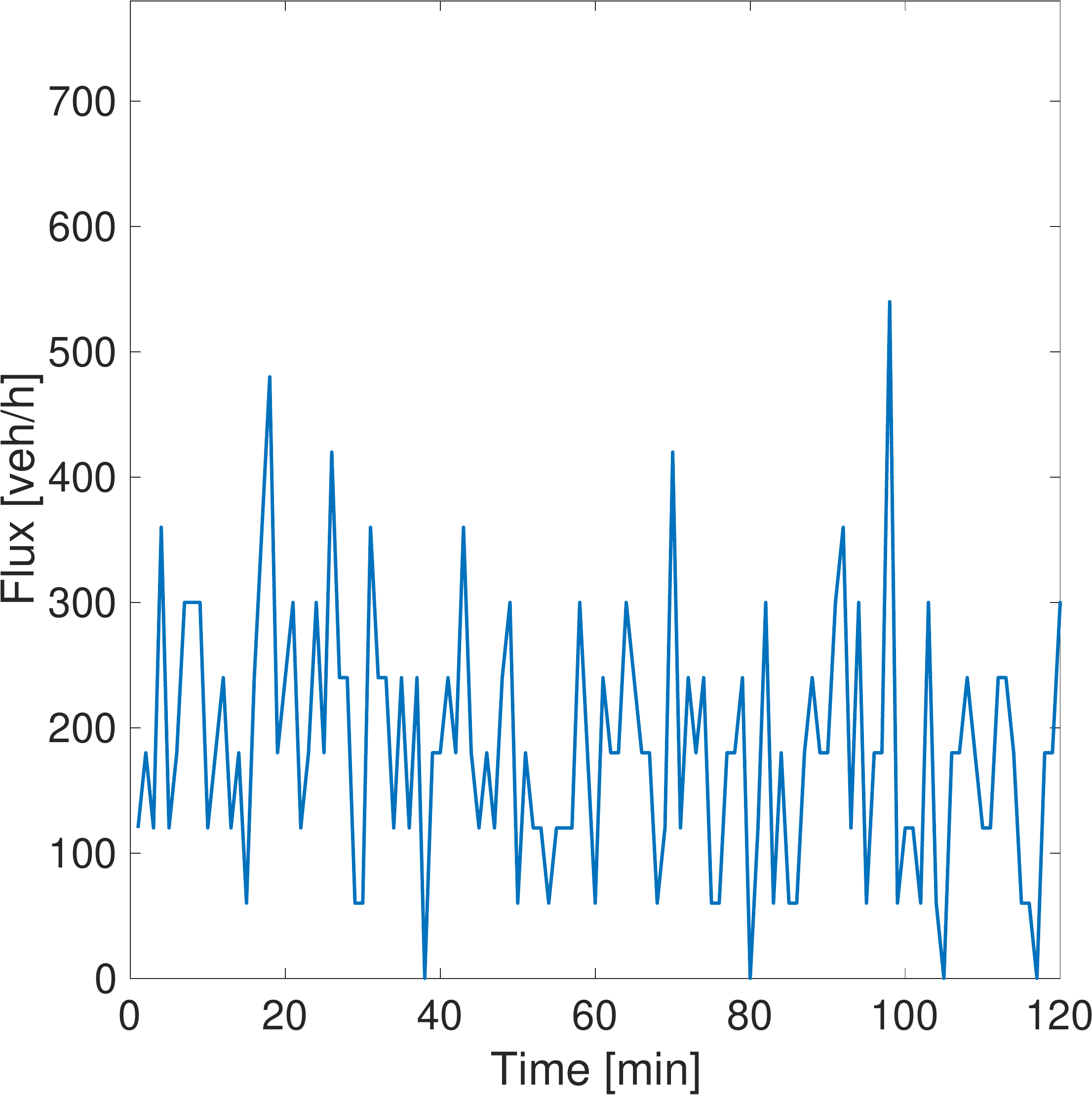}
}
\subfloat[][Sensor on road 5]{
\includegraphics[width=0.3\columnwidth]{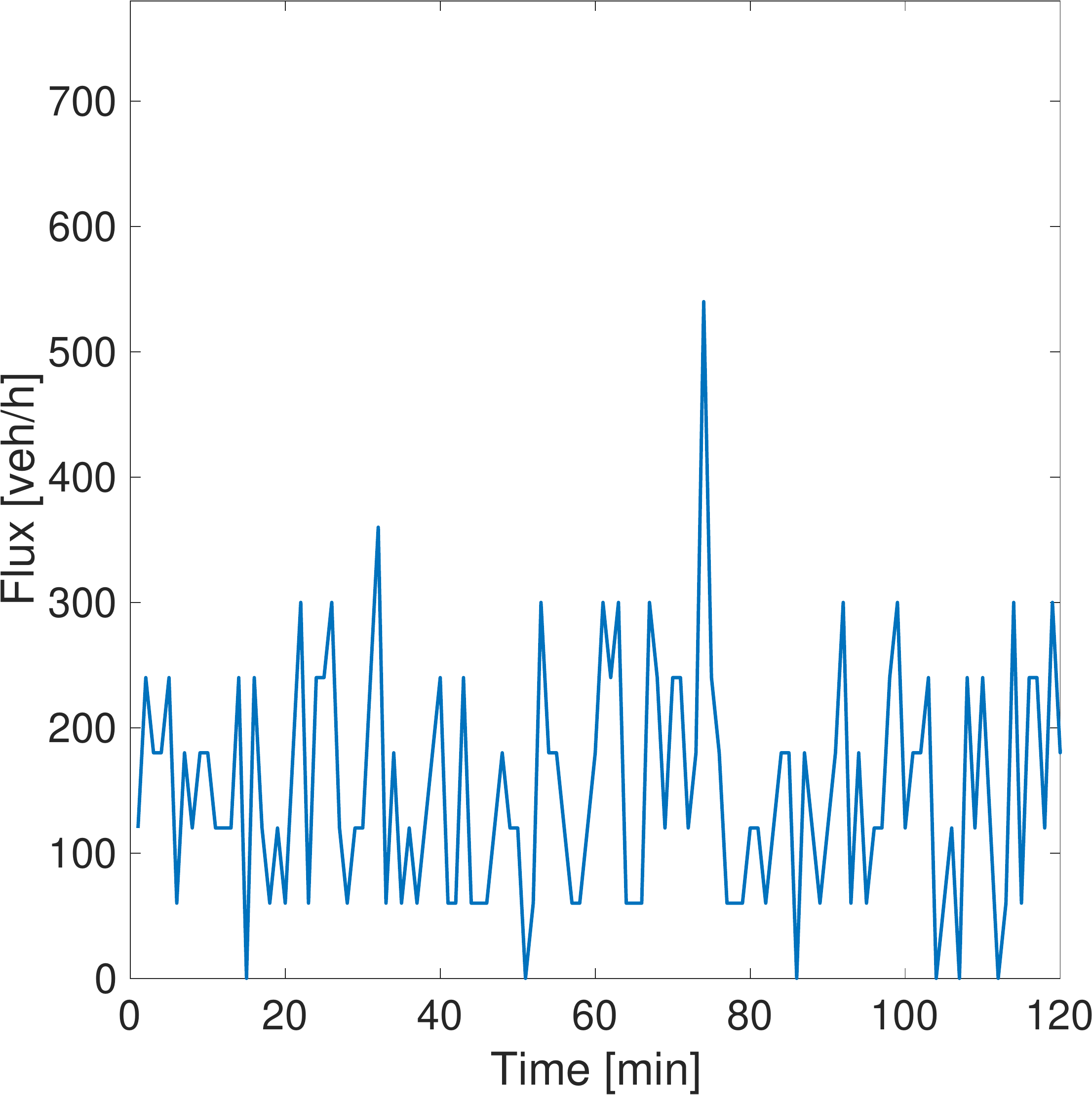}
}
\caption{Section \ref{sec:autovieData} test. Variation in time of the flux per minute of heavy vehicles recorded by the three sensors on 27/08/2021. The data were provided by Autovie Venete S.p.A and are not publicly available.}
\label{fig:sensori}
\end{figure}

\medskip
To simulate the traffic dynamics we proceed as follows: first, we estimate the initial condition starting from an empty network; then, we approximate the dynamics of vehicles by means of model \eqref{eq:modello2}.
To estimate the initial traffic state, we start from an empty network with $w=(\wl+\wr)/2$. As explained above, the traffic state on the incoming boundary of roads entering the network, $(\rho^{n}_{r,0},w^{n}_{r,0},y^{n}_{r,0})$ for $r\in\{1,3,5\}$, is estimated from sensors data. Moreover we assume homogeneous Neumann conditions on the outgoing side of roads 2, 4 and 6, hence vehicles are free to leave the network. The roads are then filled for half an hour through model \eqref{eq:modello2} with both trajectory and sensors data, using the CTM scheme introduced in Section \ref{sec:Ndata}. 
The traffic state after half an hour of simulation is the initial state for model \eqref{eq:modello2}, and we use it to approximate the dynamics and then estimate vehicular emissions via \eqref{eq:emissioni}.

In our test we consider the network depicted in Figure \ref{fig:rete}. We use again the CGARZ model with the flux function defined at the beginning of Section \ref{sec:test2O}. Let us denote by $L_{r}$ the length of road $r$, $r=1,\dots,6$. According to the properties of the highway, we fix $L_{1}=L_{4}=41\,\km$, $L_{2}=L_{3}=L_{4}=L_{5}=36\,\km$, $\beta_{M1}=\beta_{M3}=0.2$, $\beta_{M2}=0.5$, $\vmax=90\,\kmh$ and $\rhomax=56\,\vehkm$. In particular, by assuming that the length of each heavy vehicle is $18\,\meter$ (length + safety distance), the maximum density $\rhomax$ is computed by dividing the total employable space (two lanes) by the space occupied by the vehicles, thus $\rhomax=56\,\vehkm$. Moreover, we set $\rhof=10\,\vehkm$, $\wl=g(\rhof)=739$ and $\wr=g(\rhomax/2) = 1260$, with $g$ defined in \eqref{eq:g}. 

Once estimated the initial state, we consider a simulation with $T=1.5\,\myhour$ influenced by the trajectory data represented in Figure \ref{fig:dati}. In Figure \ref{fig:testAV} we show the density of vehicles along the network at different times, excluding the six junctions and focusing only on the main roads. 
We observe that the inclusion of real trajectory data allows model \eqref{eq:modello2} to capture the formation of small congestions, mainly on the horizontal roads, which could not be observed otherwise.

\begin{figure}[h!]
\centering
\subfloat[][Density $t=0\,\mymin$]{\label{fig:AVdens1}
\includegraphics[width=0.235\columnwidth]{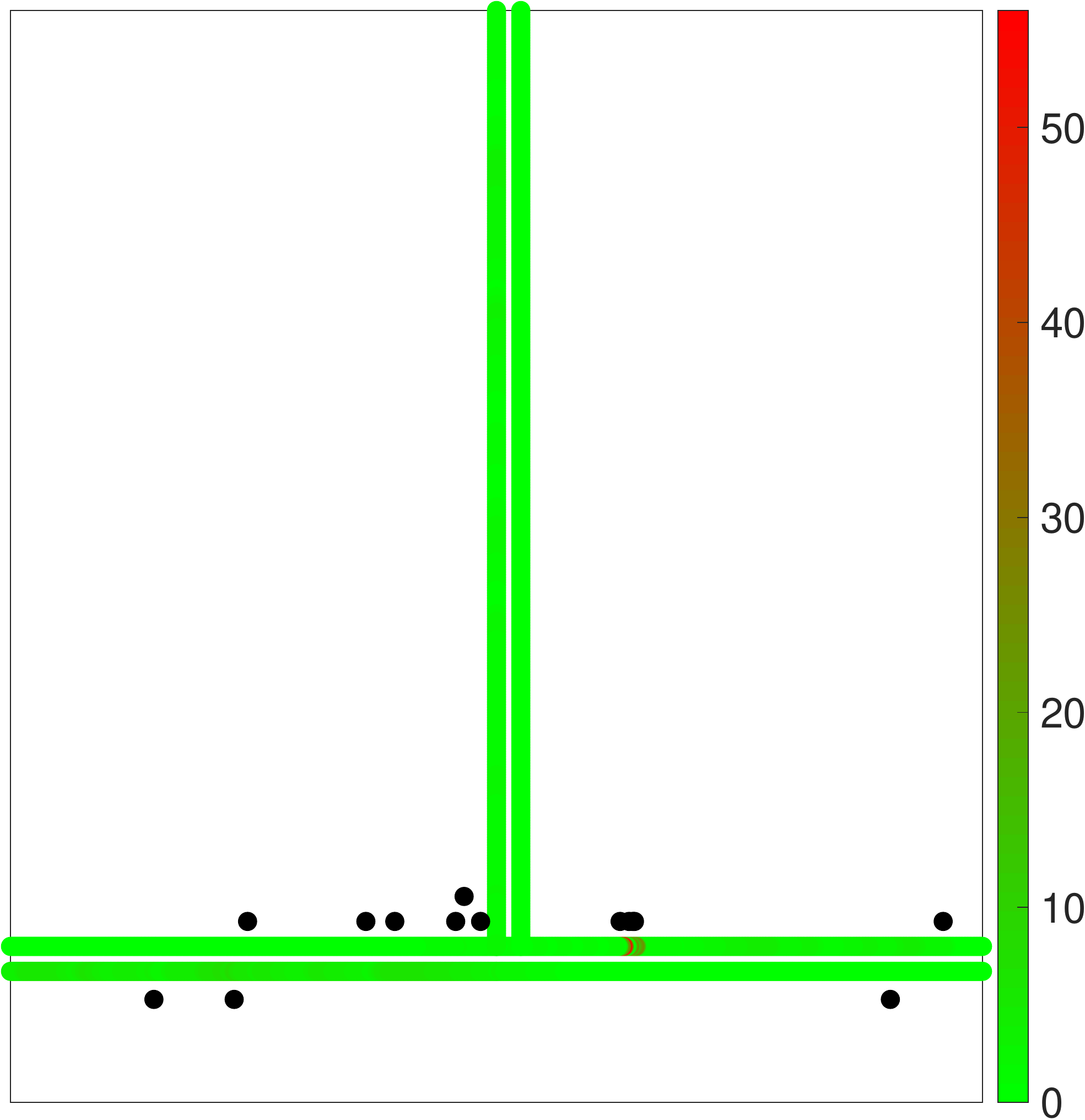}
}
\subfloat[][Density $t=15\,\mymin$]{\label{fig:AVdens2}
\includegraphics[width=0.235\columnwidth]{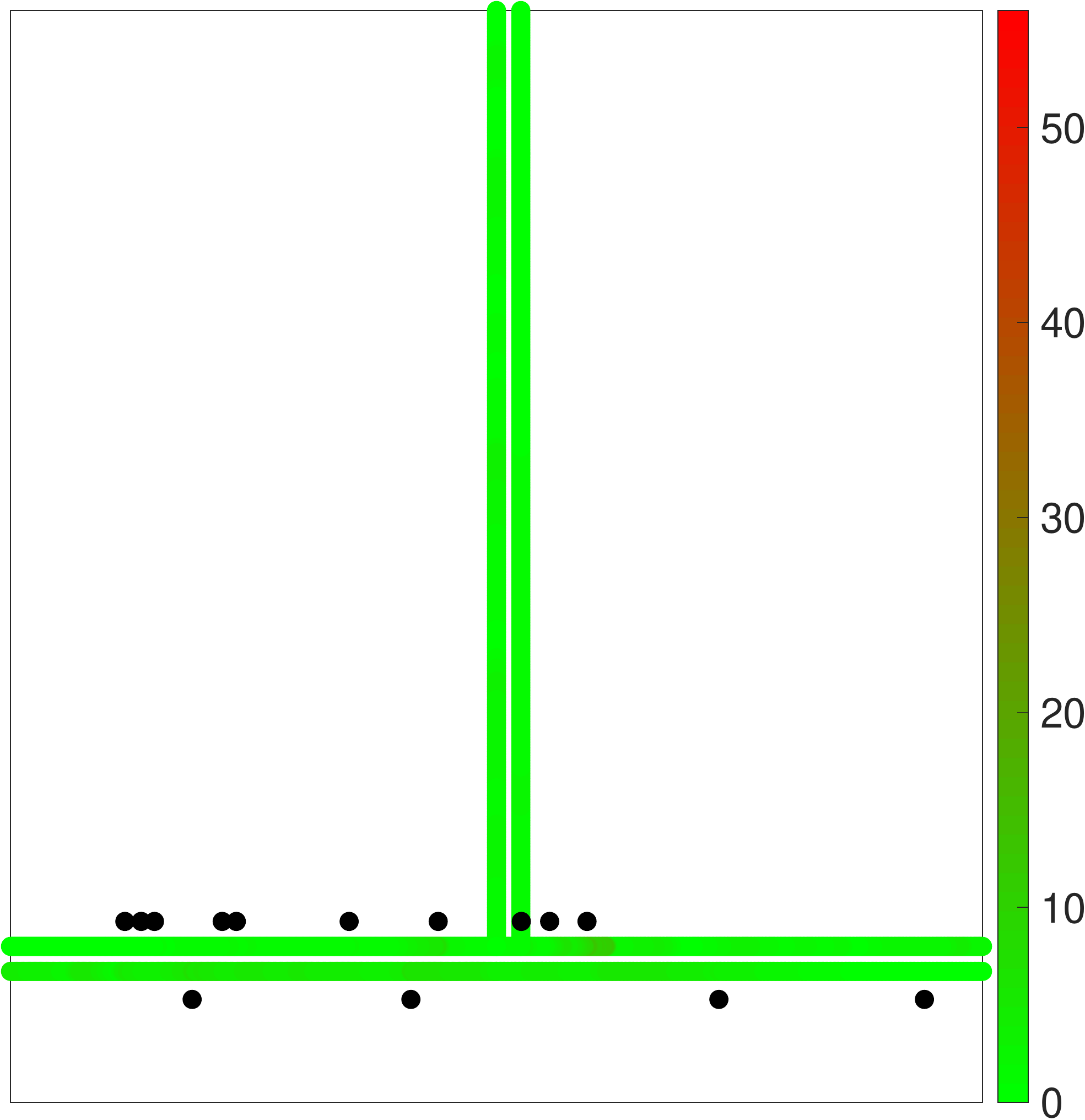}
}
\subfloat[][Density $t=60\,\mymin$]{\label{fig:AVdens3}
\includegraphics[width=0.235\columnwidth]{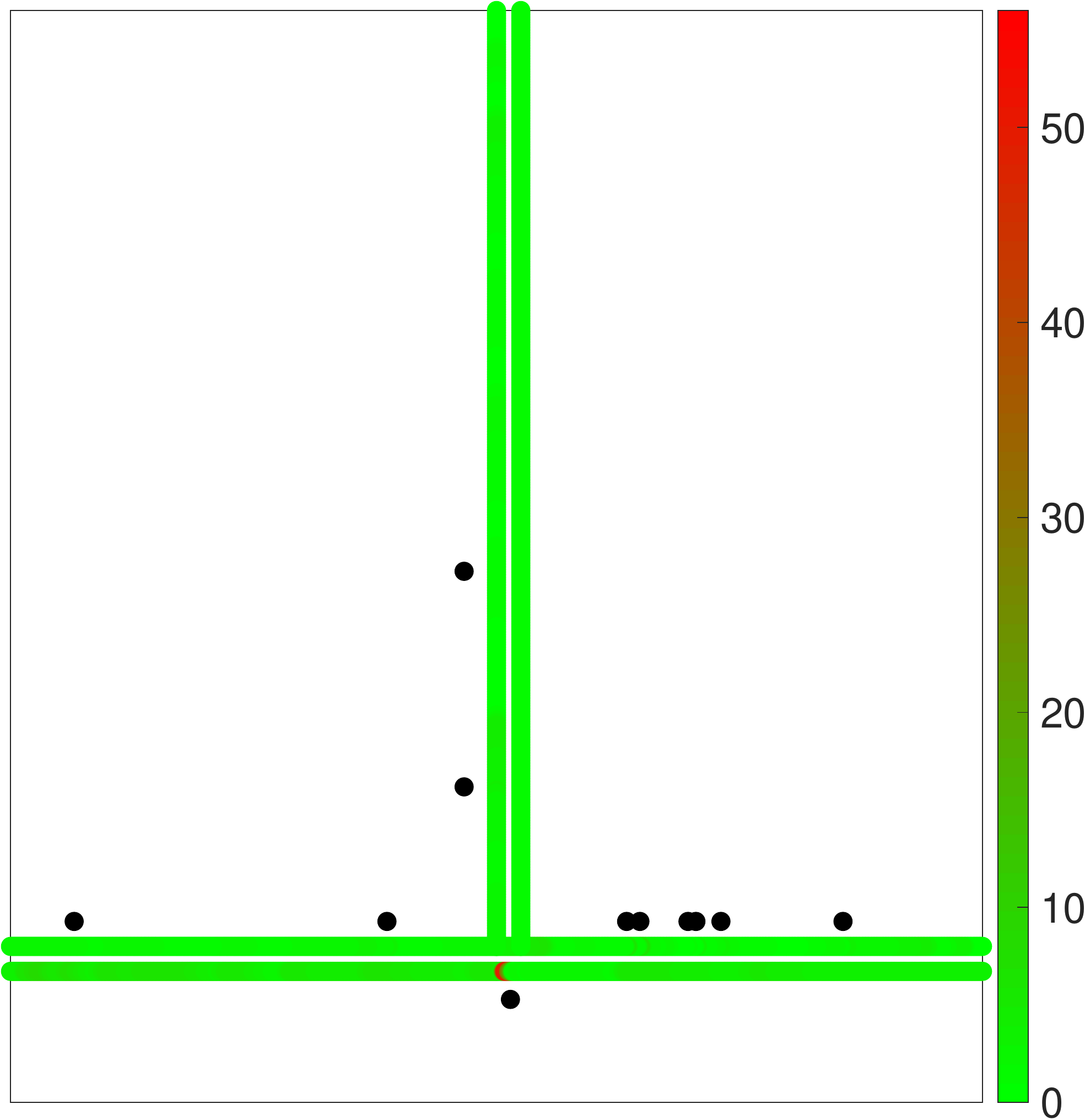}
}
\subfloat[][Density $t=90\,\mymin$]{\label{fig:AVdens4}
\includegraphics[width=0.235\columnwidth]{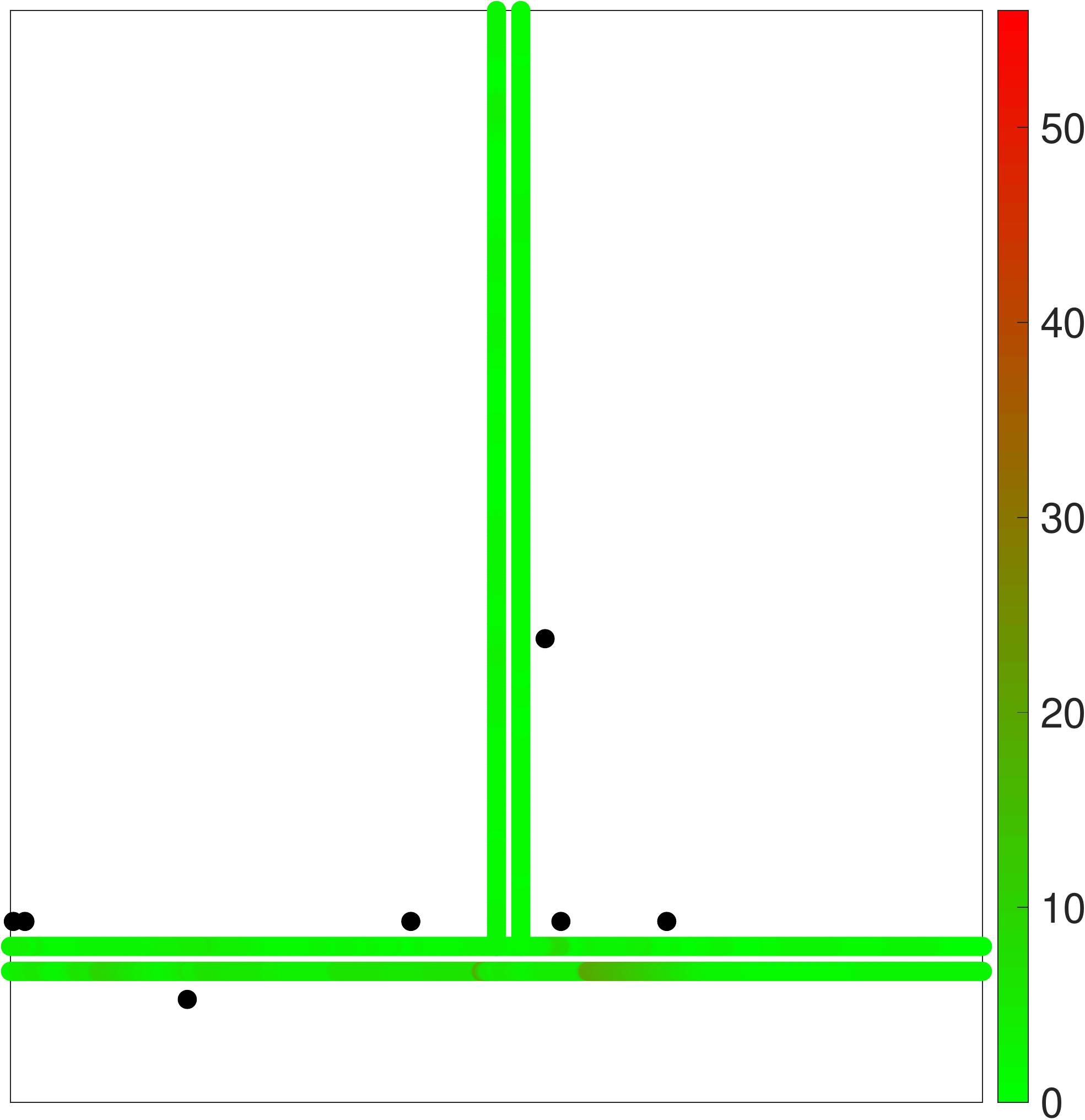}
}
\caption{Section \ref{sec:autovieData} test. Density of vehicles at different times of the simulation.}
\label{fig:testAV}
\end{figure}

In Figure \ref{fig:testAV2} we draw the variation in time of the total emissions produced along roads 2 and 3, where the dynamics are mainly influenced by the presence of GPS data. 
To explain the different trend of the curves, we need to look at the traffic dynamics after $60\,\min$ of simulation, see Figure \ref{fig:testAV}\subref{fig:AVdens3}. Here we observe a slowing vehicle at the end of road 1 that forces the macroscopic dynamics to slow down abruptly. Therefore, few trucks enter road 2, causing the sharp decrease in emissions shown in Figure \ref{fig:testAV2} (blue line). This phenomenon is not captured by the model without GPS data (red line). Then, the congestion at the end of road 1 slowly melts away and the trucks cross the intersection towards road 2, causing the slowdown still visible at the end of the simulation, see Figure \ref{fig:testAV}\subref{fig:AVdens4}. The trend of the two emission curves, calculated with and without GPS data (blue and red line respectively), differ from then on, since the first model is able to detect significant changes in traffic speed and acceleration.
On road 3, instead, the presence of real data produces higher emission values than the model without vehicle trajectories. 

\begin{figure}[h!]
\centering
\includegraphics[width=0.3\columnwidth]{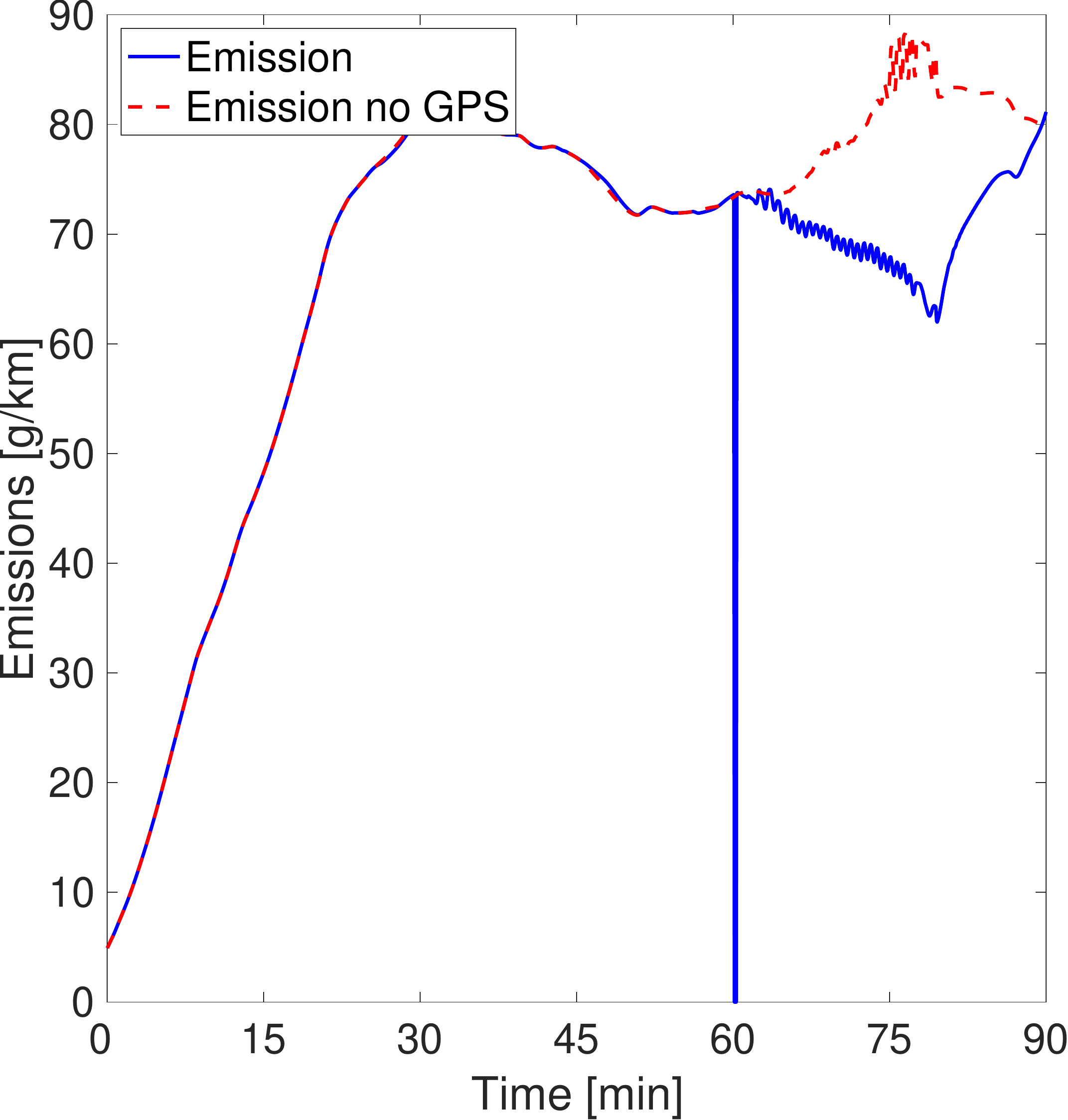}
\quad
\includegraphics[width=0.3\columnwidth]{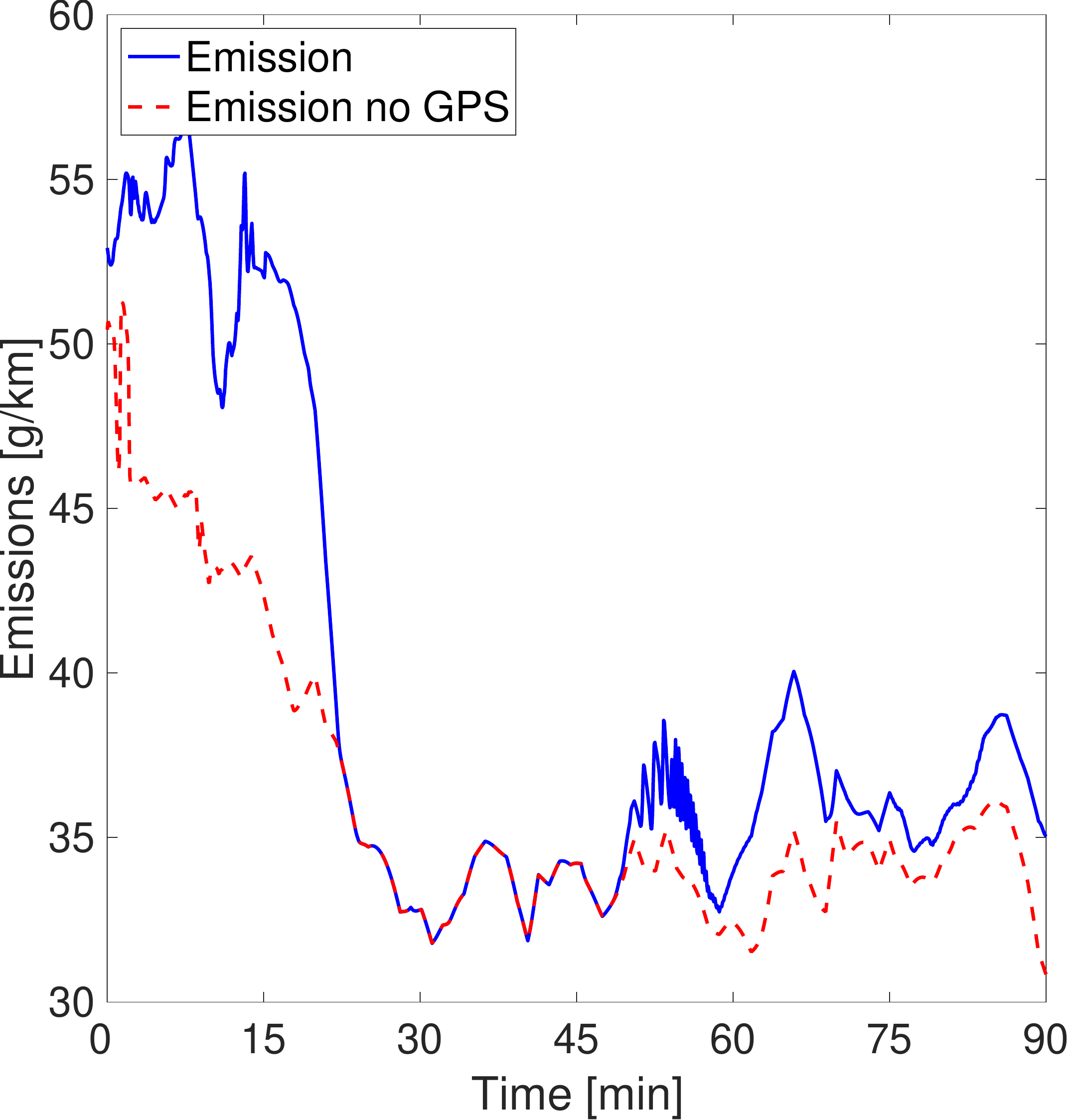}
\caption{Section \ref{sec:autovieData} test. Total emissions on road 2 (left) and road 3 (right).}
\label{fig:testAV2}
\end{figure}

In Figure \ref{fig:strada3}, we further investigate the dynamics on the last 15 km of road 3, drawing the density, speed, acceleration and $\nox$ emissions estimated at different times. Note that the direction of motion is from right to left, according to the network depicted in Figure \ref{fig:rete}. We observe that the presence of real trajectory data is responsible for the sudden changes in speed and acceleration that strongly influence the density and the emissions estimates.

\begin{figure}[h!]
\centering
\subfloat[][Density $t=2\,\min$]{
\includegraphics[width=0.23\columnwidth]{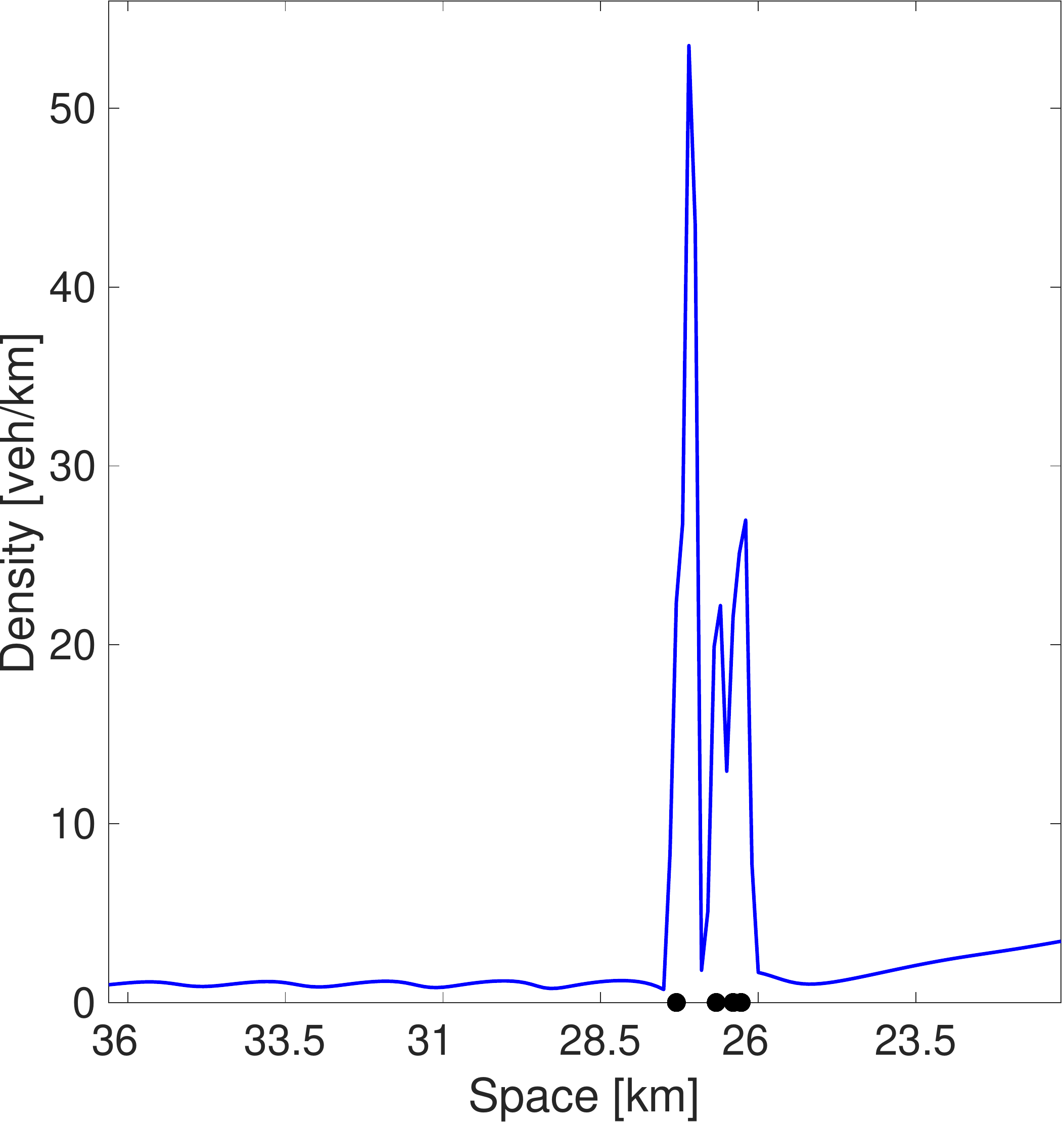}
}
\subfloat[][Speed $t=2\,\min$]{
\includegraphics[width=0.23\columnwidth]{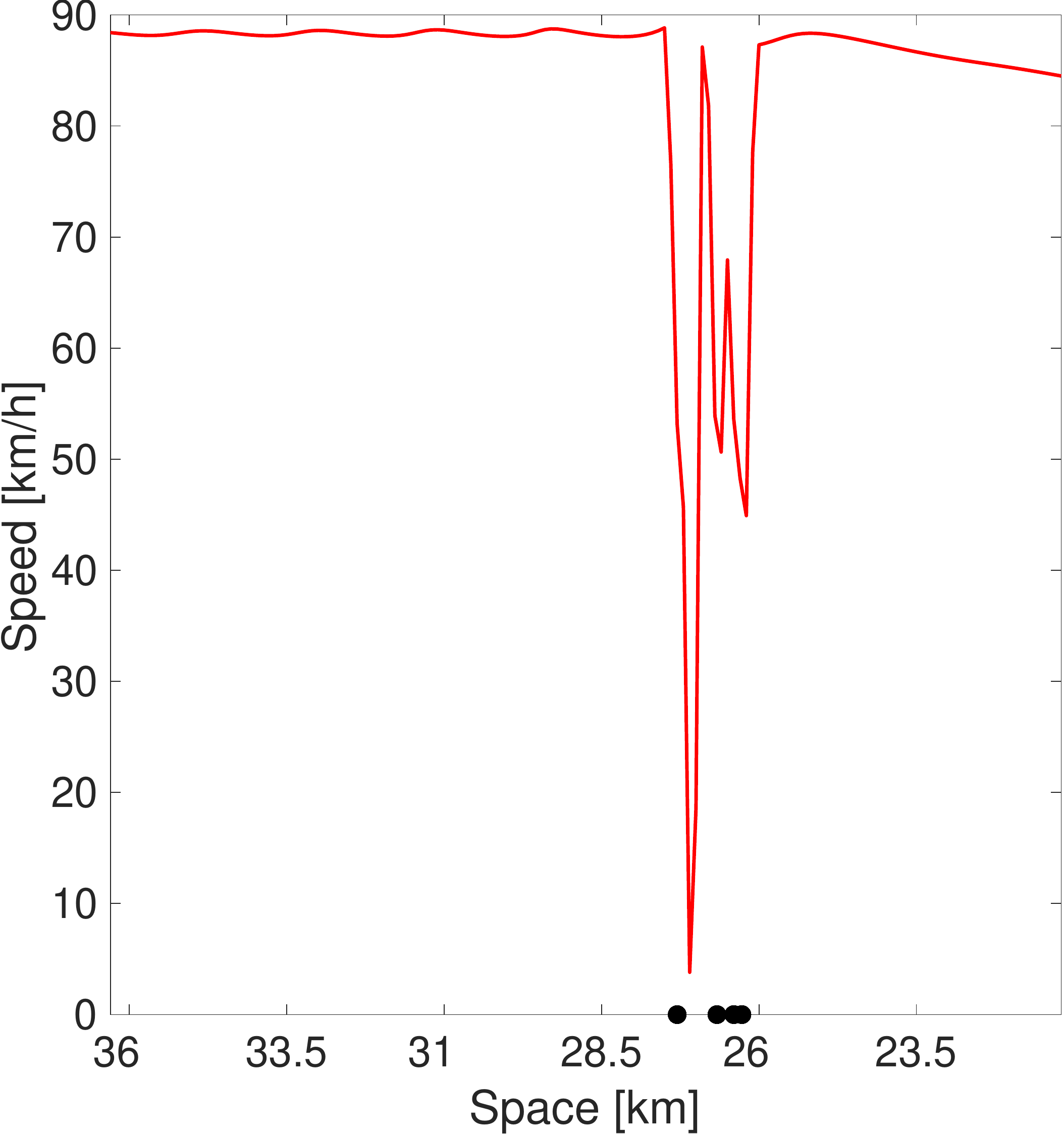}
}
\subfloat[][Accel. $t=2\,\min$]{
\includegraphics[width=0.231\columnwidth]{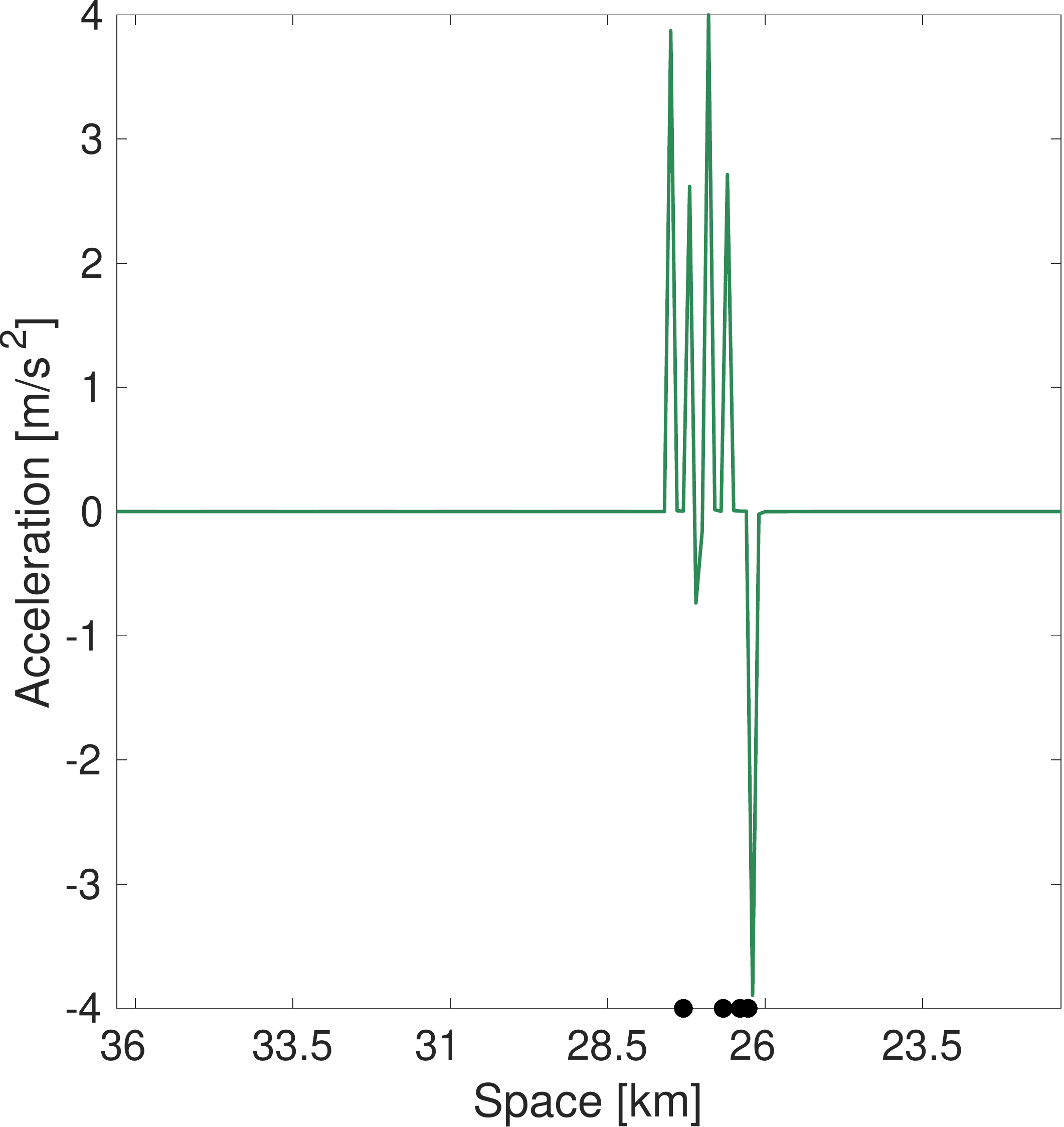}
}
\subfloat[][Emissions $t=2\,\min$]{
\includegraphics[width=0.231\columnwidth]{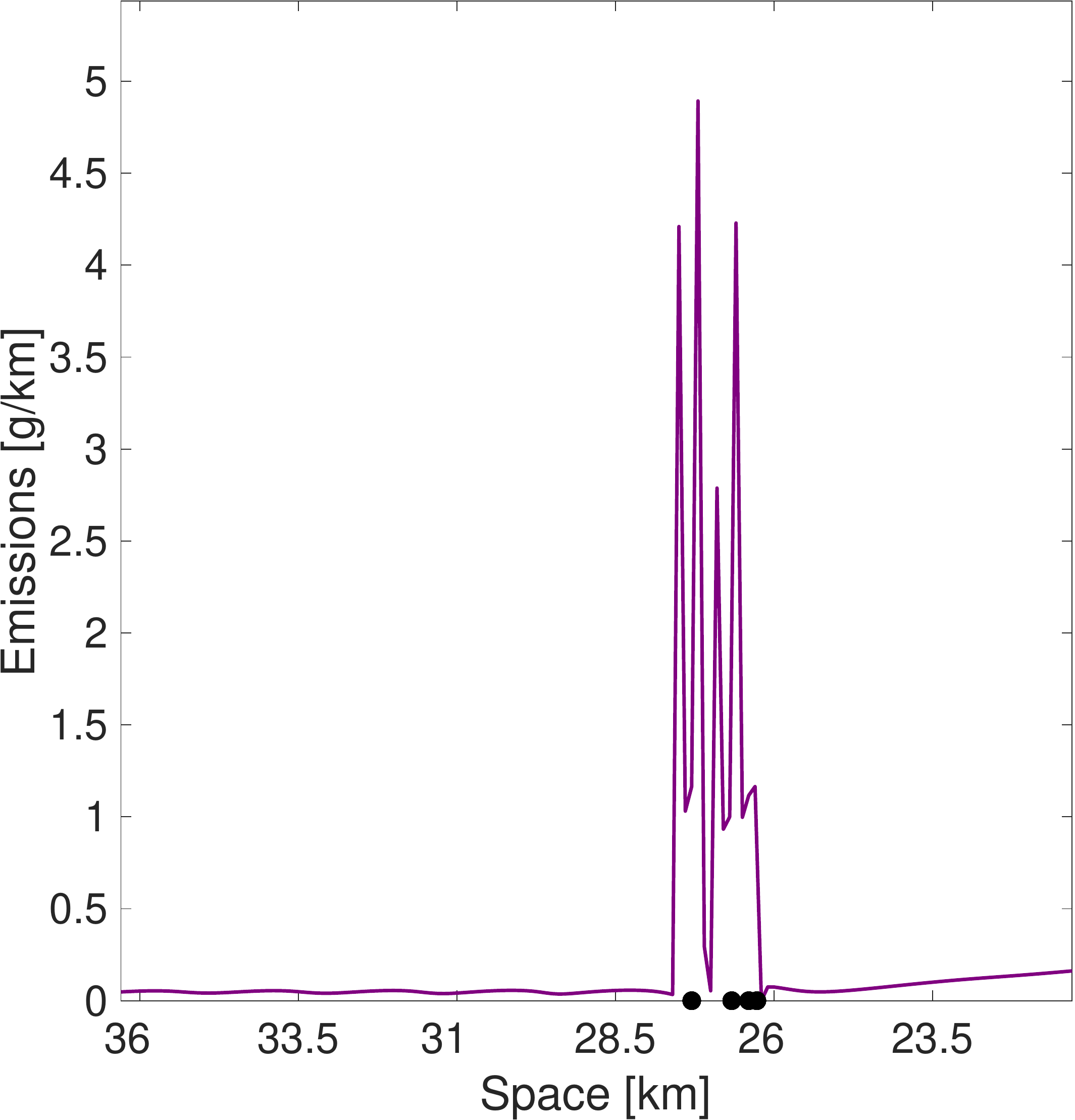}
}\\
\subfloat[][Density $t=67\,\min$]{
\includegraphics[width=0.23\columnwidth]{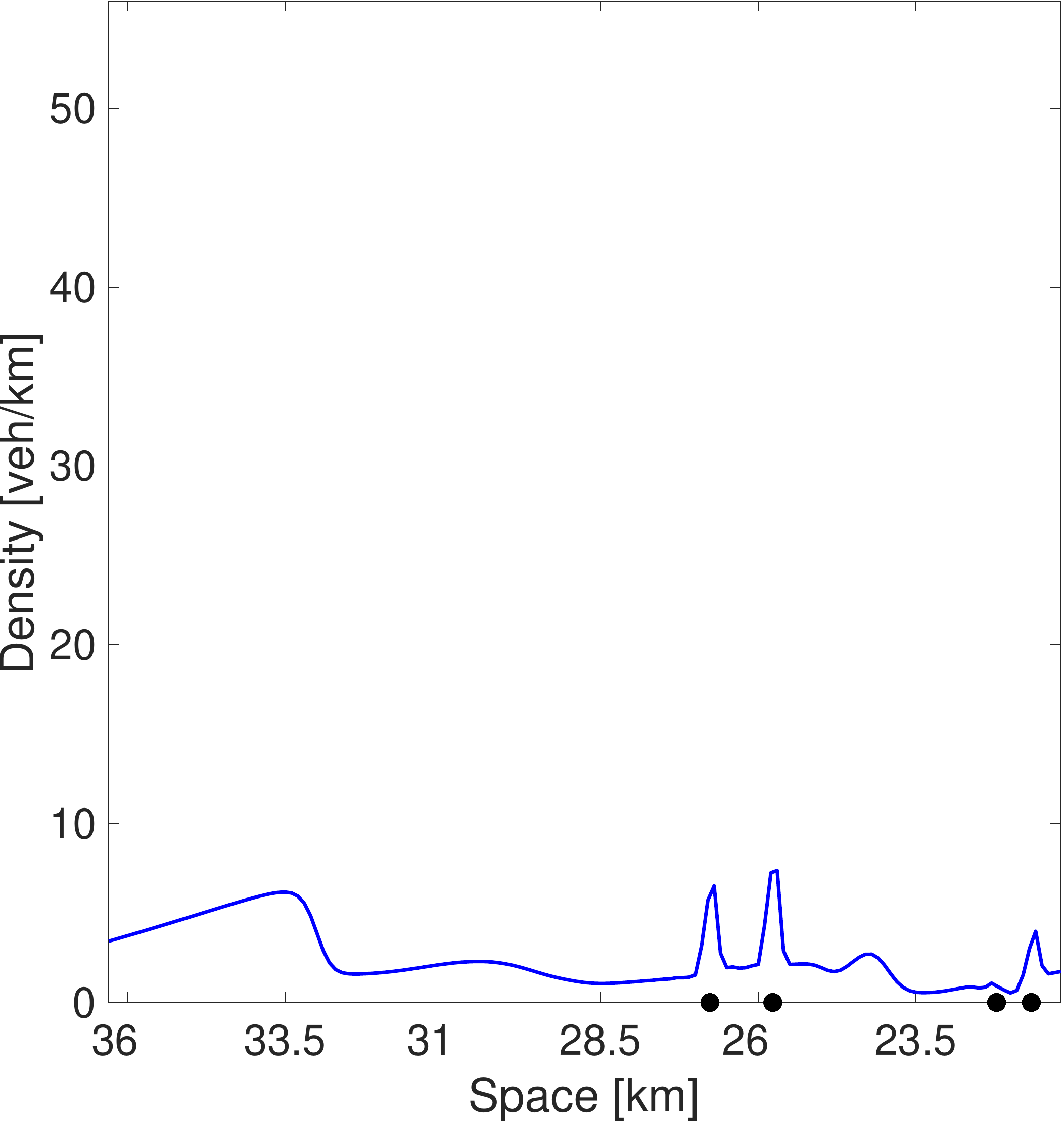}
}
\subfloat[][Speed $t=67\,\min$]{
\includegraphics[width=0.23\columnwidth]{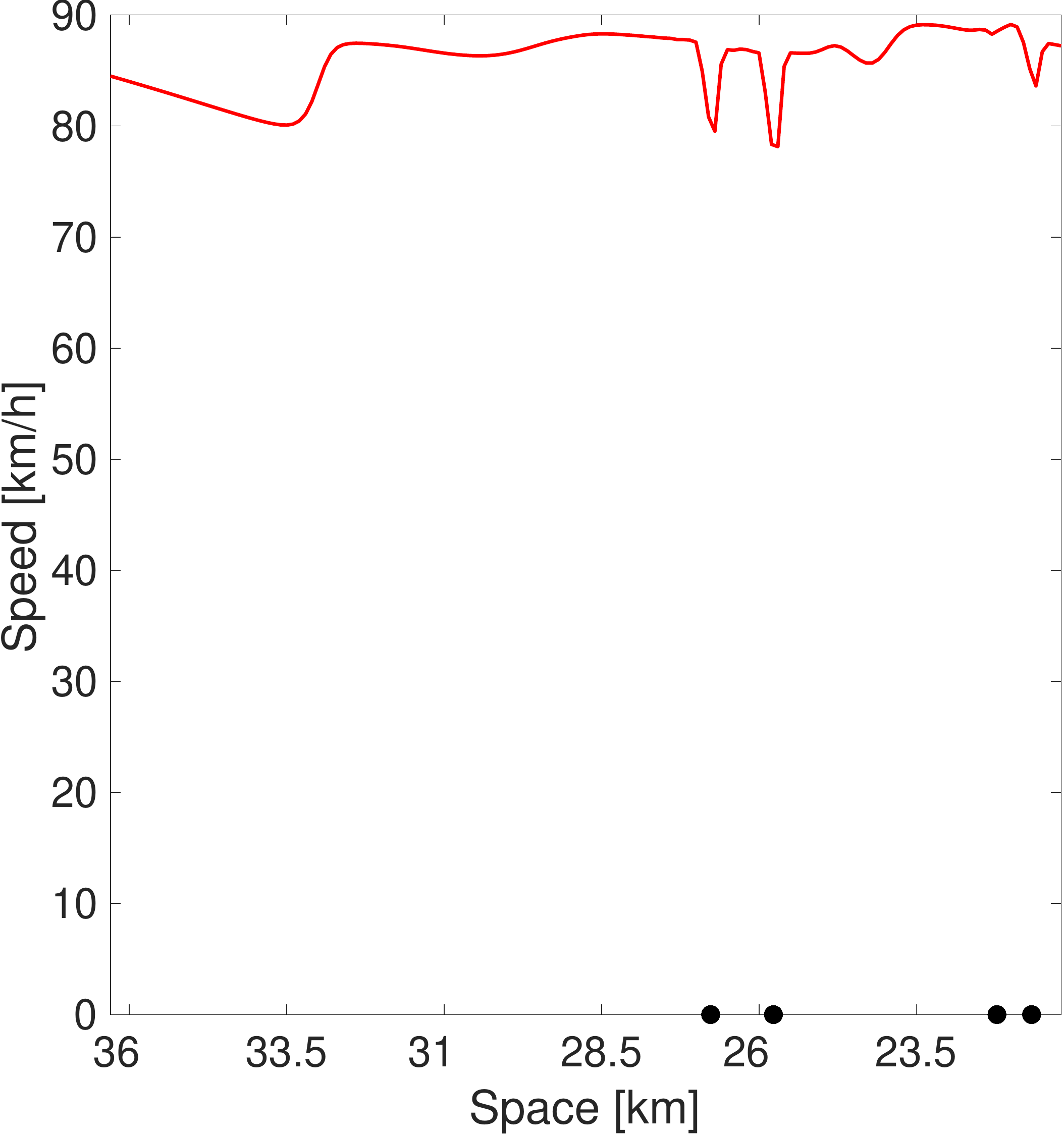}
}
\subfloat[][Accel. $t=67\,\min$]{
\includegraphics[width=0.231\columnwidth]{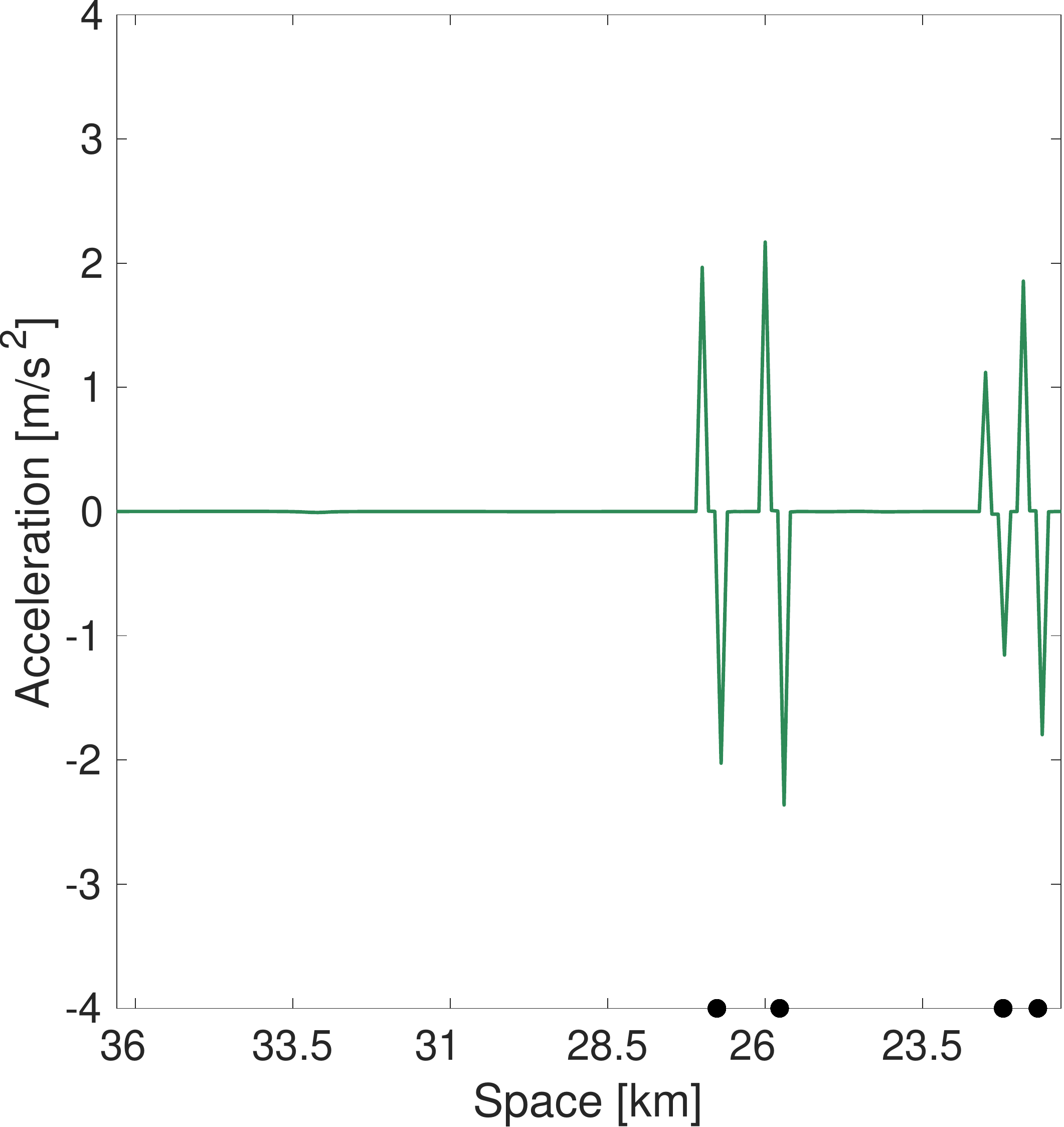}
}
\subfloat[][Emissions $t=67\,\min$]{
\includegraphics[width=0.231\columnwidth]{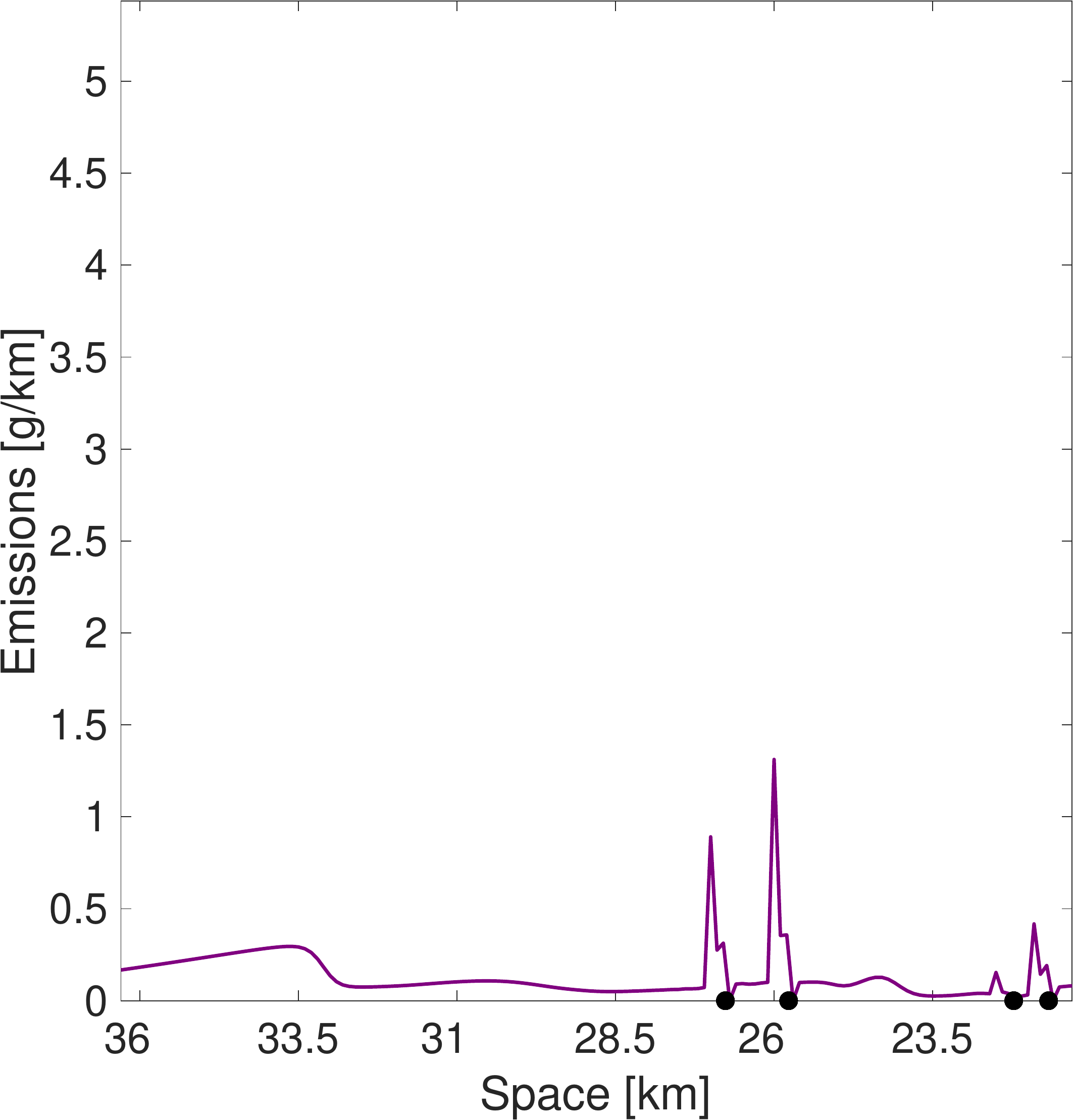}
}\\
\subfloat[][Density $t=85\,\min$]{
\includegraphics[width=0.23\columnwidth]{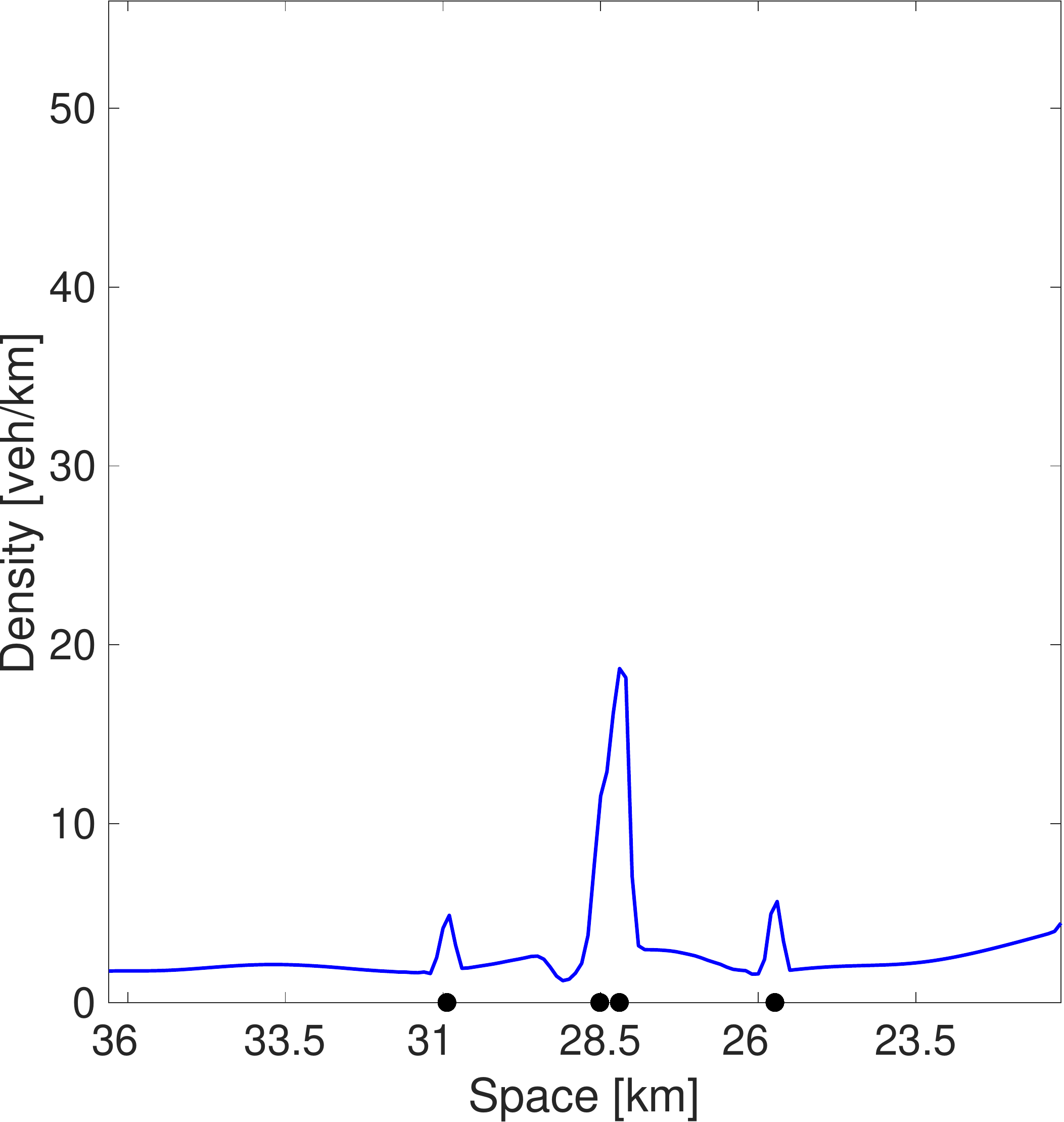}
}
\subfloat[][Speed $t=85\,\min$]{
\includegraphics[width=0.23\columnwidth]{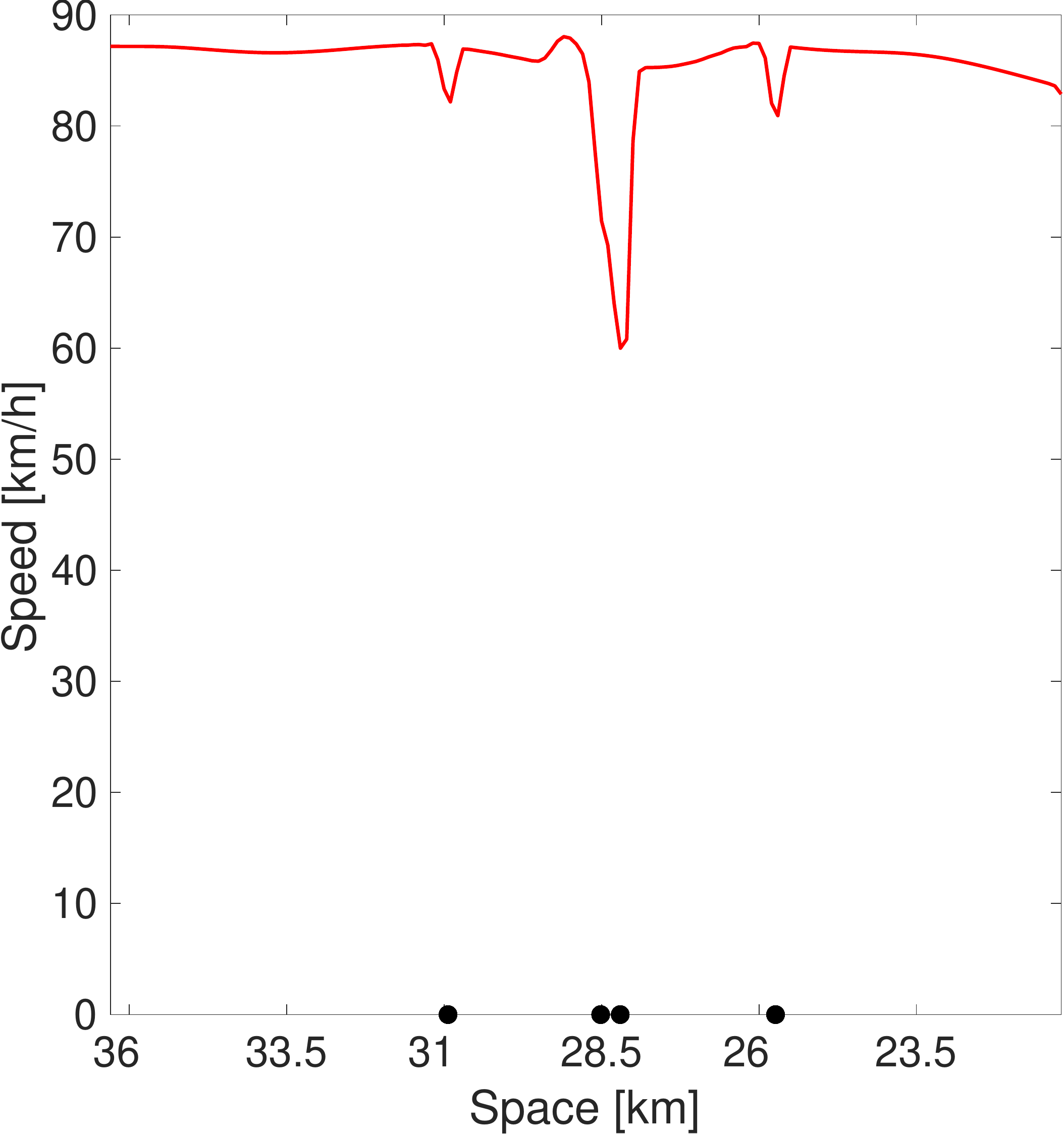}
}
\subfloat[][Accel. $t=85\,\min$]{
\includegraphics[width=0.231\columnwidth]{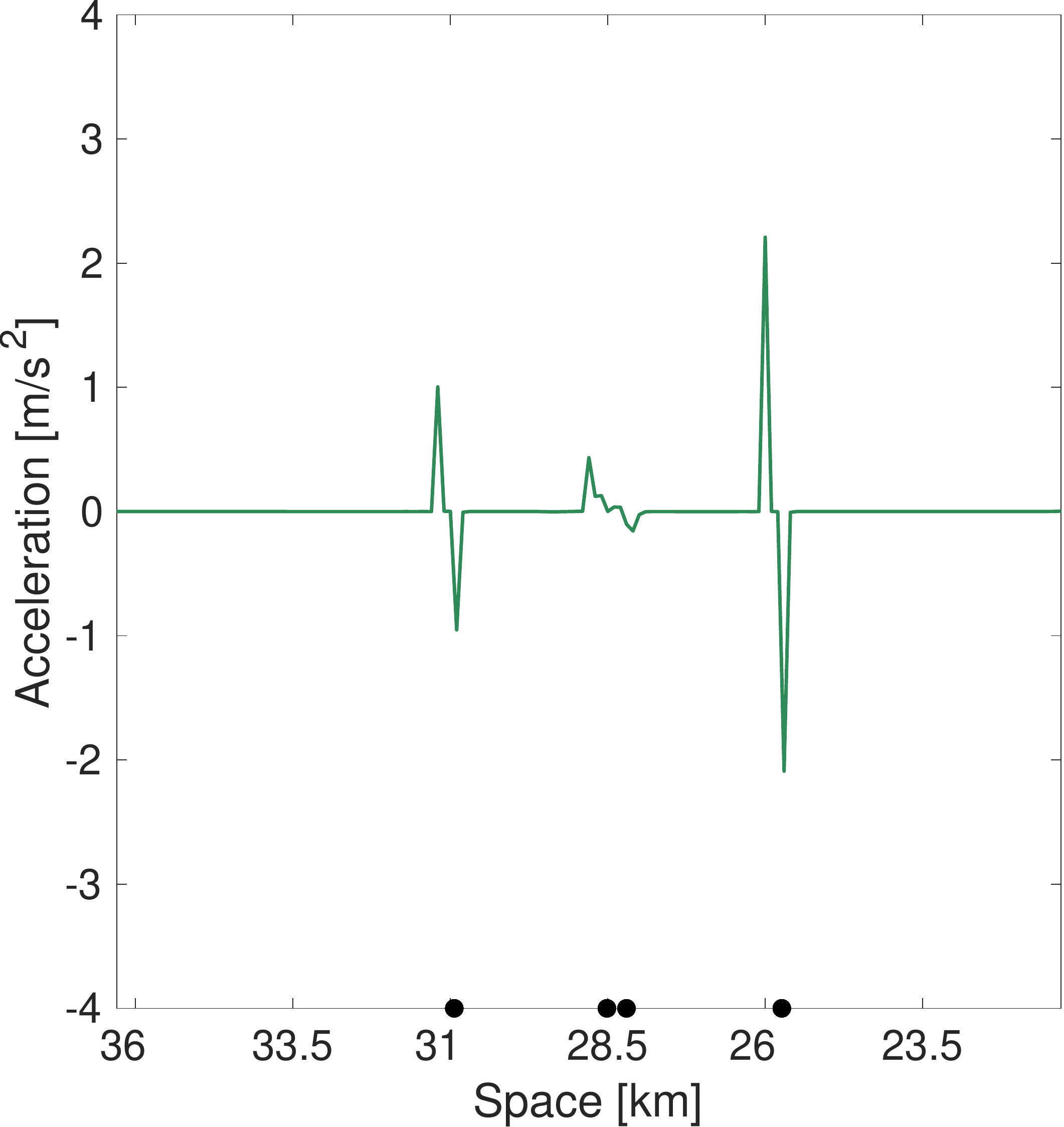}
}
\subfloat[][Emissions $t=85\,\min$]{
\includegraphics[width=0.231\columnwidth]{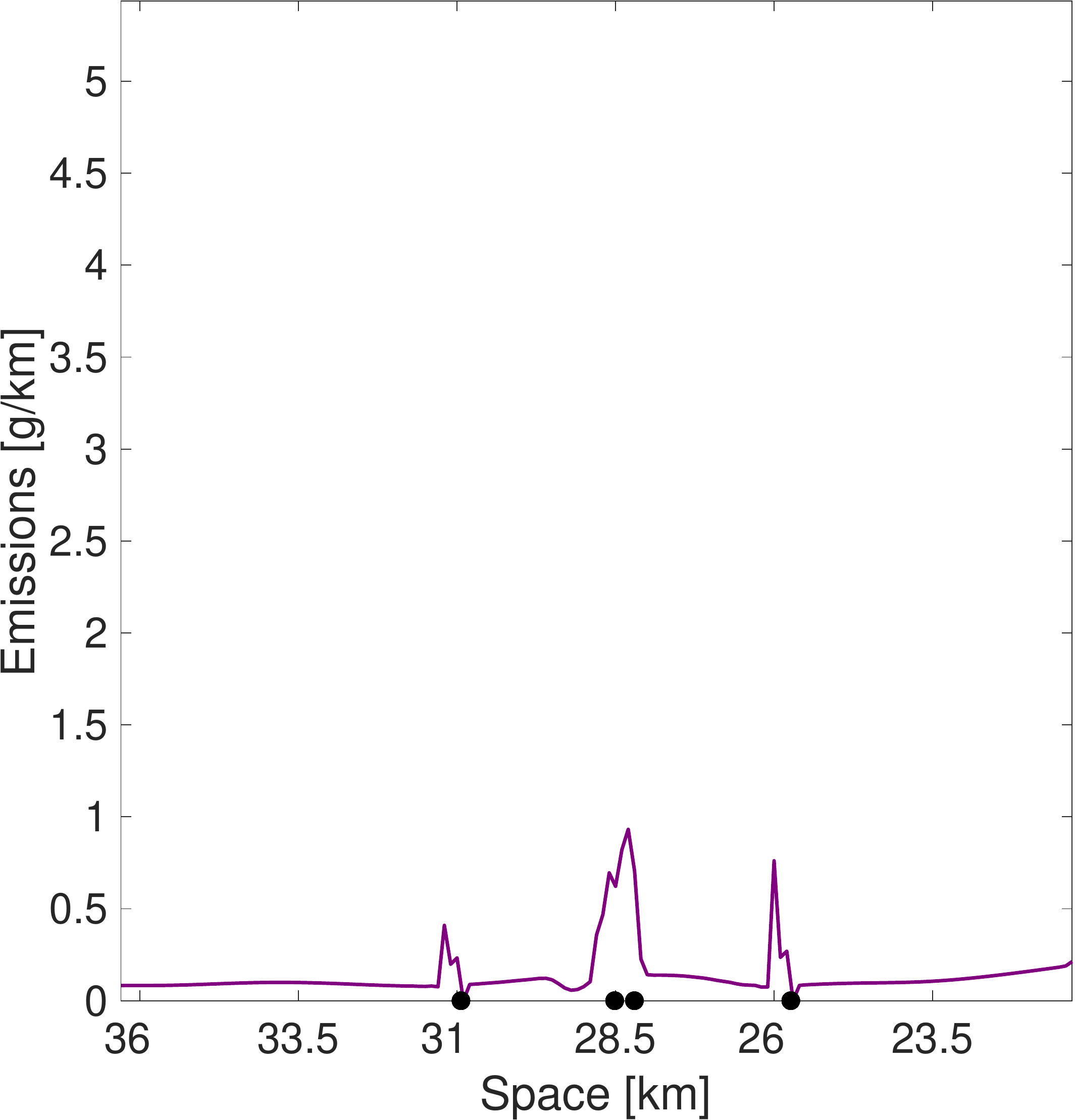}
}
\caption{Section \ref{sec:autovieData} test. Density, speed, acceleration and $\nox$ emissions on the last 15 km of road 3 at different times.}
\label{fig:strada3}
\end{figure}

\medskip

Finally, we analyze the diffusion of $\nox$ pollutants into the air. Let us denote by $\psi$ the concentration of $\nox$ in unit of weight per unit of volume. First of all we assume that $\psi$ is constant along the $z$-axis, in order to reduce to a two-dimensional problem on a domain $\Omega$ in which the network is located. More precisely, we set $\Omega=[-L_{x},L_{x}]\times[-L_{y},L_{y}]$ where $L_{x}=15\,\km$ and $L_{y}=5\,\km$; hence we focus our attention on a rectangle around the junctions that involves only the four horizontal roads, see Figure \ref{fig:scheletro}.
The domain $\Omega$ is discretized through a grid of steps $\dx=\dy=10\,\mymeter$; to do this, we divide each cell of the traffic dynamics ($100\,\mymeter$ long) into 10 smaller cells that inherit the density, velocity, acceleration and emissions previously computed.
Then, we use the emissions obtained from the traffic dynamics as source term of the following diffusion problem
\begin{equation}\label{eq:diffOriz}
	\begin{cases}
		\disp\frac{\de \psi}{\de t}(x,y,t)-\mu\Delta \psi(x,y,t) = S(x,y,t) &\quad\text{in $\Omega\times(0,T]$}\smallskip\\
		\psi(x,y,0) = 0 &\quad\text{in $\Omega$},	
	\end{cases}
\end{equation}
where $\mu$ is the diffusion coefficient, fixed as $\mu=10^{-8}\,\km^{2}/\myhour$ for aerosols \cite{Sportisse2010}. To define the source term $S$ we exploit the emission rates $E_{r}(x,t)$ computed through \eqref{eq:emissioni} on the horizontal roads of the network. Since the traffic model is one-dimensional in space, we trivially extend the emissions into the $y$-axis as
\begin{equation*}
	E(x,y,t) = \begin{cases}
		E_{r}(x,t) &\quad\text{if $x\in[a_{r},b_{r}]$, $y\in[c_{r},d_{r}]$, $r=1,\dots,4$, $t\in[0,T]$}\\
		0&\quad\text{otherwise,}
	\end{cases} 
\end{equation*}
where $[a_{r},b_{r}]$ and $[c_{r},d_{r}]$ are the horizontal and vertical length of the road $r$ in $\Omega$, respectively (in our one-dimensional configuration $d_r-c_r=\dy$). Then, the source term $S$ is defined as $S(x,y,t)=E(x,y,t)/\dx^{3}$.

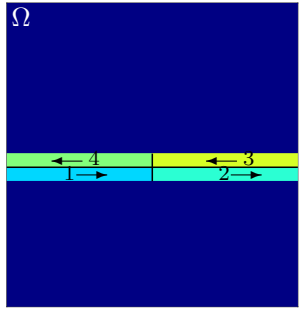
\begin{figure}[h!]
\centering
\begin{overpic}[width=0.25\columnwidth]{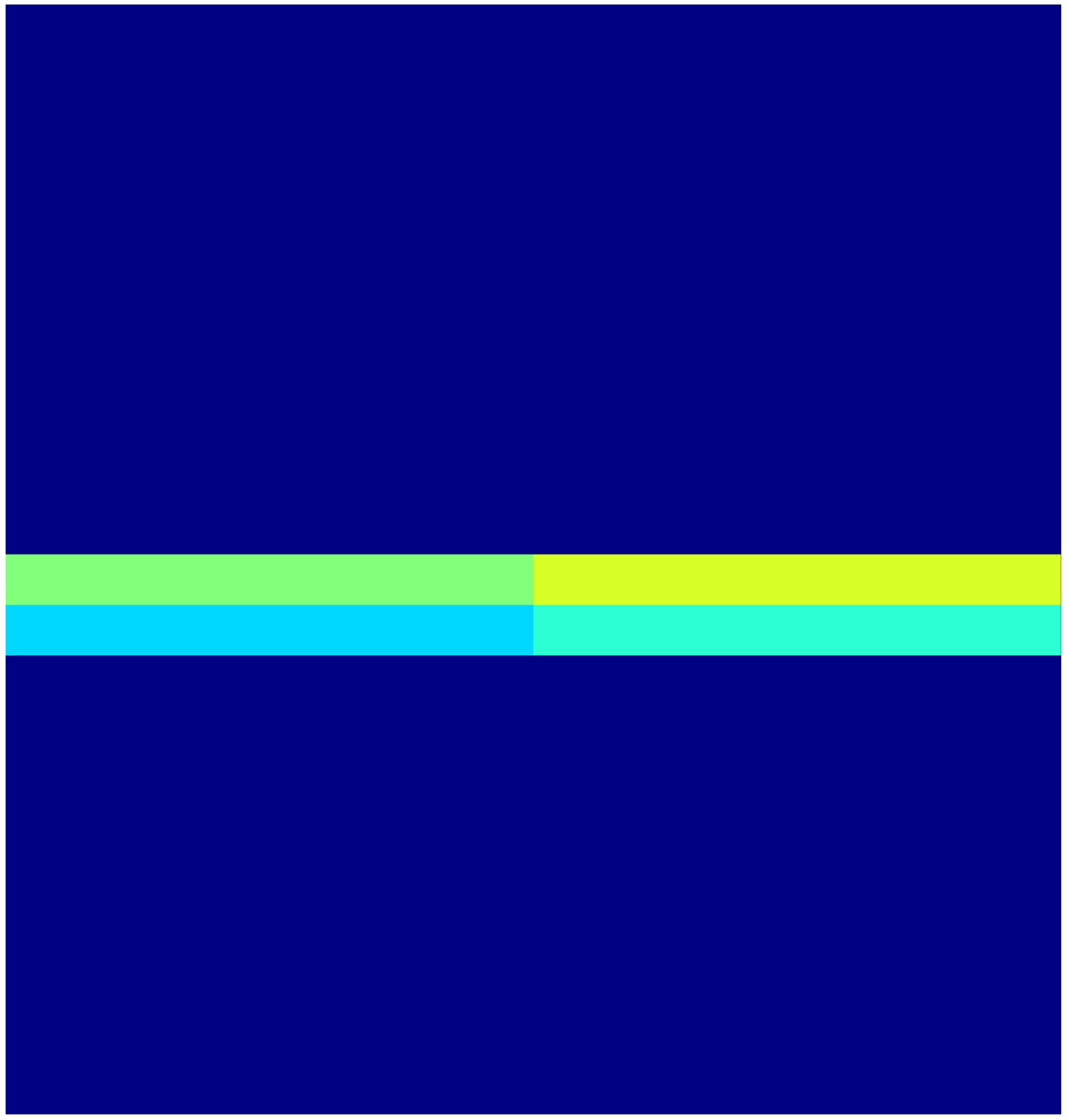}
\put(0.5,46){\color{black}\linethickness{0.02cm}\line(92,0){94.2}}
\put(47.6,41.45){\color{black}\linethickness{0.02cm}\line(0,2){9}}
\put(25,48){\vector(-1,0){10}}\put(27,46.6){\footnotesize 4}
\put(19,41.8){\footnotesize 1} \put(23,43.5){\vector(1,0){10}}
\put(75,48){\vector(-1,0){10}}\put(77,46.6){\footnotesize 3}
\put(69,41.8){\footnotesize 2}\put(73,43.5){\vector(1,0){10}}
\put(2,92){\color{white} $\Omega$}
\end{overpic}
\caption{Section \ref{sec:autovieData} test. Domain $\Omega$ built around the junctions involving the four horizontal roads.}
\label{fig:scheletro}
\end{figure}

The diffusion problem \eqref{eq:diffOriz} is numerically solved with an explicit finite difference scheme.
In Figure 
\ref{fig:diffusione2} we show the source term of $\nox$ emissions and their diffusion in air. In the top plots we observe several red peaks corresponding to high amounts of $\nox$ emissions due to the traffic dynamics. The red peaks on road 3 are connected to the high emission values shown in the last column of Figure \ref{fig:strada3}. The bottom plots show the spread of $\nox$ concentration in $\Omega$. We observe that the diffusion of pollutants is slower than the source term, and is strongly influenced by microscopic information. Indeed, the presence of real data allows us to take into account variations in speed and acceleration that we would not be aware of without the trajectory data.
Therefore, the use of GPS data enables us to approximate the source term well and to reproduce at macroscopic level the emission peaks due to unpredictable traffic dynamics.

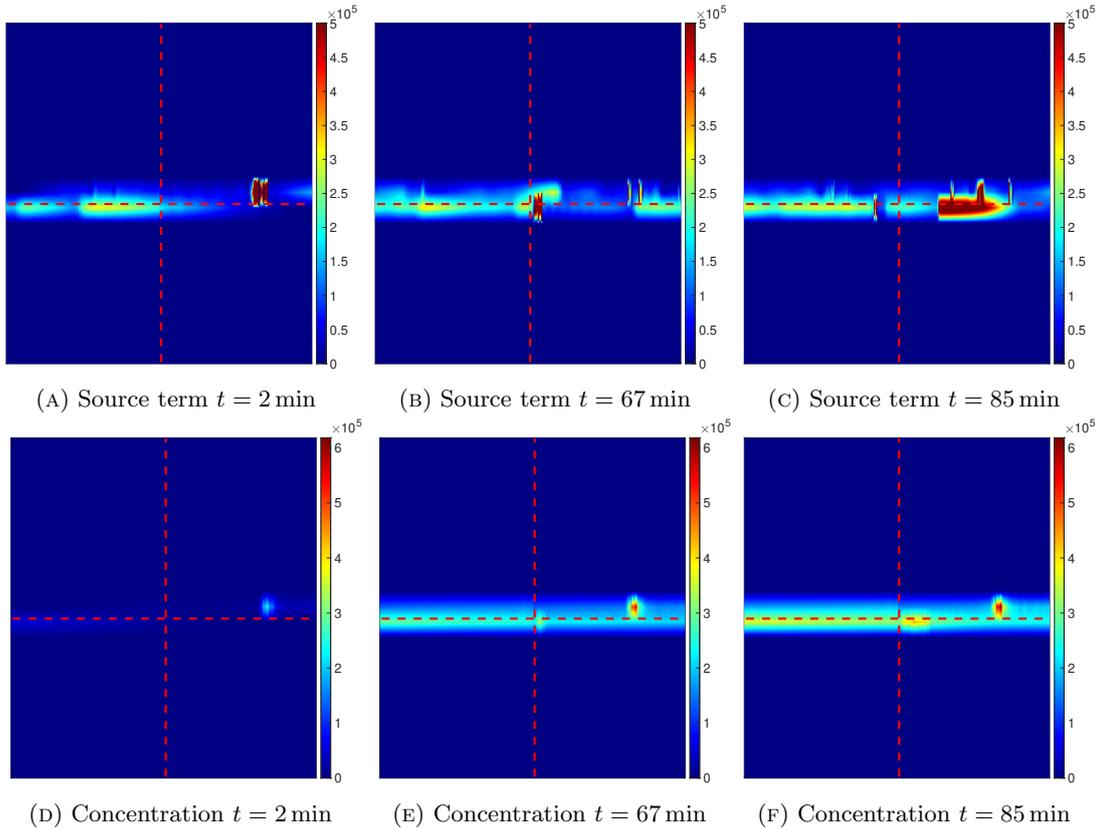
\begin{figure}[h!]
\centering
\subfloat[][Source term $t=2\,\mymin$]{\label{fig:source1}
\begin{overpic}[width=0.3\columnwidth,tics=5]{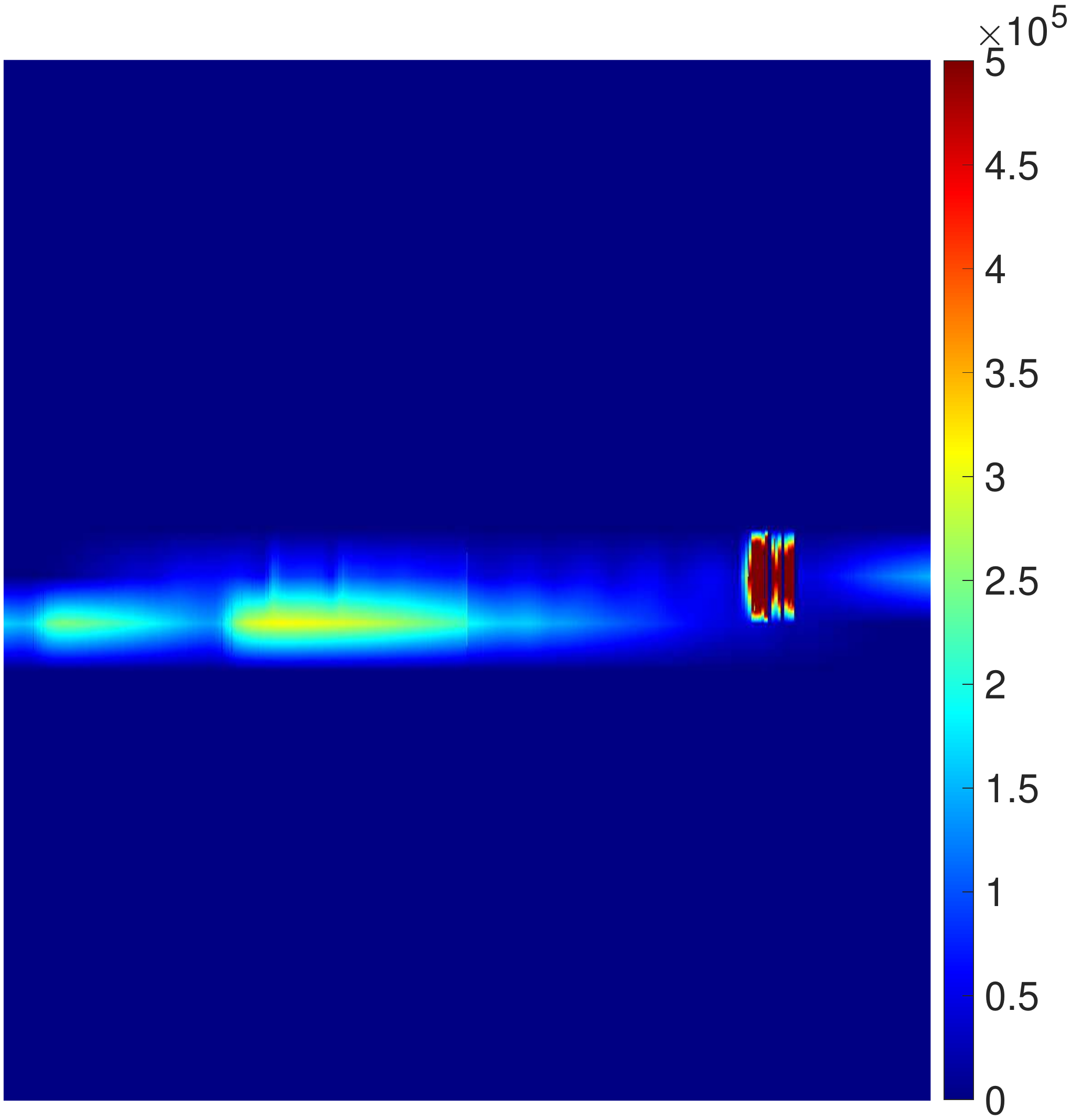}
\put(0.6,45){\tikz \draw[red,dashed,thick] (0,0)--(4,0);}
\put(42,1.9){\tikz \draw[red,dashed,thick] (0,0)--(0,4.48);}
\end{overpic}
}
\subfloat[][Source term $t=67\,\mymin$]{
\begin{overpic}[width=0.3\columnwidth,tics=5]{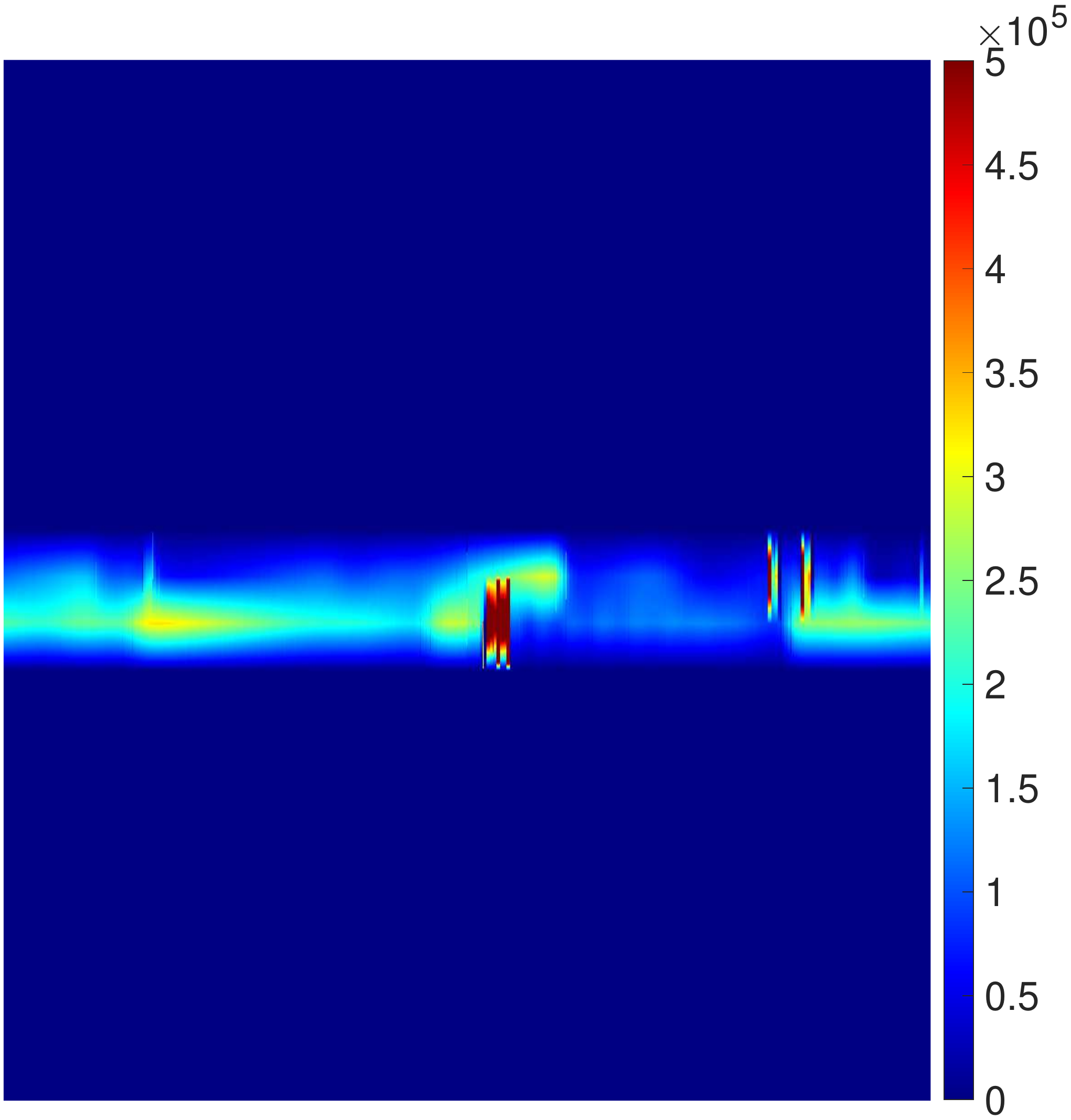}
\put(0.6,45){\tikz \draw[red,dashed,thick] (0,0)--(4,0);}
\put(42,1.9){\tikz \draw[red,dashed,thick] (0,0)--(0,4.48);}
\end{overpic}
}
\subfloat[][Source term $t=85\,\mymin$]{
\begin{overpic}[width=0.3\columnwidth,tics=5]{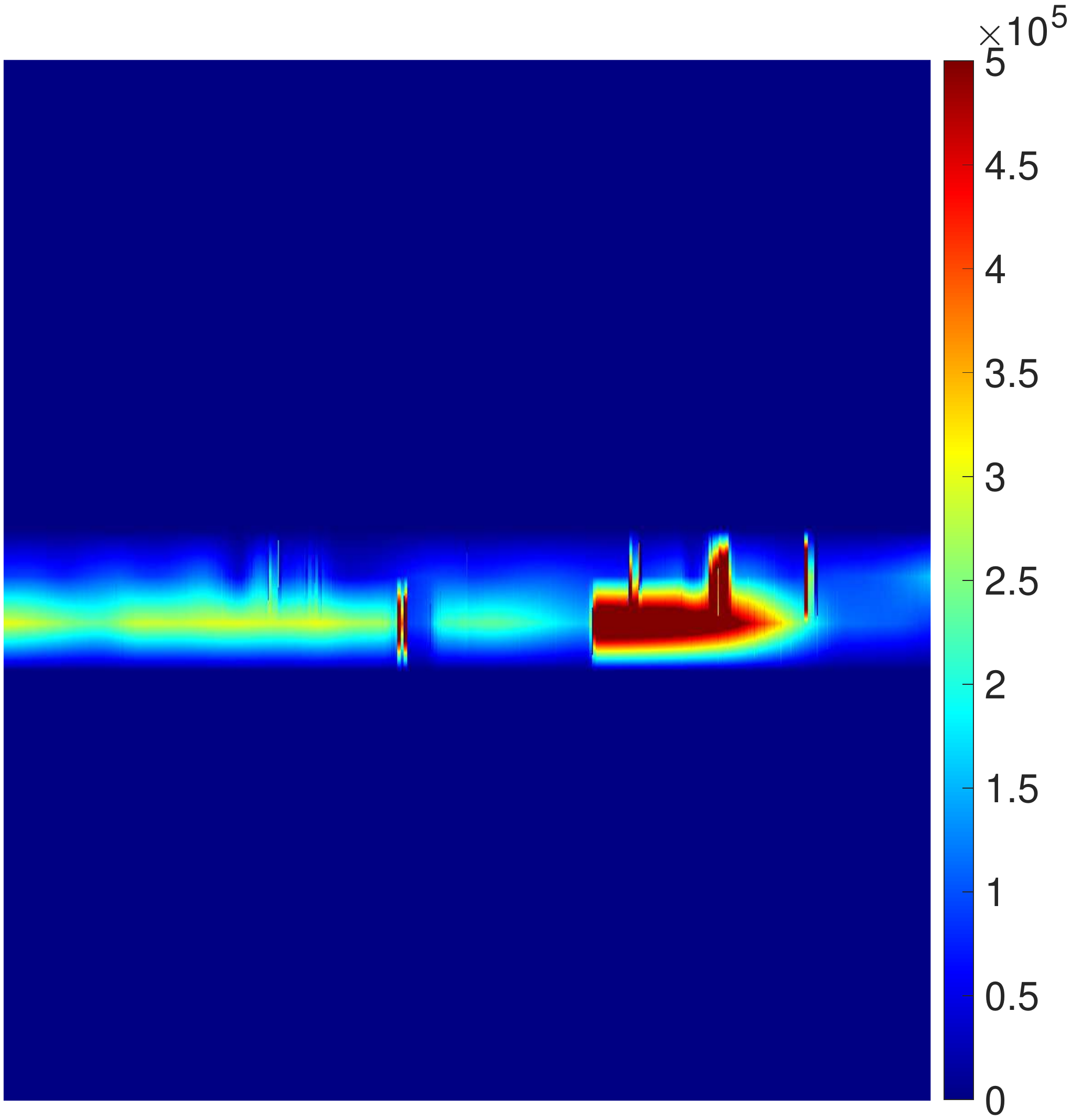}
\put(0.6,45){\tikz \draw[red,dashed,thick] (0,0)--(4,0);}
\put(42,1.9){\tikz \draw[red,dashed,thick] (0,0)--(0,4.48);}
\end{overpic}
}
\\
\subfloat[][Concentration $t=2\min$]{
\begin{overpic}[width=0.3\columnwidth,tics=5]{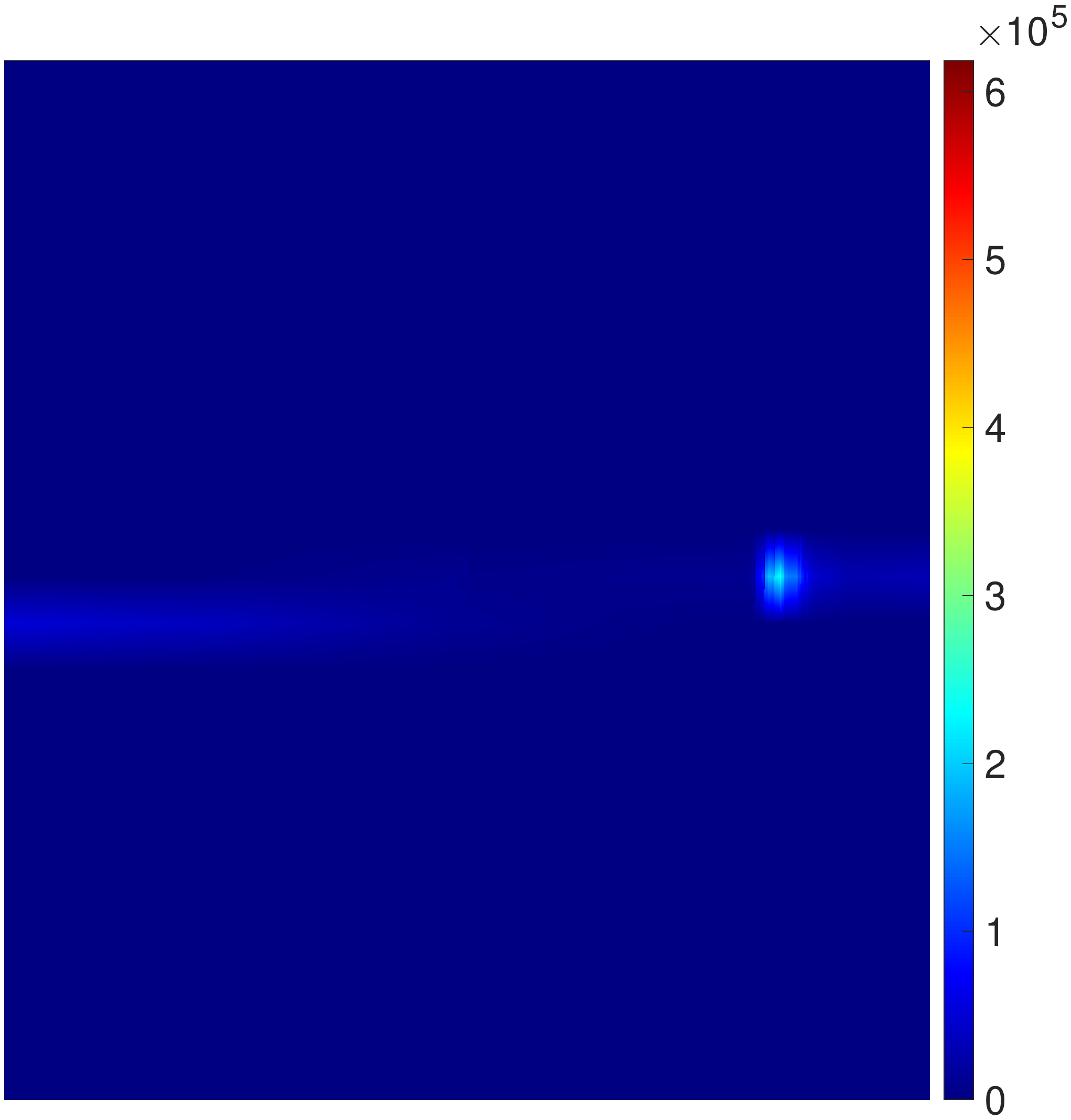}
\put(0.6,45){\tikz \draw[red,dashed,thick] (0,0)--(4,0);}
\put(42,1.9){\tikz \draw[red,dashed,thick] (0,0)--(0,4.48);}
\end{overpic}
}
\subfloat[][Concentration $t=67\min$]{
\begin{overpic}[width=0.3\columnwidth,tics=5]{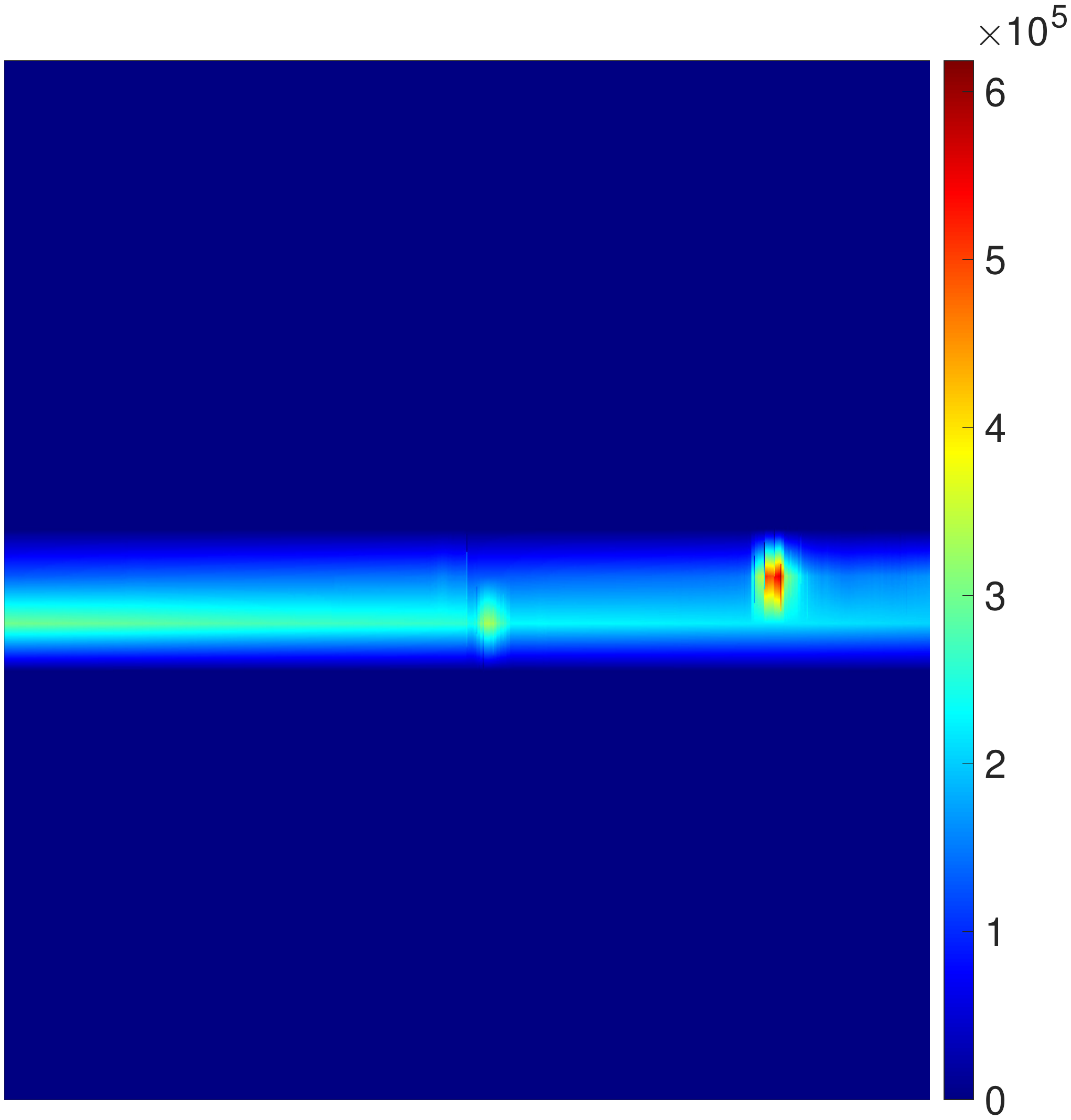}
\put(0.6,45){\tikz \draw[red,dashed,thick] (0,0)--(4,0);}
\put(42,1.9){\tikz \draw[red,dashed,thick] (0,0)--(0,4.48);}
\end{overpic}
}
\subfloat[][Concentration $t=85\min$]{
\begin{overpic}[width=0.3\columnwidth,tics=5]{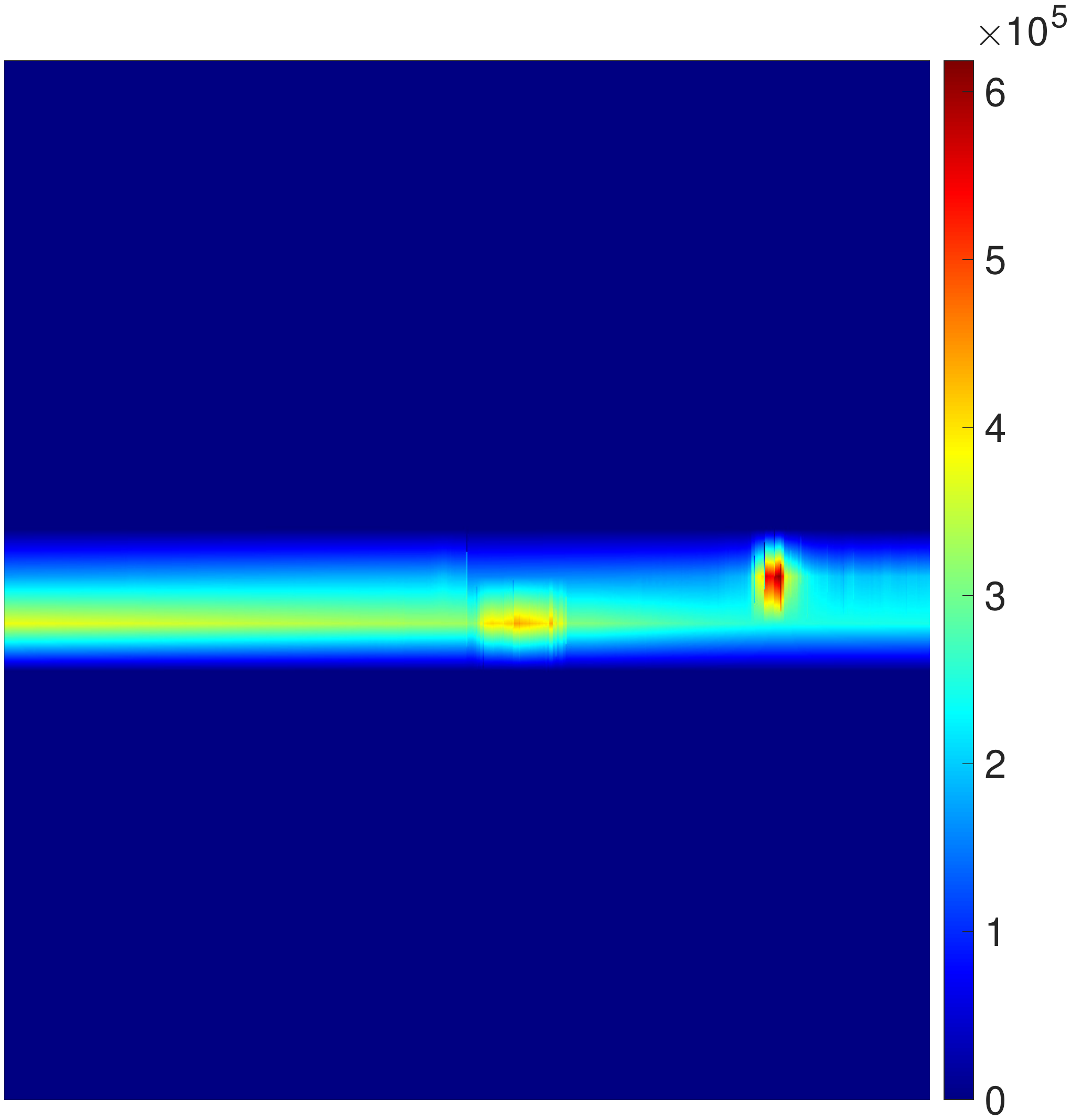}
\put(0.6,45){\tikz \draw[red,dashed,thick] (0,0)--(4,0);}
\put(42,1.9){\tikz \draw[red,dashed,thick] (0,0)--(0,4.48);}
\end{overpic}}
\caption{Section \ref{sec:autovieData} test. Source term of $\nox$ emissions (top) and diffusion of $\nox$ concentration into the air (bottom) at different times. The dashed red lines are used to identify the four roads in $\Omega$.}
\label{fig:diffusione2}
\end{figure}

\section{Conclusions}\label{sec:conclusion}

In this paper, we proposed a second order macroscopic traffic model that integrates multiple trajectory data into the velocity function as a tool to compute traffic quantities for estimating emissions. The combination of a macroscopic model with microscopic data is suggested by the computational efficiency of the former and the high accuracy of the latter.  The proposed numerical tests show that, even when trajectory data are sparse, they make it possible to reproduce variations in speed and acceleration that otherwise would not be observable. As a result, we obtained more accurate approximations of emissions.

In the near future we plan to improve the study by distinguishing between light and heavy vehicles. We will consider macroscopic multi-class models and appropriate emission formulas that are calibrated with respect to the vehicle types.

\appendix
\section{Technical proofs.}\label{sec:proof}

\begin{proof}[Proof of Proposition \ref{prop:monotonia1}]
It is sufficient to prove that under condition \eqref{eq:CFL_flux} the method is monotone. Let  
$$
H(j,n,\boldsymbol\rho^n) = \rho^{n}_{j}-\frac{\dt}{\dx}(F^{n}_{j+1/2}-F^{n}_{j-1/2}).
$$
Below we check that
$$
\frac{\partial}{\partial \rho^n_l}H(j,n,\boldsymbol\rho^n)\geq 0 \quad \mbox{ for all } l,j,\boldsymbol\rho^n.
$$
We have 
$$
\frac{\partial}{\partial \rho^n_l}H(j,n,\boldsymbol\rho^n) =
\left\{\hspace{-0.15cm}\begin{array}{ll}
\disp\frac{\dt}{\dx} \disp\frac{\partial F^n_{j-1/2}}{\partial \rho^n_{j-1}} & \mbox{ if } l=j-1
\medskip\\
1 + \disp\frac{\dt}{\dx} \disp\frac{\partial}{\partial \rho^n_{j}}\Big(F^n_{j-1/2}-F^n_{j+1/2}\Big) & \mbox{ if } l=j
\medskip\\
-\disp\frac{\dt}{\dx} \disp\frac{\partial F^n_{j+1/2}}{\partial \rho^n_{j+1}} & \mbox{ if } l=j+1.
\end{array}\right.
$$
By construction, the sending and receiving function in \eqref{eq:SRfunct} are monotone in $\rho$, i.e.\
$\partial S(\cdot,\cdot,\rho)/\partial\rho \geq 0$ and $\partial R(\cdot,\cdot,\rho)/\partial \rho \leq 0$.

In the following we set $S^n_j = S(x_j,t^n,\rho^n_j)$ and $R^n_j = R(x_j,t^n,\rho^n_j)$.
For $l=j-1$, we have
$$
\disp\frac{\partial F^n_{j-1/2}}{\partial \rho^n_{j-1}} =
\left\{\hspace{-0.15cm}\begin{array}{ll}
0 & \mbox{ if } R^n_j \leq S^n_{j-1}
\medskip\\
\disp\frac{\partial S^n_{j-1}}{\partial \rho^n_{j-1}} \geq 0 
& \mbox{ if } R^n_j > S^n_{j-1}.
\end{array}\right.
$$
Similarly, for $l=j+1$
$$
\disp\frac{\partial F^n_{j+1/2}}{\partial \rho^n_{j+1}} =
\left\{\hspace{-0.15cm}\begin{array}{ll}
0 & \mbox{ if } S^n_j \leq R^n_{j+1}
\medskip\\
\disp\frac{\partial R^n_{j+1}}{\partial \rho^n_{j+1}} \leq 0 
& \mbox{ if } S^n_j > R^n_{j+1}.
\end{array}\right.
$$
Therefore, in both cases $l=j\pm 1$, $\partial H(j,n,\boldsymbol\rho^n)/\partial \rho^n_l \geq 0$.
When $l=j$, we have to check that
\begin{equation}\label{eq:apx_lj}
1+\disp\frac{\dt}{\dx}\disp\frac{\partial}{\partial \rho^n_{j}}\Big(F^n_{j-1/2}-F^n_{j+1/2}\Big) =
1+\disp\frac{\dt}{\dx}\disp\frac{\partial}{\partial \rho^n_{j}}\Big(\min\{S^n_{j-1},R^n_{j}\} - \min\{S^n_{j},R^n_{j+1}\}\Big) \geq 0.
\end{equation}
We analyze the following four cases:
\begin{enumerate}[label=\roman*)]
\item If $S^n_{j-1} \leq R^n_j$ and $S^n_{j} \leq R^n_{j+1}$, then
$$
\disp\frac{\partial}{\partial \rho^n_{j}}\big(\min\{S^n_{j-1},R^n_{j}\} - \min\{S^n_{j},R^n_{j+1}\}\big) = 
\disp\frac{\partial}{\partial \rho^n_{j}}\left(S^n_{j-1} - S^n_j\right)
= -\disp\frac{\partial}{\partial \rho^n_{j}}S^n_j \leq 0.
$$
Thus, to obtain the inequality \eqref{eq:apx_lj}, we assume
\begin{equation*}
1-\disp\frac{\dt}{\dx}\Big|\disp\frac{\partial}{\partial \rho^n_{j}}S^n_j\Big|\geq 0,
\end{equation*}
where $S^n_j = \ff(x_j,t^n,\rho^n_j)$ or $S^n_j=\ff^{\max}(x_j,t^n)$. 
This leads to the condition
\begin{equation*}
1-\disp\frac{\dt}{\dx}\Big|\disp\frac{\partial}{\partial \rho}\ff(x_j,t^n,\rho)\Big|\geq 0.
\end{equation*}

\item If $S^n_{j-1} \leq R^n_j$ and $S^n_{j} > R^n_{j+1}$, then
$$
\disp\frac{\partial}{\partial \rho^n_{j}}\Big(\min\{S^n_{j-1},R^n_{j}\} - \min\{S^n_{j},R^n_{j+1}\}\Big) = 
\disp\frac{\partial}{\partial \rho^n_{j}}\left(S^n_{j-1} - R^n_{j+1}\right)
= 0.
$$
Hence, \eqref{eq:apx_lj} follows.

\item If $S^n_{j-1} > R^n_j$ and $S^n_{j} > R^n_{j+1}$, then
$$
\disp\frac{\partial}{\partial \rho^n_{j}}\Big(\min\{S^n_{j-1},R^n_{j}\} - \min\{S^n_{j},R^n_{j+1}\}\Big) = 
\disp\frac{\partial}{\partial \rho^n_{j}}\left(R^n_{j} - R^n_{j+1}\right)
= \disp\frac{\partial R^n_{j}}{\partial \rho^n_{j}} \leq 0.
$$
Therefore, we assume 
\begin{equation*}
1-\disp\frac{\dt}{\dx}\Big|\disp\frac{\partial}{\partial \rho^n_{j}}R^n_j\Big|\geq 0,
\end{equation*}
where $R^n_j = \ff(x_j,t^n,\rho^n_j)$ or $R^n_j=\ff^{\max}(x_j,t^n)$. 
This leads again to
\begin{equation*}
1-\disp\frac{\dt}{\dx}\Big|\disp\frac{\partial}{\partial \rho}\ff(x_j,t^n,\rho)\Big|\geq 0.
\end{equation*}

\item If $S^n_{j-1} > R^n_j$ and $S^n_{j} \leq R^n_{j+1}$, then
$$
\disp\frac{\partial}{\partial \rho^n_{j}}\Big(\min\{S^n_{j-1},R^n_{j}\} - \min\{S^n_{j},R^n_{j+1}\}\Big) = 
\disp\frac{\partial}{\partial \rho^n_{j}}\left(R^n_{j} - S^n_{j}\right).
$$
Moreover,
\begin{align*}
	\frac{\partial}{\partial \rho^n_{j}}\left(R^n_{j} - S^n_{j}\right)
&=
	\begin{dcases}
		\disp\frac{\partial}{\partial \rho^n_{j}}\Big(\ff^{\max}(x_j,t^n)-\ff(x_j,t^n,\rho^n_j)\Big) &\quad\text{if $\rho^n_j \leq \sigma^n_j$}\\
		\disp-\frac{\partial}{\partial \rho^n_{j}}\Big(\ff^{\max}(x_j,t^n)-\ff(x_j,t^n,\rho^n_j)\Big) &\quad\text{if $\rho^n_j > \sigma^n_j$}
	\end{dcases}
	\\
	& = \begin{dcases}
		-\disp\frac{\partial \ff(x_j,t^n,\rho^n_j)}{\partial \rho^n_j} < 0  &\quad\text{if $\rho^n_j \leq \sigma^n_j$}\\
		\disp\frac{\partial \ff(x_j,t^n,\rho^n_j)}{\partial \rho^n_j} < 0 &\quad\text{if $\rho^n_j > \sigma^n_j$.}
	\end{dcases}
\end{align*}
Hence, we need again
\begin{equation*}
1-\disp\frac{\dt}{\dx}\Big|\disp\frac{\partial \ff(x_j,t^n,\rho)}{\partial \rho}\Big|\geq 0.
\end{equation*}
\end{enumerate}
Summing up, if  
$$
1-\disp\frac{\dt}{\dx}\Big|\frac{\partial \ff(x,t,\rho)}{\partial\rho}\Big|\geq 0 \quad \mbox{ for all } (x,t,\rho)\in\R\times\R^+\times[0,\rhomax],
$$
then the operator $H(j,n,\boldsymbol\rho)$ is non-decreasing with respect to all the components of $\boldsymbol\rho$. Thus, 
\begin{align*}
\mbox{if } 0\leq \rho^n_j \quad \forall j\in\Z \quad &\Rightarrow \quad 0=H(j,n,\boldsymbol 0) \leq H(j,n,\boldsymbol\rho^n) = \rho^{n+1}_j \quad \forall \ j\in\Z, \\
\mbox{if } \rho^n_j\leq \rhomax \quad \forall j\in\Z \quad &\Rightarrow \quad  \rho^{n+1}_j = H(j,n,\boldsymbol \rho^n) \leq H(j,n,\boldsymbol\rhomax) = \rhomax \quad \forall \ j\in\Z, 
\end{align*}
from which the thesis follows.
\end{proof}

\begin{proof}[Proof of Proposition \ref{prop:CFL1}]
The flux function $\ff(x,t,\rho)=\rho\uu(x,t,\rho)$, with $\uu$ as in \eqref{eq:utrue}, satisfies
$\partial_\rho\ff = \uu+\rho\partial_\rho\uu$.
Since $\partial_\rho\uu \leq 0$ by construction, we have
$$
\disp\frac{\partial\ff(x,t,\rho)}{\partial\rho} \leq \max_{(x,t,\rho)}\uu(x,t,\rho) \quad \mbox{ for all } \quad (x,t,\rho)\in\R\times\R^+\times[0,\rhomax]. 
$$
Let us consider the single trajectory $p_\kappa(t)$, with $\kappa=\kappa(x,t)$ defined in \eqref{eq:kvicino}, then 
\begin{align*}
\uu(x,t,\rho) &= 
	\disp\chi(x-p_{\kappa}(t))\frac{2 \dot{p}_{\kappa}(t)u(\rho)}{\dot{p}_{\kappa}(t)+u(\rho)}+\big(1-\chi(x-p_{\kappa}(t))\big)u(\rho) \\
&\leq \max\{\tilde u(x,t,\rho),u(\rho)\},
\end{align*}
where 
\begin{equation*}
\tilde u(x,t,\rho) = \frac{2\dot{p}_\kappa(t) u(\rho)}{\dot{p}_\kappa(t)+u(\rho)} \quad \mbox{ and } \quad u(\rho) \leq \umax=u(0). 
\end{equation*}
By simple computations we have
$$
\tilde u(x,t,\rho) \leq \frac{2\dot{p}_{\kappa}(t)u(\rho)}{\dot{p}_\kappa(t)+u(\rho)} \leq \frac{\dot{p}_{\kappa}(t)+u(\rho)}{2} \leq \max\big\{\sup_{t}\dot{p}_{\kappa}(t), \umax\big\},
$$
which depends on the selected trajectory $p_\kappa(t)$ at each point $(x,t)$.
Therefore, for all $(x,t,\rho)\in\R\times\R^+\times[0,\rhomax]$
\begin{align*}
\frac{\partial\ff(x,t,\rho)}{\partial\rho} \leq \max\big\{\sup_{(x,t)}\dot{p}_{\kappa(x,t)}(t), \umax\big\}
\end{align*}
and thus we can assume the CFL condition
\begin{equation*}
\disp \frac{\dt}{\dx} \leq \disp\frac{1}{\max\big\{\sup_{(x,t)}\dot{p}_{\kappa(x,t)}(t), \umax\big\}}.
\end{equation*}
\end{proof}

\section*{Acknowledgments}
This work was partially funded by the company Autovie Venete S.p.A.
The authors acknowledge the Italian Ministry of University and Research (MUR) to support this research with funds coming from the project ``Innovative numerical methods for evolutionary partial differential equations and applications'' (PRIN Project 2017, No.\ 2017KKJP4X). The work was also carried out within the research project ``SMARTOUR: Intelligent Platform for Tourism'' (No.\ SCN\_00166) funded by the MUR with the Regional Development Fund of European Union (PON Research and Competitiveness 2007-2013).
The authors are members of the INdAM Research group GNCS.

\bibliographystyle{siam}
\bibliography{references_complete.bib}

\end{document}